\documentclass[10pt]{article}
\usepackage[utf8]{inputenc}
\usepackage{amsmath}
\usepackage{amsfonts}
\usepackage{amssymb}
\usepackage{enumerate}
\usepackage{mathtools}
\usepackage{soul}
\usepackage{enumitem}
\usepackage{adjustbox}
\usepackage{indentfirst}
\usepackage{xcolor}
\usepackage{esint}
\usepackage{authblk}
\usepackage{eucal}
\usepackage{mathrsfs}

\usepackage{etoolbox}
\apptocmd{\thebibliography}
  {\setlength{\itemsep}{0pt}
   \setlength{\parskip}{0pt}}
  {}{}

\usepackage[normalem]{ulem}

\newcommand{\cS}{\mathcal{S}}

\renewcommand{\bar}[1]{\overline{#1}}
\renewcommand{\tilde}{\widetilde}
\renewcommand{\hat}{\widehat}
\newcommand{\pt}{\partial}

\DeclareMathOperator{\vol}{vol}
\DeclareMathOperator{\Vol}{Vol}

\newcommand{\rv}{\mathrm{v}}
\newcommand{\Ric}{\mathrm{Ric}}

\DeclareMathOperator{\Hess}{Hess}
\renewcommand{\div}{\operatorname{div}}

\newcommand{\dd}[1]{\ensuremath{\operatorname{d}\!{#1}}}
\newcommand{\abs}[1]{\lvert #1 \rvert}
\newcommand{\norm}[1]{\lVert #1 \rVert}

\newcommand{\Lip}{\operatorname{Lip}}
\newcommand{\loc}{\operatorname{loc}}
\newcommand{\bLap}{\boldsymbol{\Delta}}
\newcommand{\diam}{\operatorname{diam}}

\mathchardef\mhyphen="2D

\usepackage{amsthm}
\usepackage{aliascnt}
\usepackage{hyperref}
\usepackage{cleveref}

\theoremstyle{plain}
\newtheorem{theorem}{Theorem}[section]

\theoremstyle{definition}

 \newaliascnt{lemma}{theorem} 
 \newaliascnt{proposition}{theorem} 
 \newaliascnt{corollary}{theorem} 
 \newaliascnt{remark}{theorem} 
 \newaliascnt{definition}{theorem}

 \newtheorem{definition}[definition]{Definition}
 \newtheorem{proposition}[proposition]{Proposition}
 \newtheorem{corollary}[corollary]{Corollary}
 \newtheorem{lemma}[lemma]{Lemma}

 \newtheorem{remark}[remark]{Remark}
 
\aliascntresetthe{corollary}
\aliascntresetthe{proposition}
\aliascntresetthe{lemma}
\aliascntresetthe{remark}
\aliascntresetthe{definition}

\crefname{subsection}{subsection}{subsections}
\Crefname{subsection}{Subsection}{Subsections}

\newtheorem{example}{Example}[section]

 \usepackage{titlesec}
 \titleformat{\section}{\large\bfseries}{\thesection.}{0.5em}{}

\titleformat{\subsection}
  {\normalfont\normalsize\bfseries}  
  {\thesubsection.}                   
  {0.5em}                             
  {}                                  

    \usepackage[
      paper=a4paper,
      top=1.2in,
      bottom=1.2in,
      left=1.2in,
      right=1.2in
    ]{geometry}

\usepackage{setspace}
\title{\large\bfseries
QUANTITATIVE VOLUME ESTIMATE ON SINGULAR SETS OF
ONE-PHASE FREE BOUNDARY PROBLEMS\\ IN RICCI LIMIT SPACES}

\author[1]{Alejandro Bellati}
\author[2]{Kai-Hsiang Wang}
\affil[1]{Department of Mathematics, Princeton University\\ abellati@math.princeton.edu}
\affil[2]{Faculty of Mathematics, Technion\\
khwang2025@campus.technion.ac.il}
\date{}

\allowdisplaybreaks[4]

\begin{document}

\maketitle
\vspace{-1cm}
\begin{abstract}
The study of free boundary problems in Euclidean spaces has a long history. In \cite{CZZ22}, the study was extended to one-phase problems on noncollapsed $\mathsf{RCD}(K,n)$ spaces, where singularities
may arise from both the ambient geometry and the free boundary itself. Regarding the singular set of the free boundary in this setting, they proved that it has a sharp Hausdorff dimension bound $n-3$. However, in the absence of a general monotonicity formula, finer properties of the singular set remain inaccessible. 
\par
In this work, we study \textit{limit solutions} of the one-phase problem on noncollapsed Ricci limit spaces.
We introduce a quantitative stratification $\mathcal{S}^k_{\varepsilon,r}$ of the singular set that simultaneously captures the approximate splitting of the ambient space and the approximate homogeneity and translational invariance of the solution. We also introduce a new truncated parabolic Weiss quantity that provides quantitative homogeneity via pinching across scales. 
Although this quantity may not be uniformly bounded across all small scales, it effectively shows that, among the first $j$ scales in a fixed decreasing geometric progression, at most $C\sqrt{j}$ scales fail to exhibit quantitative homogeneity. Such slow growth in $j$ enables us to adapt the covering argument from \cite{CN13} to show, for every $\eta>0$, a $Cr^{n-k-\eta}$ upper bound for the volume of the $r$-tubular neighborhood of the $k$-th quantitative singular stratum, $\mathcal{S}^k_{\varepsilon,r}$. 
\par
The same tubular-neighborhood estimate holds for the effective stratum $\mathcal S_\varepsilon^k$, and in particular we get the Minkowski dimension bound $\dim_{\mathcal M} \mathcal S_\varepsilon^k \le k$, from which the previously known sharp bound $\dim_{\mathcal H} \mathcal S \le n-3$ follows.
As another corollary, we prove that the free boundary admits a H\"older manifold structure away from a set of Minkowski dimension at most $n-3$. 
\end{abstract}

\tableofcontents

\setlength{\parskip}{0.5em}

\section{Introduction}

Let $(X,d,\mu)$ be a suitable metric measure space and $\Omega\subset X$ a bounded open domain.
The one-phase Bernoulli free boundary problem, for short \textit{the one-phase problem}, is concerned with minimizers of the Alt--Caffarelli functional:
\begin{equation}\label{1-phase_Ber}
    E(v)=E_Q(v):= \int_\Omega |\nabla v|^2 + Q^2 1_{v>0} \dd \mu,
\end{equation}
among $v\in W^{1,2}(\Omega)$ with $v-v_0\in W^{1,2}_0(\Omega)$, for some given nonnegative $v_0\in W^{1,2}(\Omega)$ and strictly positive $Q \in C^{\alpha}(\Omega)$.
In the case when $\Omega$ is not bounded, we require that $u\in W^{1,2}_{\text{loc}}(\Omega)$ minimizes $E$ over any bounded subdomain of $\Omega$ among competitors with the same boundary data. In particular, when $\Omega=X$, we say that $u$ is a global minimizer.
By considering the positive part, we may assume that a minimizer $u$ is nonnegative, and hence the name \emph{one-phase}.
Note that the boundary of the positive part, $\pt \{u>0\}\cap \Omega$, is undetermined, and hence the name \textit{free boundary} in this context.
Formally, $u$ solves the following free boundary problem:
\begin{equation}\label{eq:PDE}
    \begin{cases}
        \Delta u = 0& \text{on $\{u>0\}\cap \Omega$},\\
        |\nabla u|=Q& \text{on $\pt \{u>0\}\cap \Omega$},\\
        u=v_0 & \text{on $\pt \Omega$}.
    \end{cases}
\end{equation}
This free boundary problem is an example of a broader class of free boundary problems that arise in a wide variety of fields, including physics, economics, and biology.
The fundamental questions about free boundary problems concern the regularity of minimizers and the associated free boundary. Along these lines, the free boundary is partitioned into a regular set and a singular set, and hence the question becomes how smooth the regular set is, and how small the singular set is.
\par
When $X$ is the Euclidean space (with the canonical metric measure structure), this free boundary problem has been intensively studied in the literature. The pioneer work \cite{AC81} established the existence and Lipschitz continuity of minimizers $u$, and showed that $u$ essentially solves the Dirichlet problem \eqref{eq:PDE}.
For the free boundary, they proved that it is a $C^{1,\alpha}$-manifold outside a relatively closed set $\mathcal{S}_u$ with zero Hausdorff $(n-1)$-measure. This $\mathcal{S}_u$ is the singular set in this context.
\par
A more systematic study of $\mathcal{S}_u$ in the Euclidean setting begins with \cite{Wei99}, who introduced the Weiss monotonicity formula (see Theorem \ref{thm:orginal_weiss}) and confirmed the conjecture in \cite{AC81} that blow-ups of minimizers are one-homogeneous. Together with the dimension reduction argument, the Hausdorff dimension bound on $\mathcal{S}_u$ was improved to $\dim_\mathcal{H} (\cS_u)\leq n - k^*$, with $k^*$ being the smallest dimension such that $\mathbb{R}^{k^*}$ admits a non-linear, one-homogeneous (global) minimizer with constant $Q$. Later on with the works \cite{caffarelli2004global}, \cite{jerison2015some}, \cite{de2009singular}, we know $k^* \in \{5,6,7\}$, although the exact value is still unknown. 
\par
With the development of the study of singular sets in various geometric and analytic contexts, it has been understood that for a finer understanding of singular sets, it is useful to group points according to a quantitative notion of how symmetric solutions are at given scales, instead of measuring the amount of exact symmetries of their blow-ups.
This gives rise to the so-called \emph{quantitative stratification} $\mathcal{S}^k_{\eta,r}$, which first appeared in the context of Ricci limits \cite{CN13} and was subsequently used in many other problems.
Using these ideas, \cite{EE18} defined such stratification for free boundary problems in the Euclidean setting. Combined with the rectifiable Reifenberg-type theorems from \cite{NV17} and a corona-type covering argument, they improved the previous Hausdorff dimension bounds to the stronger packing number estimates on the singular strata, and moreover proved their rectifiability. For more background on the Euclidean theory of the one-phase problem, we refer to the monograph \cite{caffarelli_geometric_2005}, \cite{Vel23}, and the introduction in \cite{EE18}.
\par
In \cite{CZZ22}, the study of this free boundary problem is extended to nonsmooth spaces. The ambient spaces are now the so-called $\mathsf{RCD}(K,N)$ spaces, which are metric measure spaces that have Riemannian Ricci curvature bounded from below by $K$ and dimension bounded from above by $N$ in a synthetic sense. They also assume that the space is noncollapsed in volume, which amounts to saying that the measure is the $n$-dimensional Hausdorff measure, with $N=n\in\mathbb{N}$ (namely, the $\mathsf{ncRCD}(K,n)$ space as defined in \cite{DPG2018}).
Among their results, they recovered most of the regularity results in \cite{AC81} for minimizers in this nonsmooth setting. As for the free boundary, notice that the singularity can now come from either the free boundary structure itself or from the background space.
In this situation of possible mixed singularities, they proved that, away from a subset of Hausdorff dimension at most $n-3$, the free boundary is a H\"older manifold, and this dimension bound is sharp.
However, finer properties of the singular set remain inaccessible without a general monotone quantity.  In fact, even the homogeneity of blow-ups of the minimizers is currently unknown.
\par
Classical usage of monotone quantities to control singular sets via symmetry relied on the following principle: Monotonicity, together with boundedness, allows one to pass the quantity to the limit as the scale goes to $0$.
This limit value is then attained by the corresponding quantity of a blow-up at every scale, and rigidity of the monotone quantity can now be invoked to show cone-like behavior ($0$-symmetry) of the blow-up.
In the singular setting as in \cite{CZZ22}, however, such bounded monotone quantities are known to exist only on $\mathsf{ncRCD}$ cones (see \cite[Lemma A.3]{CZZ22}). 
Consequently, it is not clear a priori whether blow-ups exhibit any cone-like behavior.
In parallel to this discussion, \cite{FMS25} studied perimeter-minimizing sets in $\mathsf{RCD}$ spaces. 
When the ambient space is an $\mathsf{ncRCD}$ cone, they proved a monotonicity formula for the perimeter density ratio, together with the corresponding rigidity statement.
This led to parallel regularity results for the boundaries of perimeter-minimizing sets. However, it remains unknown whether every blow-up at a boundary point is a perimeter-minimizing cone.

\par
This article aims to show finer properties of the singular set in the nonsmooth setting, where no bounded monotone quantity is currently available.
Based on the preceding developments, it is natural to ask whether the quantitative stratification techniques can be adapted to this situation.
We follow the strategy first developed in \cite{CN13}, and later used in \cite{CheegerNaberValtorta2015CriticalSets}, \cite{CheegerNaber2013QuantitativeStratification}, \cite{ABS19}, among others.
The work \cite{CN13} marked the first step toward making full use of the monotone quantity itself, rather than only using the existence of its limiting value. 
More precisely, they exploited the fact that, at each point $x$, the quantity is $\delta$-pinched at all but (universally) finitely many scales, and hence provides quantitative cone-like behavior at such abundant scales via almost rigidity. 
\par
As it turns out, \cite{CN13} provides a way to circumvent the boundedness
requirement on the monotone quantity.
A key observation is that, in order to obtain a $Cr^{n-k-\eta}$ upper bound on the volume of the $r$-tubular neighborhood of the $k$-th quantitative stratum, the covering argument does not require the monotone quantity to remain finite down to arbitrarily small scales. In our setting, for every $0<t_0\ll 1$, we introduce a truncated parabolic Weiss quantity $W_{t_0,t}$, which is increasing for $t\in[t_0,1]$ and has an effective bound that explicitly depends on $t_0$ (see \eqref{eq:Wbound} below). This $W_{t_0,t}$ is designed to provide almost rigidity for scales between $t_0$ and $1$, and the rate at which $W_{t_0,t}$ diverges, as $t_0\to 0$, is sufficiently slow to adapt the covering argument from \cite{CN13}.
Effectively, at each point, we have the following control
\begin{center}
    $\text{\# \{bad scales between $t_0$ and $1$}\}\lesssim \sqrt{\text{\# \{scales between $t_0$ and $1$\}}}$.
\end{center}
At the remaining good scales, the quantity is well-pinched and hence gives almost homogeneity of the minimizer $u$.

Our main result (see \Cref{thm:main}) improves the Hausdorff dimension bound of the singular set in the nonsmooth setting by establishing Minkowski content estimates on the effective singular strata $\mathcal{S}^k_\varepsilon$.
In particular, this implies sharp Minkowski dimension bounds for $\mathcal{S}^k_\varepsilon$. Nonetheless, our results are obtained under the following restrictions\label{intro:lim_sol}: 
\begin{enumerate}\label{restrictions}
\item We only work on noncollapsed Ricci limit spaces, which are contained in the $\mathsf{ncRCD}$ framework used in \cite{CZZ22}. The reason for this restriction is that we use a domain variation formula to prove the truncated parabolic Weiss formula (\Cref{thm:parabolic_weiss}), and such variations involve the flow of vector fields on smooth manifolds. Therefore, we only work with their noncollapsed limits. See \Cref{sec:ParabolicWeissIntro} for more explanation.

\item  For the same reason as above, we only consider \textit{limit solutions}, which are limits of minimizers to the one-phase problem \eqref{1-phase_Ber} on smooth manifolds that pGH-converge to the limit space. In particular, the approximating sequence of minimizers satisfies a uniform Lipschitz bound.
See \Cref{subsec:lim_sol} for the rigorous definition.   
\end{enumerate}

In the remainder of this introduction, we provide further details on the ingredients of our argument and state our main results. In \Cref{sec:QuantitiveStrataIntro}, we define the quantitative stratification with respect to both the space and the minimizer. In \Cref{sec:ParabolicWeissIntro}, we provide more details about the truncated parabolic Weiss formula and show the explicit divergence bound. In \Cref{sec:mainresultsintro}, we present our main results, along with some other consequences.

\subsection{Quantitative Singular Stratification of Two Aspects}\label{sec:QuantitiveStrataIntro}

Let $u$ be a minimizer to the one-phase problem \eqref{1-phase_Ber} in a noncollapsed Ricci limit space, or more generally a continuous nonnegative function on an $\mathsf{ncRCD}(K,n)$ space.
Following \cite{CN13}, our quantitative stratification is based on whether there is a definite amount of symmetry at a fixed scale.
To deal with the possible singularities from either the ambient space or the free boundary structure, we stratify $\pt\{u>0\}\cap \Omega$ in both of these aspects.
The resulting quantitative strata $\mathcal{S}^k_{\eta,r}$ can be viewed as a combination of those in \cite{CN13,ABS19} for studying Ricci limit spaces and $\mathsf{RCD}$ spaces, and those in \cite{EE18} for studying free boundary problems on the Euclidean space. To be precise, we first define the quantitative symmetry in our context.

\begin{definition}[$(k,\eta)$-symmetric ball]\label{def:quant_sym}
    Let $(X,d,\mathcal{H}^n)$ be an $\mathsf{ncRCD}(K,n)$ space, $u:\Omega \subset X \to [0,\infty)$ a continuous function, and $x\in \Omega$ such that $u(x)=0$.
    Fix $k\in \mathbb{N}_0\cap[0,n]$, and $\eta>0$.
    We say that $B_r(x)\subset \Omega$ is $(k,\eta)$-symmetric if the following hold:
    \begin{enumerate}
        \item There is an $\mathsf{ncRCD}(0,n)$ cone $(\mathbb{R}^k\times C(Z),d,\mathcal{H}^n, (0,z_*))$  and a Borel $\eta r$-pGH approximation $F:B_r(0,z_*)\to B_r(x)$, where $B_r(0,z_*)$ is the $r$-ball in $\mathbb{R}^k\times C(Z)$ centered at the vertex point $(0,z_*)$. Here $C(Z)$ denotes the metric cone over some metric space $Z$.
        \item There exist a $1$-homogeneous and $\mathbb{R}^k$-invariant function $\tilde{u}$ on $\mathbb{R}^k\times C(Z)$ such that
        \begin{equation}\label{eq:L2etaclosenesstohomogeneity}
            r^{-2-n}\int_{B_r(0,z_*)} |u\circ F-\tilde{u}|^2 \dd{}\mathcal{H}^n< \eta. 
        \end{equation}
        Here by $\mathbb{R}^k$-invariant, we mean $\tilde{u}(p+v,q) = \tilde{u}(p,q)$ for every $p,v \in \mathbb{R}^k$ and $q \in C(Z)$.
        By $1$-homogeneous, we mean $\lambda u(p,q)= u(\lambda p,\lambda q)$ for every $p\in \mathbb{R}^k$, $q \in C(Z)$, and $\lambda\in [0,\infty)$, where $(\lambda p, \lambda q)$ represents the radial dilation by the factor $\lambda$.
    \end{enumerate}
    We call $(B_r(0,z_*), \tilde{u})$ a \emph{symmetry model} for the $(k,\eta)$-symmetric ball $B_r(x)$.
\end{definition}
We also define balls with exact symmetry.
\begin{definition}[$k$-symmetric ball]
    Under the same assumptions above, we say that $B_r(x)$ is $k$-symmetric if 
    \begin{enumerate}
        \item  There is an $\mathsf{ncRCD}(0,n)$ cone $(\mathbb{R}^k\times C(Z),d,\mathcal{H}^n, (0,z_*))$ and an isometry $F:B_r(0,z_*)\to B_r(x)$.
        \item $u\circ F$ is $1$-homogeneous and $\mathbb{R}^k$-invariant.
    \end{enumerate}
\end{definition}

In the definitions above, the first items describe the local symmetry of the ambient spaces since, by \cite{CC96,DPG2018}, any tangent space of a noncollapsed Ricci limit space or an $\mathsf{ncRCD}$ space is an $\mathsf{ncRCD}(0,n)$ cone, whose symmetry is measured by the number of Euclidean splitting factors. 
The second items describe the (almost) symmetry of $u$, in terms of the radial homogeneity and the invariance along those Euclidean splitting factors.
We remark that the almost symmetry of a minimizer $u$ translates into the almost symmetry of $\pt\{u>0\}$; see \Cref{prop:almostsymualmostfree}. More remarks follow below.

\begin{remark}
    From \cite{CZZ22} (see also \Cref{thm:mainCZZ22}), minimizers of the one-phase problem are locally Lipschitz. This is the standing regularity condition upon which we build our argument throughout this paper. Notice that we do not require the symmetry models $\tilde{u} $ to be minimizers.
\end{remark}

\begin{remark}[$L^2$-distance]
We measure the symmetry of the function in \Cref{def:quant_sym} using the $L^2$-distance.
In our applications, the relevant functions satisfy uniform local Lipschitz bounds.
Hence, the arguments could alternatively be formulated
using uniform closeness, provided we modify the quantitative constants accordingly.
We use the $L^2$-formulation to remain consistent with other applications of quantitative stratification.
\end{remark}

\begin{remark}[Compatibility of measures]
{Instead of requiring the cone models to be noncollapsed, one could impose an additional quantitative compatibility between the measure
of the ball in $X$ and that in the cone model $C$, for instance $\left|\log\frac{m_C(B_r(x))}{m_X(B_r(x))}\right|<\eta$.
Such a condition would provide the measured compactness of the cone models needed in our arguments. From this point of view, the choice made in \Cref{def:quant_sym} is mostly a matter of convenience: ultimately, the relevant class of cone models is determined by the almost-rigidity \Cref{lem:alm_rigid}. We therefore work directly with $\mathsf{ncRCD}(0,n)$ cone models.} 
\end{remark}
\begin{remark}[Cone models and the metric-measure perspective]
Within the noncollapsed category, every metric isometry between cone models automatically preserves the reference measure, since the Hausdorff measure is invariant under isometries. Hence, every isometry as such is automatically an isomorphism of metric measure spaces.
Moreover, using the results of \cite{DPG2018,Ketterer2015}, the reference
measure on $\mathbb{R}^k\times C(Z)$ decomposes according to the product and cone structures as
$\mathcal{H}^n=\mathscr{L}^k\otimes\mathcal{H}^{n-k}_{C(Z)}
=\mathscr{L}^k\otimes r^{n-k-1}\,dr\otimes\mathcal{H}^{n-k-1}_Z$;
see \Cref{lem:ncRCD_cone_struct}.
\end{remark}

\begin{remark}\label{rmk:rescaling}
By rescaling $(X,d,\mathcal{H}^n,u) \mapsto (X,r^{-1}d,\mathcal{H}^n_{r^{-1}d} = r^{-n}\mathcal{H}^n, r^{-1}u)$, the conditions in \Cref{def:quant_sym} can be expressed in the following manner:
$$
\left\{\begin{aligned}
& F:B_1^{r^{-1}d}(0,z_*)\subset \mathbb R^k \times C(Z) \to B_1^{r^{-1}d}(x) \subset X \ \ \ \text{is a $\eta$-pGH approximation},  \\
& \int_{B_1(0,z_*)} \abs{r^{-1} u\circ F - r^{-1}\tilde{u}}^2 \dd{}  \mathcal H^n < \eta.\\
\end{aligned}\right.
$$
Hence $(B^d_r(x),u)$ is $(k,\eta)$-symmetric if and only if $(B^{r^{-1}d}_1(x),r^{-1}u)$ is $(k,\eta)$-symmetric, and the symmetric function models are related by $\tilde{u} \mapsto r^{-1}\tilde{u}$.

\end{remark}
Given the definition of quantitative symmetry, we proceed to define the quantitative strata.
\begin{definition}[Quantitative singular stratum]\label{def:quant_sing_stra}
Under the same assumptions and notation as in \Cref{def:quant_sym}, for $r\in(0,1)$ we define 
\begin{equation}
\mathcal{S}^k_{\eta,r}(u):=\{x\in \pt \{u>0\}\cap \Omega \mid \text{for all $s\in [r,\bar{d}(x)]$, $B_s(x)$ is not $(k+1,\eta)$-symmetric}\},
\end{equation}
where $\bar{d}(x) = \text{min}\{1,d(x,\partial \Omega)/2\}$.
\end{definition}
Clearly, we have $\mathcal S^k_{\eta,r}(u)\subset
\mathcal S^{k'}_{\eta',r'}(u)$ whenever
$k'\geq k$, $r'\geq r$, and $\eta'\leq\eta$.
This motivates the following definitions:
\begin{equation}    \mathcal{S}^k_\eta(u):=\bigcap_{r>0}\mathcal{S}^k_{\eta,r}(u),\quad \mathcal{S}^k(u):= \bigcup_{\eta>0} \mathcal{S}^k_\eta(u).
\end{equation}
\par
When there is no room for confusion, we will drop the dependence on \(u\) in the notation. In \Cref{sec:QuantitiveStrataTwoTypes} we present some basic properties and relevant results for these strata. In particular, \Cref{prop:basicpropstrata} shows that $\mathcal{S}^k(u)$ coincides with the classical singular strata, which are characterized by the amount of symmetry of blow-ups.

\subsection{Truncated Parabolic Weiss Quantity}\label{sec:ParabolicWeissIntro}

For the one-phase problem in Euclidean space, the original Weiss quantity introduced in \cite{Wei99} (see also Theorem \ref{thm:orginal_weiss}) is bounded and (almost-)monotone in the radial variable.
It also has the rigidity that, for a constant coefficient $Q$, it is constant if and only if the solution is $1$-homogeneous. 
In particular, this implies that blow-ups of solutions are $1$-homogeneous, which is the starting point for obtaining additional symmetries (invariant directions) through the cone-splitting principle.
A version of the Weiss formula was further developed for $\mathsf{ncRCD}$ spaces in \cite[Eq.\@ 8.1]{CZZ22}.
They showed that the corresponding quantity is monotone (for constant $Q$) when the ambient space is an $\mathsf{ncRCD}$ \emph{cone}.
Without the cone structure, however, no monotonicity formula for the one-phase problem is known on general $\mathsf{ncRCD}$ spaces.
\par
There are, at least, two distinct difficulties. The first concerns the derivation of such formulas on a nonsmooth space. Although a substantial second-order calculus is available in the $\mathsf{RCD}$ setting, carrying out the necessary domain-variation argument is already a delicate problem; see \cite[Remark 1.7]{CZZ22} and partial progress in \cite{bruè2018regularitylagrangianflowsrcdkn}. Another example is the extension of Perelman's entropy monotonicity to $\mathsf{RCD}$ spaces carried out in \cite{brena2025perelman}, where the use of the nonsmooth second-order calculus introduces additional technical difficulties.
In this work, we use smooth manifolds approximating the Ricci limit space to avoid addressing this issue directly, though it remains an interesting question for future research.
\par
The second difficulty is geometric and is already present on smooth manifolds. In the Euclidean space, the derivation underlying the Weiss formula relies on the identity $\Hess\left(\frac{1}{2}d_x^2\right)=g_{\text{Euc}}$.
On a general manifold, the discrepancy between $\Hess(\frac{1}{2}d_x^2)$ and the metric tensor $g$ produces an additional term with no definite sign.
Moreover, this term need not remain uniformly bounded, in a scale-wise $L^1$ sense, along a sequence of manifolds $M_i$ converging to a Ricci limit space.
To handle this difficulty, we introduce a lower cutoff $t_0$ of the scales in the corresponding correction term, resulting in a truncated quantity $W_{t_0,t}$ that has an effective bound in terms of $t_0$.
\par
Another modification we make to solve the second difficulty above is that our monotone quantity uses the heat kernel $\rho_{t,x}$ of the ambient space. 
More precisely, we replace the radial vector field $\nabla\frac{1}{2}d_x^2$ with $2t\nabla f_{t,x}$, where $f_{t,x}:=-\log\bigl((4\pi t)^{n/2}\rho_{t,x}\bigr)$.
These two vector fields agree on the Euclidean space and, more generally, on $\mathsf{ncRCD}(0,n)$ cones when $x$ is the vertex.
This heat-kernel formulation avoids the regularity issues associated with the distance function. 
Most importantly, the heat kernel satisfies an $L^2$-Hessian estimates based on the local entropy \cite[Theorem 4.22]{CJN21}, and converges strongly in $W^{1,2}$ along converging sequences of manifolds (see \Cref{thm:heatkernelconvergence} and the references cited above it).
These properties allow us to control the geometric correction term and to pass to the limit along the approximating sequence.
\par
Our monotonicity formula is the following: 
\begin{theorem}[Truncated Parabolic Weiss Quantity]\label{thm:parabolic_weiss}
Let $u$ be a minimizer of the one-phase problem \eqref{1-phase_Ber} with $X$ being a smooth manifold $(M,g)$. 
Fix some ball $B_{4\bar{r}}(\bar{x})\subset\Omega$. Assume $\bar{r}\leq \delta^{-1}$, $\Ric \geq -(n-1)\delta^2$, and $\Vol(B_{\bar r}(\bar x))\ge \mathrm{v}\,\bar r^n$, for some $0<\delta\leq 1$ and $\mathrm{v}>0$.
Assume also that $\abs{Q+1/Q}_{C^0,B_{4\bar r}(\bar x)} + \overline{r}^{\alpha}[Q]_{\alpha, B_{4\bar r}(\bar x)}<\Lambda$, and that $u$ is $L$-Lipschitz on $B_{4\bar r}(\bar x)$.
Pick some $0<t_0 < \frac{1}{2}\bar{r}^2$. Let $\phi:M\to [0,1]$ be a cutoff function with support in $B_{4\bar{r}}(\bar{x})$, satisfying
\begin{equation}\label{cutoff}
    \phi|_{B_{2\bar{r}}(\bar{x})}\equiv 1, \quad |\nabla \phi|^2+|\Delta \phi|\leq C(n)\bar{r}^{-2}.
\end{equation} 
For $x\in B_{\bar{r}}(\bar{x})\cap \pt\{u>0\}$ and $t\in [t_0,\overline{r}^2]$, define
    \begin{align}
        W_{t_0,t}(u,g,Q,x):=&\int_M \left[2(|\nabla u|^2+{Q^2(x)}1_{u>0})-t^{-1}u^2\right] \phi \rho_{t,x}\dd\vol\\
    +&4\int_{t_0}^t \int_M \left|(H_{s,x} +\frac{\rho_{s,x}}{2s}g)(\nabla u, \nabla u)\right| \phi\dd\vol \dd s
    +\mathcal{E}_{t}(u,g,Q,x),
    \end{align}
    where $\mathcal{E}_t(u,g,Q,x)\geq 0$ is an error term defined by
\begin{align}\label{eq:parabolic_weiss_error}
\mathcal{E}_t(u,g,Q,x)
:=&
\int_0^t \int_M
\left[
(6|\nabla u|^2+2{Q^2(x)})|\nabla f_{s,x}||\nabla\phi|
+s^{-1}u^2|\Delta\phi|
\right]
\rho_{s,x}\,d\vol\,ds
\\
&
+{2\int_0^t \int_{\partial\{u>0\}}
|\delta_x Q^2|\,|\nabla f_{s,x}|\,
\rho_{s,x}\phi\,
d\mathcal{H}^{n-1}\,ds},
\end{align}
where $(\delta_{x}Q^2)(y):=Q^2(y)-Q^2(x)$.
Then $ t \mapsto W_{t_0,t}(u,g,x)$ belongs to $C^1([t_0,\overline{r}^2])$, and we have 
\begin{align}
\frac{d}{dt}W_{t_0,t}(u,g,Q,x) & \geq t^{-2}\int_M \left[ u-\nabla u\cdot \nabla(2tf_{t,x})\right]^2 \phi \rho_{t,x}\dd\vol, \label{eq:monder} \\
|W_{t_0,t}(u,g,x)| & \le C(n,\mathrm{v},L,\Lambda,\alpha)
\left(
1+\left|\log\frac{t}{t_0}\right|^{1/2}
+[Q]_{\alpha,B_{4\bar r}(\bar x)}t^{\alpha/2}
\right), \label{eq:Wbound}\\
\mathcal E_t(u,g,x)
&\leq
C(n,\mathrm{v},L,\Lambda)\exp\left(-\frac{\bar r^2}{10t}\right)
+
C(n,\mathrm{v},L,\Lambda,\alpha)[Q]_{\alpha, B_{4\overline{r}}(\overline{x})}
t^{\alpha/2}.
\label{eq:errorb}
\end{align}
\end{theorem}

The usefulness of $W_{t_0, t}$ boils down to two facts. First, from \eqref{eq:monder}, pinching of the quantity between scales implies that $u$ is quantitatively close to a $1$-homogeneous function, and therefore provides effective control of its symmetry; see its precise usage with the volume ratio in the Almost Rigidity \Cref{lem:alm_rigid}.
Secondly, although the bound \eqref{eq:Wbound} diverges as $t_0 \to 0$, it still provides some effective control.
Indeed, it follows that the number of bad scales $r=\gamma^{i}$ with  $i \leq i_0$ is bounded by $C\sqrt{i_0}$. This, in turn, allows us to adapt the covering argument from \cite{CN13} to our setting; see \Cref{lem:decomp}.

\subsection{Main Results}\label{sec:mainresultsintro}

In this article, we improve the Hausdorff dimension estimate for the singular set of the free boundary $\partial\{u>0\}\cap\Omega$ in the nonsmooth setting by establishing Minkowski content estimates for the quantitative singular strata. 
The proof relies on the truncated parabolic Weiss formula (\Cref{thm:parabolic_weiss}), which is derived in the smooth category. 
Therefore, we only consider \emph{limit solutions} as mentioned in \Cref{intro:lim_sol}; see \Cref{subsec:lim_sol} for more details. Our main result is the following.
\begin{theorem}\label{thm:main}
    Let $X$ be an $n$-dimensional noncollapsed Ricci limit space,
    and let $u$ be a limit solution with $C^\alpha$ coefficient $Q$ on a bounded domain $\Omega\subset X$.
     Without loss of generality, we fix $B_4(x)\Subset \Omega$ and assume $\mathcal{H}^n(B_4(x))\geq \mathrm{v}>0$.
    Let $L,\Lambda$ denote, respectively, the Lipschitz and coefficient bounds for the approximating sequence of $u$ as in \Cref{subsec:lim_sol} on $B_4(x)$.
    Then for any $\eta\in(0,1)$ and $\varepsilon \in (0,1)$, there exists $C=C(n,\mathrm{v},L,\Lambda,\alpha,\eta,\varepsilon)>0$ such that for any $r\in(0,1)$, we have
    \begin{equation}
      \Vol(T_r(\mathcal{S}^k_{\varepsilon}) \cap B_1(x)) \leq \Vol(T_r(\mathcal{S}^k_{\varepsilon, r})\cap B_1(x))\leq Cr^{n-k-\eta}, \label{eq:vol1}
    \end{equation}
    where $T_r(\cdot)$ denotes the $r$-tubular neighborhood, and $\mathcal{S}^k_{\varepsilon,r}$ is the $(k,\varepsilon,r)$-singular strata as in \Cref{def:quant_sing_stra}.
    We also have 
\begin{equation}
\mathcal{H}^{n-1} \llcorner \partial \{u > 0\}\left(T_r(\mathcal{S}^k_{\varepsilon, r})\cap B_1(x)\right)\leq Cr^{n-k-1-\eta} \label{eq:vol2},
\end{equation}
and the Minkowski dimension bound
    \begin{equation}
    \dim_{\mathcal{M}}(\mathcal{S}^{k}_\varepsilon\cap B_1(x))\leq k. \label{eq:vol3}
    \end{equation}
\end{theorem}
\par

Combined with the Reifenberg-type results in \Cref{prop:reifenbergprop} and \Cref{cor:Reifenberg}, we get:
\begin{corollary}\label{cor:intro_biholder}
Under the hypothesis above, there is a set $\tilde{\mathcal{S}}$ with Minkowski dimension at most $n-3$ and bounded Minkowski $(n-3 + \varepsilon)$-content, such that $\partial \{u>0\} \cap \Omega \backslash \tilde{\mathcal{S}}$ is a H\"older manifold. 
\end{corollary}
\par
Now assume in addition that the ambient spaces satisfy two-sided Ricci curvature bounds. Combining the above theorem with a result of the second author \cite{wang2026dimensionboundsingularset} and an $\varepsilon$-regularity theorem adapted from the Euclidean setting (\Cref{prop:epsregdoublericcibounds} and \Cref{cor:2sided_Ricci_FBP_mfd_struct}), we obtain
\begin{corollary}\label{cor:doublericciboundreg}
Assume additionally that $X$ is an $n$-dimensional noncollapsed limit space with two-sided Ricci curvature bounds.
Then there is a set $\tilde{\mathcal{S}}$ with Minkowski dimension at most $n-5$ and bounded Minkowski $(n-5 + \varepsilon)$-content, such that $\partial \{u>0\} \cap \Omega\backslash\tilde{\mathcal{S}}$ is a $C^{1,\beta}$ manifold for some $\beta>0$.
\end{corollary}

\subsection{Structure of Article}

In \Cref{sec:prelim_and_limit_sol}, we review the one-phase problem in the Euclidean and in the nonsmooth setting, and collect preliminary results about compactness, convergence, and heat kernels that are used throughout the paper.
In \Cref{sec:QuantitiveStrataTwoTypes}, we establish some basic properties about the quantitative stratification and prove some regularity results of the top stratum.
In \Cref{sec:paraweiss}, we derive the truncated parabolic Weiss formula.
In \Cref{sec:vol_est}, we first prove the quantitative rigidity, then establish the version of the cone-splitting principle needed here, and finally carry out the covering argument that leads to the main volume estimates.

\subsection*{Acknowledgment} 
The first author would like to thank his advisor, Aaron Naber, for his guidance and many helpful discussions. Part of the research in this article was done when the second author was visiting Aaron Naber at IAS during the academic year 2024--2025. 
He would like to thank them for fruitful discussions and warm hospitality.
The authors would also like to thank Dennis Kriventsov for his comments on free boundary problems.
The research leading to these results is part of a project that has received funding from the European Research Council (ERC) under the European Union's Horizon 2020 research and innovation programme (grant agreement No 101001677).

\section{Preliminaries}\label{sec:prelim_and_limit_sol}

In this section, we collect the background material needed throughout the paper. 
In \Cref{subsec:onephase_euc} we recall the one-phase problem in the Euclidean setting.
In \Cref{subsec:onephase_rcd}, we review the formulation and basic regularity theory of the one-phase problem on $\mathsf{ncRCD}$ spaces.
In \Cref{subsec:pmgh_convg_and_cptness}, we introduce pmGH convergence and the compactness and blow-up results that will be used repeatedly in our arguments. In \Cref{subsec:lim_sol}, we define the class of limit solutions considered in this work.
Finally in \Cref{subsec:heat_kernel}, we recall the heat kernel estimates and convergence results needed for the parabolic Weiss quantity.

\subsection{One-phase Problem in Euclidean Setting}\label{subsec:onephase_euc} 
In this section, we review several results of the one-phase problem \eqref{1-phase_Ber} in the Euclidean setting, where $(X,d,\mu)=(\mathbb{R}^n, d_{\text{Euc}}, \mathscr{L}^n)$, with the primary aim of keeping them in mind when thinking about the nonsmooth setting. As commented in the introduction, the fundamental questions about the minimizers were answered in \cite{AC81}.
Regarding the regularity of the free boundary, they established:

\begin{theorem}[{\cite{AC81}}]
Let $u$ be a minimizer of the one-phase problem \eqref{1-phase_Ber}, with strictly positive $Q\in C^{\alpha}(\Omega)$. Then the free boundary $\partial\{u>0\}\cap \Omega$ is a $C^{1,\alpha}$-manifold away from a relative closed subset $\mathcal{S}$ with $\mathcal{H}^{n-1}(\mathcal{S})=0$.
\end{theorem}
Based on this theorem, one defines the regular set of the free boundary as the set of points where the free boundary is locally a $C^{1,\alpha}$-manifold, and the singular set as the complement, which is the subset $\mathcal{S}$ in the theorem. We also mention that a compactness theorem is given in \cite{AC81}, allowing blow-up analysis. The next step towards understanding the singular set was taken in \cite{Wei99}, where a monotone quantity was introduced.
For the sake of comparison with our \Cref{thm:parabolic_weiss}, we quote the original Weiss monotonicity formula:

\begin{theorem}[\cite{Wei99}, Theorem 1.2]\label{thm:orginal_weiss}
Let $u$ be as above. Fix $x\in\partial\{u>0\}\cap\Omega$. For
$0<r<d(x,\partial\Omega)$, define
\begin{equation}
W_{u,Q,x}(r)=W(r):=
r^{-n}\int_{B_r(x)}|\nabla u|^2
+Q^2(x)1_{u>0}\dd{} \mathscr{L}^n
-r^{-n-1}\int_{\partial B_r(x)}u^2\dd\sigma.
\end{equation}
Then, for every
$0<s<r<R<d(x,\partial\Omega)$,
\begin{equation}
\int_{B_r(x)\setminus B_s(x)}
\frac{|u(y)-(y-x)\cdot\nabla u(y)|^2}
{|y-x|^{n+2}}\dd{} \mathscr{L}^n(y)
\leq
W(r)-W(s)
+
C(n,\alpha)r^\alpha[Q]_{\alpha,B_R(x)}.
\end{equation}
In particular, when $Q$ is constant, $W(r)$ is increasing,
and it is constant if and only if $u$ is $1$-homogeneous.
\end{theorem}
The quantity $W$ is clearly bounded since $u$ is locally Lipschitz.
Hence the preceding monotonicity formula implies that $W(r)$ has a
limit as $r\to 0^+$.
Moreover, under blow-up operations at $x$ along some $t_i\to 0$, the coefficient $Q$ freezes to the constant $Q(x)$, and the corresponding blow-up of $u$:
\begin{equation}
    u_{x}(y):=\lim_{t_i\to 0} \frac{1}{t_i}u(x+t_iy)
\end{equation}
has $W_{u_x,Q(x),0}\equiv W_{u,Q,x}(0^+)$.
The rigidity of $W$ then implies that $u_x$ is a $1$-homogeneous global minimizer on $\mathbb{R}^n$ with $Q\equiv Q(x)$. Together with Federer's dimension-reduction argument, \cite{Wei99} proved that the Hausdorff dimension of the singular set $\mathcal{S}$ of $\pt\{u>0\}\cap\Omega$ can be bounded by $n-d^*$, where $d^*$ is the lowest dimension such that there exist singular non-linear $1$-homogeneous global minimizers on $\mathbb{R}^{d^*}$.
Thanks to \cite{caffarelli2004global}, \cite{jerison2015some}, \cite{de2009singular}, we know $d^* \in \{5,6,7\}$, although the exact value is still an open question.
\par

A more refined structure theory for the singular set was achieved in \cite{EE18}.
Inspired by \cite{CN13}, the authors introduced a quantitative stratification of the singular set.
Using a corona-type covering argument and a rectifiable version of a Reifenberg-type theorem from \cite{NV17}, they improved the Hausdorff dimension bound to Hausdorff measure estimates on the singular strata, and additionally established their rectifiability.
Their results actually provided quantitative bounds on both the Minkowski content and the packing numbers of the quantitative singular strata.
\subsection{One-phase Problem in Nonsmooth Setting}\label{subsec:onephase_rcd}
In \cite{CZZ22}, the authors extended the study of free boundary problems to the class of nonsmooth spaces known as $\mathsf{RCD}$ spaces.
Under an additional noncollapsed assumption, \cite{CZZ22} addressed several fundamental questions and established some regularity results for the free boundary. In this section, we recall the formulation of the problem in this nonsmooth setting and summarize the main results from \cite{CZZ22}.
See also \cite{lin2025notebernoullitypefree} for partial generalizations to the collapsed setting.

\subsubsection{RCD Spaces, Sobolev Spaces, Measure-valued Laplacian}
\noindent\textit{$\mathsf{RCD}$ spaces:} 

\noindent Let $(X,d,\mu)$ be a metric measure space; that is, a complete separable (proper) metric space $(X,d)$ equipped with a Radon measure $\mu$ of full support. We say that $X$ is an $\mathsf{RCD}(K,N)$ space if it has Riemannian Ricci curvature bounded below by $K$ and dimension bounded above by $N$, in a suitable synthetic sense.
The theory of $\mathsf{RCD}$ spaces provides an intrinsic framework for studying limits of smooth manifolds and allows the development of second-order calculus.
Classical results in Riemannian geometry, such as Laplacian comparison and volume comparison, admit generalizations on such $\mathsf{RCD}$ spaces.
One key feature is that the class of $\mathsf{RCD}(K,N)$ spaces is closed under pointed measured Gromov--Hausdorff convergence (see \Cref{def:GH_pmGH}). Canonical examples of such spaces include Ricci limit spaces arising in the Cheeger--Colding theory, and finite-dimensional Alexandrov spaces with curvature bounded from below (see \cite{Petrunin2011Alexandrov}). 
We refer the reader to the survey \cite{ambrosio2018calculus} and the references therein for a comprehensive overview of the $\mathsf{RCD}$ theory.
In the following, we recall only those relevant to our purposes, most of which are either contained in or referenced in \cite{CZZ22}.
\par
In this article, we work exclusively with noncollapsed $\mathsf{RCD}(K,N)$ spaces ($\mathsf{ncRCD}(K,N)$ spaces for short). 
Such spaces were introduced by De Philippis and Gigli in \cite{DPG2018} as a synthetic counterpart of noncollapsed limits of Riemannian manifolds with Ricci curvature bounded below.
By definition, the reference measure satisfies $\mu=\mathcal{H}_d^N$, and it was shown (\cite[Theorem~1.12]{DPG2018}) that $N$ is necessarily an integer, which we denote by $n$ from now on.

Within the noncollapsed category, the metric splitting of cones is compatible with the corresponding Hausdorff measure structure.
We record this in the following lemma, which follows from \cite{Ketterer2015} and \cite{DPG2018}, and is explained in detail in \cite[Lemma 4.1]{KapovitchMondino2021}.

\begin{lemma}[Metric-measure structure of cones]\label{lem:ncRCD_cone_struct}
Let $(Y,d,\mathcal H^n)=(\mathbb R^k\times C(Z),d,\mathcal H^n)$ be an $\mathsf{ncRCD}(0,n)$ cone, and assume that $n-k\geq 2$. Then $(Y,d,\mathcal H^n)\cong(\mathbb R^k\times C(Z),d,\mathcal L^k\otimes\mathcal H^{n-k})$, and $(C(Z),d_C,\mathcal H^{n-k})$ is an $\mathsf{ncRCD}(0,n-k)$ cone. Moreover, $(Z,d_Z,\mathcal H^{n-k-1})$ is an $\mathsf{ncRCD}(n-k-2,n-k-1)$ space, and $\mathcal H^{n-k}_{C(Z)}=r^{n-k-1}\,dr\otimes\mathcal H^{n-k-1}_Z$.
\end{lemma}

\noindent \textit{Sobolev spaces:} 

Several notions of first-order Sobolev spaces on metric measure spaces have been introduced. One of the most commonly used approaches originates in Cheeger's work \cite{cheeger1999differentiability}. Alternative formulations based on weak upper gradients and test plans were developed in \cite{Sha00,AmbrosioGigliSavare2013,AmbrosioGigliSavare2014}; see also \cite{HK00} for a related Haj{\l}asz-type approach. In the setting of metric measure spaces with $\sigma$-finite nonnegative Borel measures, these formulations agree; see \cite{AmbrosioGigliSavare2013,AmbrosioGigliSavare2014,GigliPasqualetto2020}. 
Since we only work under the $\mathsf{RCD}$ setting, we may equivalently define these spaces as the completion of Lipschitz functions with respect to the Sobolev norm. We briefly recall this approach below, together with some additional properties that will be useful later.
\par
Let $(X,d,\mu)$ be an $\mathsf{RCD}(K,N)$ space, $\Omega\subset X$ an open domain, and $f\in\Lip_{\loc}(\Omega)$. The \emph{pointwise Lipschitz constant}, also
called the \emph{local Lipschitz constant}, of $f$ at $x\in\Omega$ is
defined by $$\Lip f(x) := \limsup_{y \to x} \frac{|f(y)-f(x)|}{d(x,y)}.$$ 
Clearly, the map $x\mapsto\Lip f(x)$ is $\mu$-measurable. 
\par
Fix $1<p<\infty$. For
$f\in\Lip_{\loc}(\Omega)$ such that
$f,\Lip f\in L^p(\Omega)$, set
\[
\|f\|_{W^{1,p}(\Omega)}
:=
\|f\|_{L^p(\Omega)}
+
\|\Lip f\|_{L^p(\Omega)}.
\]
The Sobolev space $W^{1,p}(\Omega)$ is defined as the completion, with respect to this norm, of the space of locally Lipschitz functions having finite $W^{1,p}(\Omega)$-norm. Similarly, $W^{1,p}_0(\Omega)$ is defined as the closure of $\Lip_c(\Omega)$ in $W^{1,p}(\Omega)$, where $\Lip_c(\Omega)$ denotes the space of Lipschitz functions compactly supported in $\Omega$. Finally, a function $f$ belongs to $W^{1,p}_{\loc}(\Omega)$ if $f|_{\Omega'}\in W^{1,p}(\Omega')$ for every open set $\Omega'\Subset\Omega$. We collect below the properties of these spaces that will be used throughout the paper.

\begin{proposition}\label{prop:sobolev}
Under the settings above, we have:
\begin{enumerate}
\item For each $f \in W^{1,p}(\Omega)$, there is a function in $L^p(\Omega)$,
denoted by $|\nabla f|$, such that
$\|f\|_{W^{1,p}(\Omega)} = \|f\|_{L^p(\Omega)} + \||\nabla f|\|_{L^p(\Omega)}$.
If $f \in \Lip_{\loc}(\Omega)$, then $|\nabla f| = \Lip f$ holds $\mu$-a.e.\ in $\Omega$.
\item If $f_j \in W^{1,p}(\Omega)$ and $f_j \to f$ in $L^p(\Omega)$, then
$
\liminf_{j\to\infty} \|\nabla f_j\|_{L^p(\Omega)}
\ge \|\nabla f\|_{L^p(\Omega)}.$
\item
If $f,g \in W^{1,p}(\Omega)$ and $f=g$ $\mu$-a.e.\ on some Borel set
$A \subset \Omega$, then $|\nabla f|(x)=|\nabla g|(x)$ $\mu$-a.e.\ in $A$.
\item
The space $W^{1,2}(\Omega)$ is a Hilbert space, and the inner product
$\langle \nabla f,\nabla g\rangle \in L^1(\Omega)$ for $f,g \in W^{1,2}(\Omega)$
can be defined by the polarization identity:
\begin{equation}\label{eq:polarization}
\langle \nabla f,\nabla g\rangle
= \frac14\bigl(|\nabla(f+g)|^2 - |\nabla(f-g)|^2\bigr).
\end{equation}
\end{enumerate}
We remark that $W^{1,2}(\Omega)=H^{1,2}(\Omega)$, with $H^{1,2}(\Omega)$ as defined and explained in \cite{AH18}.
\end{proposition}

\noindent \textit{Measure-valued Laplacian}: 

\noindent The following concept was developed in \cite{gigli2015differential}; see also \cite{GigliPasqualetto2020}.
Let $(X,d,\mu)$ be an $\mathsf{RCD}(K,N)$ space, $\Omega\subset X$ a domain, and $f\in W^{1,2}_{\loc}(\Omega)$.
We say that $f$ has a measure-valued Laplacian, denoted by $f \in D(\bLap)$, if there is a finite signed Radon measure $\lambda$ on $\Omega$ such that
$$\int_\Omega g\ d\lambda = - \int_\Omega \nabla g\cdot \nabla f \ d\mu,$$
for every $g \in \Lip_{c}(\Omega)$.

\subsubsection{Existence, Regularity and Euler-Lagrange Equation}

Let $(X,d,\mu)$ be an $\mathsf{RCD}(K,N)$ space and $\Omega\subset X$ a bounded domain. 
In \cite{CZZ22}, the authors studied\footnote{More generally, \cite{CZZ22} considered the vector-valued problem for maps $u:\Omega\to[0,\infty)^m$ and \emph{local} minimizers.
Also, note that \cite{CZZ22} used $Q$ instead of $Q^2$ for the coefficient. Their results have been rewritten to be consistent with our notation.} 
the one-phase problem \eqref{1-phase_Ber} in this $\mathsf{RCD}$ setting and answered the main questions regarding existence and regularity of minimizers. In what follows, we quote some of the results in a simplified form or with simplified hypotheses.

\begin{theorem}[{\cite[Proposition 1.3, Theorem 1.4, 1.6]{CZZ22}}]\label{thm:mainCZZ22}
 Let $(X,d,\mathcal H^n)$ be an $\mathsf{ncRCD}(K,n)$ space and $\Omega\subset X$ a bounded domain.
 \begin{enumerate}[leftmargin=25pt, topsep=0pt]
     \item[1.] (Existence) If $\diam(\Omega) \leq \diam(X)/3$, then for each boundary data $u_0 \in W^{1,2}(\Omega, [0,\infty))$, there exists a minimizer $u \in W^{1,2}(\Omega)$.
 \end{enumerate}
 \begin{enumerate}[leftmargin=25pt,topsep=0pt]
     \item[2.] (Local Lipschitz regularity) For each ball $B_{R}(x)\subset \Omega$, there exists a constant $L$, depending on $n,K,\Omega,R,\|Q\|_{L^\infty},$ and $\|u\|_{L^1(B_R(x))}$, such that $\abs{u(y)-u(z)}\leq L \cdot d(y,z) \ \forall y,z \in B_{R/4}(x)$.
     \item[3.] (Local finiteness) $\{u>0\}\cap\Omega$ is a set of locally finite perimeter in $\Omega$.
     Moreover, we have $\pt^e\{u>0\}=\pt\{u>0\}$ in $\Omega$, where $\partial^e$ is the essential boundary.
     Hence $\mathcal{H}^{n-1}(\partial\{u>0\}\cap \Omega') < +\infty$ for every $\Omega' \Subset\Omega$.
     \item[4.] (Euler--Lagrange equation) There is a nonnegative Borel function $q$ such that $$\bLap u = q\mathcal{H}^{n-1}\llcorner(\partial\{u>0\}\cap\Omega),$$ with $q(x) = Q(x)$ for $\mathcal{H}^{n-1}$-a.e.\@ $x \in \partial\{u>0\}\cap\Omega$.
 \end{enumerate}
\end{theorem}

\subsubsection{Free Boundary and Its Regularity}

The next question addressed in \cite{CZZ22} concerns the regularity of the free boundary in the nonsmooth $\mathsf{ncRCD}$ setting. 
Recall that, in the Euclidean setting, the singular set of the free boundary has codimension \(k^* \in \{5,6,7\}\), while the regular set is a \(C^{1,\alpha}\)-manifold.
In the nonsmooth setting, however, one must first clarify what \emph{regular} and \emph{singular} parts mean. 
The key distinction is that singularities may originate from the ambient space itself.
Recall that for noncollapsed lower Ricci limit spaces, or more generally $\mathsf{ncRCD}$ spaces, one has a well-defined notion of tangent cones, obtained by blowing up the space at a given point. Since only lower Ricci curvature bounds are available, 
almost Euclidean behavior at small scales does not in general imply smooth regularity on a neighborhood. 
Consequently, it may happen that a point has an Euclidean tangent cone, while the space is not a smooth manifold in any neighborhood of the point. In fact, the singularities of the space itself may be dense.
\par

Turning back to the one-phase problem, \cite{CZZ22} developed the necessary tools to introduce blow-ups of minimizers (see \Cref{thm:compactness}).
They then define the regular set $\mathcal{R}$ as the collection of points $x \in \partial\{u>0\}\cap \Omega$ such that every tangent cone at $x$ is $(\mathbb{R}^n,d_{\mathrm{Euc}})$, and every blow-up limit of $\partial\{u>0\}\cap\Omega$
at $x$ is an affine hyperplane in $\mathbb{R}^n$. The singular set $\mathcal{S}$ is then defined as the complement of $\mathcal{R}$ in the free boundary.

Two basic examples illustrate the singular phenomena that may arise in this nonsmooth setting.

\begin{example}[Codimension $3$ singularities, {\cite[Example 1.8]{CZZ22}}]
There exists a $2$-dimensional Alexandrov space $Y$ with nonnegative curvature, constructed by doubling a triangle. Under this construction, the edges of the triangle are identified and disappear, while the vertices persist as isolated singular points. On the product space $(-1,1)_t\times Y$, the function
\[
f(t,y):=\begin{cases}
t, & t\ge 0,\\
0, & t<0,
\end{cases}
\]
is a minimizer of the one-phase problem with $Q\equiv 1$, and its free boundary is given by $\{0\}\times Y$. Consequently, the free boundary contains exactly the three vertices of $Y$ as singular points.
\end{example}

\begin{example}[Dense singularities in the free boundary, {\cite[Remark 1.12]{CZZ22}}]
The same construction applies if one replaces $Y$ by a two-dimensional Alexandrov space without boundary whose singular set is dense. Such a space was constructed in \cite{otsu1994riemannian}. In this case, the resulting free boundary has a dense singular set.
\end{example}

The first example tells us that the best we can hope for is an $(n-3)$-dimensional bound on the singular set. The second example tells us that we cannot expect a manifold structure on the whole regular set as defined above.
Indeed, we have the following sharp results, again stated in a simplified form.

\begin{theorem}[{\cite[Theorem 1.9, Corollary 1.11]{CZZ22}}]
Let $u$ be a minimizer to the one-phase problem \eqref{1-phase_Ber} on an $\mathsf{ncRCD}(K,n)$ space $(X,d,\mathcal{H}^n)$ without boundary\footnote{For an $\mathsf{ncRCD}(K,n)$ space, the boundary is defined as $\partial X:=\overline{\mathcal S^{n-1}\setminus\mathcal S^{n-2}}$, where $\mathcal S^k:=\{x\in X:\text{ no tangent cone at $x$ splits off an $\mathbb R^{k+1}$ factor}\}$. See \cite{DPG2018,BrueNaberSemola2022}.}. Then
\begin{enumerate}[topsep=0pt, ]
    \item (Hausdorff dimension bound on the singular set) $\dim_{\mathcal{H}}\mathcal{S} \leq n-3$.
    \item (Holder Regularity away from a codim $3$ set) There is an open set $\mathcal{O}\subset\mathcal{R}$ which is $C^\beta$-biHölder homeomorphic to an $(n-1)$-dimensional manifold for some $\beta>0$, and we have $\dim_{\mathcal{H}}(\mathcal{R}\backslash\mathcal{O}) \leq n-3$.
\end{enumerate}
\end{theorem}
Note the following interesting point.
In general, in $\mathsf{ncRCD}(K,n)$ spaces without boundary, the singular set is of codimension two.
The theorem above says that, in a heuristic sense, the free boundary interacts transversely with the singularities of the ambient space, rather than fully inheriting their dimension.
This can be explained based on the proof of the theorem: after dimension reduction, it boils down to the fact that for any $1$-homogeneous global minimizer on a $2$-dimensional circular cone other than $\mathbb{R}^2$, the vertex point is not in the free boundary; see \cite{AL15} and also the proof of \Cref{prop:structurestrata}.

\subsection{pmGH-convergence and Compactness Results}\label{subsec:pmgh_convg_and_cptness}
\subsubsection{Pointed Measured Gromov--Hausdorff Convergence}  

In this work, we consider noncollapsing sequences of pointed Riemannian manifolds of fixed dimension $n$ with a uniform lower bound on the Ricci curvature.
Their limits, when equipped with the $n$-dimensional Hausdorff measure $\mathcal{H}^n$, naturally fit into the framework of $\mathsf{RCD}$ spaces.
Accordingly, we recall the notion of \emph{pointed measured Gromov--Hausdorff (pmGH) convergence}.
See e.g. \cite{GigliMondinoSavare2015} and \cite{gromov2007metric} for various definitions about convergence for metric measure spaces.

\begin{definition}\label{def:GH_pmGH}
Let $K\in\mathbb{R}$ and $N\in[1,\infty)$. Consider a sequence of
$\mathsf{RCD}(K,N)$ metric measure spaces $(X_j,d_j,\mu_j)$ with base
points $p_j\in X_j$. We say that $(X_j,d_j,\mu_j,p_j)$
\emph{converges in the pointed measured Gromov--Hausdorff sense} to a
pointed metric measure space $(X_\infty,d_\infty,\mu_\infty,p_\infty)$,
and write
\[
(X_j,d_j,\mu_j,p_j)
\xrightarrow{\mathrm{pmGH}}
(X_\infty,d_\infty,\mu_\infty,p_\infty),
\]
if there exist a proper metric space $(Z,d_Z)$ and isometric embeddings
$\Phi_j:X_j\to Z$ and $\Phi_\infty:X_\infty\to Z$ such that:
\begin{enumerate}
\item $\Phi_j(p_j)\to\Phi_\infty(p_\infty)$ in $Z$;

\item For every $R>0$, we have the convergence in Hausdorff distance:
    \[
    \Phi_j(B_R(p_j))
    \xrightarrow{\mathrm{H}}
    \Phi_\infty(B_R(p_\infty))
    \quad\text{in }Z;
    \]

\item The push-forward measures $(\Phi_j)_\#\mu_j$ converge weakly to $(\Phi_\infty)_\#\mu_\infty$ in duality with $C_c(Z)$.
\end{enumerate}
\end{definition}
We refer to such a choice of $(Z,d_Z)$ and of the embeddings $\Phi_j,\Phi_\infty$ as a \emph{pmGH realization} of the convergence.
Once a pmGH realization is fixed, it makes sense to speak of a sequence of points $x_j\in X_j$ converging to a point $x\in X_\infty$, by viewing all the spaces as subsets of the common ambient space $Z$. Similarly, convergence o f subsets may be understood through the Hausdorff distance in $Z$. In particular, whenever we write $B_R(x_i)\to B_R(x)$, this is understood as shorthand for the local convergence of these balls induced by a pmGH realization of the ambient spaces; moreover, we assume $x_i \to x$ under this notation.
\par
Equivalently, the pmGH convergence can be described intrinsically by Borel
$\varepsilon_j$-pGH approximations $F_j:X_j\to X_\infty$, with $\varepsilon_j\to0$ on every fixed bounded ball, which are compatible with the convergence of the measures: \[
\int_{X_j} \varphi \circ F_j d\mu_i \to \int_{X_\infty} \varphi \ d\mu_\infty \quad \text{for $\varphi \in C_c(X_\infty)$}.\]
Given a pmGH realization as above, such maps $F_j$ may be chosen so that, for every fixed $R>0$, $d_Z(\Phi_j(y),\Phi_\infty(F_j(y)))\leq\varepsilon_j(R)$ for every $y\in B_R(p_j)$, where $\varepsilon_j(R)\to 0$. In particular, for any sequence $x_j\in B_R(p_j)$ and any $x\in X_\infty$, we have $x_j\to x$ with respect to the pmGH realization if and only if $F_j(x_j)\to x$ in $X_\infty$. Thus, passing between the extrinsic realization and compatible Borel pGH approximations does not change the notion of convergence of points, and the choice between the two formulations is therefore largely a matter of convenience and notation.
\begin{remark}
    Under the pmGH convergence above, the limit space $(X_\infty, d_\infty, \mu_\infty)$ is also an $\mathsf{RCD}(K,N)$ space; see \cite{EKS15,GigliMondinoSavare2015} and the references therein. 
\end{remark}
\par
We now introduce the convergence of functions, mainly following \cite{AH18}; see also \cite{ambrosio2021embedding,GigliMondinoSavare2015,ambrosio2016weakstrongconvergencederivations,HondaRicciLpConvergence,KuwaeShioya2003,kuwae2005variationalconvergencemetricspaces} and the references therein. As with sets and points, convergence of functions on varying spaces makes sense only once a realization of the underlying convergence has been chosen.
\begin{definition}[Convergence of functions]
Let
  \[
(X_i,d_i,\mu_i,p_i)
\xrightarrow{\mathrm{pmGH}}
(X,d,\mu,p)
\]
as $\mathsf{RCD}(K,N)$ spaces, realized by embeddings $\Phi_i, \Phi$ into $(Z,d_Z)$.  Let $R>0$.
\begin{enumerate}[leftmargin=*]
    \item (Uniform convergence)
    Let $f_i, f$ be measurable functions on $B_R(p_i), B_R(p)$, respectively. We say that $f_i \to f$ uniformly on $B_R(p_i)$ if for any $\varepsilon>0$ there is $N(\varepsilon) \in \mathbb{N}$ and $\delta(\varepsilon)>0$ such that 
    \[
\sup_{\substack{x \in B_R(p_i),\ y \in B_R(p) \\ d_Z(\Phi_i(x), \Phi(y)) < \delta}}
\left| f_i(x) - f(y) \right| < \varepsilon, \quad \forall i \ge N(\varepsilon).
\]
\item ($L^2$ convergence) 
Let $f_i \in L^2(X_i, \mu_i)$ and $f \in L^2(X, \mu)$.
We say that $f_i$ \emph{$L^2$-weakly converge} to $f$ if $\sup_i \|f_i\|_{L^2} < \infty$ and $\int_{X_i} h f_i \dd\mu_i \to \int_X h f \dd\mu$, for all $h \in C_c(Z)$.
Moreover, we say that $f_i$ \emph{$L^2$-strongly converge}
to $f$ if $f_i$ $L^2$-weakly converge to $f$ and $\limsup_i \|f_i\|_{L^2} \le \|f\|_{L^2}$.
\item ($W^{1,2}$ convergence) We say that $f_i \in W^{1,2}(X_i, d_i, \mu_i)$ \emph{$W^{1,2}$-weakly converge}
to $f \in W^{1,2}(X, d, \mu)$ if $f_i$ $L^2$-weakly converge to $f$ and $\sup_i \||\nabla f_i|\|_{L^2} < \infty$. Strong convergence in $W^{1,2}$ is defined by requiring $L^2$-strong convergence
of the functions and $\lim_i \||\nabla f_i|\|_{L^2}=\||\nabla f|\|_{L^2}$.
\item ($L^2$ convergence on balls) Let $f_i\in L^2(B_R(p_i),\mu_i)$ and $f\in L^2(B_R(p),\mu)$.
We say that $f_i$ converge weakly (respectively, strongly) in $L^2(B_R(p_i))$ to $f$ if their zero extensions
$f_i\,1_{B_R(p_i)}\in L^2(X_i,\mu_i)$ and $f\,1_{B_R(p)}\in L^2(X,\mu)$ satisfy the weak (respectively, strong) $L^2$-convergence as in item~2.
\item ($W^{1,2}$ convergence on balls) Let $f_i\in W^{1,2}(B_R(p_i))$ and $f\in W^{1,2}(B_R(p))$. We say that $f_i$ converge weakly in $W^{1,2}(B_R(p_i))$ to $f$ if $f_i\to f$ weakly in $L^2(B_R(p_i))$ and $\sup_i \||\nabla f_i|\|_{L^2(B_R(p_i))}<\infty$. We say that $f_i$ converge strongly in $W^{1,2}(B_R(p_i))$ to $f$ if $f_i\to f$ strongly in $L^2(B_R(p_i))$ and $\lim_i \||\nabla f_i|\|_{L^2(B_R(p_i))}=\||\nabla f|\|_{L^2(B_R(p))}$.
\end{enumerate}
\end{definition}

\begin{remark}\label{rmk:L2pullback}
Suppose $B_R(p_i)\to B_R(p)$ with respect to given ambient pmGH convergence, and $f_i\to f$ strongly in $L^2(B_R(p))$. Then there exist Borel $\varepsilon_i$-pGH approximations
$G_i:B_R^X(p)\to B_R^{X_i}(p_i)$, with $\varepsilon_i\to0$, such that
\begin{enumerate}
    \item $(G_i)_\#\mu\ll\mu_i$;
    \item $f_i\circ G_i\to f$ in $L^2(B_R^X(p),\mu_\infty)$.
\end{enumerate}
Here the convergence of the balls is understood with respect to the realization induced by the ambient pmGH convergence. In our applications, the functions $f_i$ are continuous and therefore have canonical pointwise representatives. Hence, in our applications, item~(1) may be dispensed with, since the compositions \(f_i\circ G_i\) are pointwise well defined. Moreover, if $f_i\to f$ uniformly with respect to the pmGH convergence, then item~(2) follows immediately for any compatible choice of reverse pGH approximations $G_i$.
\end{remark}

\begin{lemma}[Local uniform convergence implies local strong $L^2$ convergence]
\label{lem:uniform_implies_L2}
Let $(X_i,d_i,\mu_i,p_i)$ and $(X,d,\mu,p)$ be $\mathsf{RCD}(K,N)$ spaces such that $(X_i,d_i,\mu_i,p_i) \xrightarrow{\mathrm{pmGH}} (X,d,\mu,p)$.
Let $v_i:X_i\to\mathbb R$ be locally bounded Borel functions converging uniformly on compact sets to a continuous function $v:X\to\mathbb R$, with respect to the above pmGH convergence. Then, for every $R>0$,  $v_i$ converge strongly in $L^2(B_R(p_i))$ to
$v$. 
\end{lemma}

\subsubsection{Compactness of Solutions and Blow-Ups}
Here we state the compactness result on which we rely to take limits of minimizers.
This result will be used both in the blow-up analysis and in the study of limit solutions (see \cref{subsec:lim_sol}). A more general version of the following theorem is proved in \cite[Theorem 7.1]{CZZ22}.

\begin{theorem}\label{thm:compactness}
Let $(X_j,d_j,\mathcal{H}^n,p_j)$ be a sequence of $\mathsf{ncRCD}(k,n)$ spaces such that
\[
(X_j,d_j,\mathcal H^n,p_j) \xrightarrow[]{\mathrm{pmGH}} (X,d,\mathcal{H}^n,p_\infty).
\]
Let $Q_j\in C(\overline{B_R(p_j)})$ and
$Q_\infty\in C(\overline{B_R(p_\infty)})$ satisfy
\[
0<Q_{\min}\le Q_j\le Q_{\max}<\infty,
\]
and assume that
\[
Q_j\to Q_\infty
\quad\text{uniformly on }B_R(p_j)
\]
with respect to the pmGH convergence. Let $u_j$ be a sequence of minimizers of $E_{Q_j}$ on $B_R(p_j)$. Assume that the sequence $\{u_j\}$ is uniformly bounded and $L$-Lipschitz on $B_R(p_j)$. Then there exists a subsequence, still denoted by $\{u_j\}$, and a $L$-Lipschitz function $u_\infty$ on $B_R(p_\infty)$ such that $u_j \to u_\infty$ uniformly on $B_R(p_j)$ as $j \to \infty$. Moreover, for every $R' \leq R$, the limit function $u_\infty$ is a minimizer of $E_{Q_\infty}$ on $B_{R'}(p_\infty)$. In addition, the following properties hold:
\begin{align}
\lim_{j \to \infty} \int_{B_{R'}(p_j)} |\nabla u_j|^2 \, d\mathcal H^n
&=
\int_{B_{R'}(p_\infty)} |\nabla u_\infty|^2 \, d\mathcal H^n,
\\[0.5em]
\lim_{j \to \infty} \int_{B_{R'}(p_j)} Q_j^2 1_{\{u_j>0\}} \, d\mathcal H^n
&=
\int_{B_{R'}(p_\infty)} Q_\infty^2 1_{\{u_\infty>0\}} \, d\mathcal H^n,
\\[0.5em]
\partial\{u_j>0\} \cap B_{R'}(p_j)
&\xrightarrow{}
\partial\{u_\infty>0\} \cap B_{R'}(p_\infty),\\[0.5em]
\mathcal H^n\!\left(B_{R'}(p_j)\cap\{u_j>0\}\right)
&\to
\mathcal H^n\!\left(B_{R'}(p_\infty)\cap\{u_\infty>0\}\right). \label{eq:freeBoundaryvolumeconv}
\end{align}
\end{theorem}
\begin{corollary}[Existence of blow-ups]\label{cor:blowupmaps}
Let $(X,d,\mathcal{H}^n,p)$ be an $\mathsf{ncRCD}(K,n)$ space,
$u$ a minimizer of $E_Q$ on $\Omega\subset X$, with
$Q\in C(\Omega)$ positive, and $p\in \pt\{u>0\}\cap \Omega$.
Consider the
sequence $r_j^{-1}u$ on $(X,r_j^{-1}d,\mathcal{H}^n_{r_j^{-1}d},p)$, for some $r_j\to0$.
Then, up to subsequences,
\begin{enumerate}[topsep=0pt]
    \item $(X,r_j^{-1}d,\mathcal{H}^n_{r_j^{-1}d},p)
    \xrightarrow[]{\mathrm{pmGH}}(Y,\rho,\mathcal{H}^n,y)$, and $r_j^{-1}u$ converges
    uniformly on compact sets to some $u_\infty$ on $(Y,\rho,\mathcal{H}^n,y)$.
    
    \item $(Y,\rho,\mathcal{H}^n,y)$ is an $\mathsf{ncRCD}(0,n)$ cone, with $y$ the vertex.
    Such $Y$ is called a \emph{tangent cone} of $X$ at $p$.
    
    \item $u_\infty$ is a non-trivial global minimizer of $E_{Q(p)}$ on $Y$, with $y\in \pt\{u_\infty>0\}$.
    We call $u_\infty$ a \emph{blow-up} of $u$ at $p$.

\end{enumerate}
\end{corollary}
\begin{proof}
Set $d_j:=r_j^{-1}d$ and $u_j:=r_j^{-1}u$. The rescaled spaces
$(X,d_j,\mathcal H^n_{d_j},p)$ are $\mathsf{ncRCD}(r_j^2K,n)$ spaces.
By the compactness and volume convergence \cite{DPG2018}, after passing to a subsequence, we have
\[
(X,d_j,\mathcal H^n_{d_j},p)
\xrightarrow{\mathrm{pmGH}}
(Y,\rho,\mathcal H^n,y).
\]
Item 2 is
a consequence of the adaptation of almost volume cone implies almost
metric cone in the nonsmooth context, given in \cite{DPG16}. We now consider the function level. Fix $R>0$. For all sufficiently large $j$, $B^d_{Rr_j}(p)\Subset \Omega$. Since $u$ is locally Lipschitz and $u(p)=0$, the functions $u_j$ are uniformly Lipschitz and uniformly bounded on $B^{d_j}_R(p)$. Moreover,
\[
\sup_{B^{d_j}_R(p)} |Q-Q(p)|
=
\sup_{B^d_{Rr_j}(p)} |Q-Q(p)|
\to 0,
\]
so the rescaled coefficients converge locally uniformly to the constant $Q(p)$. Hence, by \Cref{thm:compactness} and a diagonal argument, after passing to a further subsequence, $u_j$ converges locally uniformly to a function $u_\infty$ on $Y$, which minimizes $E_{Q(p)}$ on every ball and is therefore a global minimizer.
Finally, the nondegeneracy estimate \cite[Theorem~6.1]{CZZ22} implies that $u_\infty$ is nontrivial, and the convergence of the free boundary in \Cref{thm:compactness} implies that $y\in \pt\{u_\infty>0\}$. 
\end{proof}
Note that we do not assert any symmetry on $u_\infty$.
In fact, without a monotone quantity in this general nonsmooth context, it is unknown whether blow-ups are always $1$-homogeneous on the tangent cone.

\begin{remark}
Based on the Weiss quantity on $\mathsf{ncRCD}$ cones in \cite[Appendix A]{CZZ22}, we only know that double-blow-ups are indeed $1$-homogeneous on the corresponding (double) tangent cone.
This implies that some, but not all, blow-ups of $u$ are $1$-homogeneous; namely, those that can be realized as double-blow-ups. 
Alternatively, the existence of homogeneous blow-ups follows from \Cref{thm:parabolic_weiss} and \Cref{lem:alm_rigid}, since they allow us to find arbitrarily small and arbitrarily good scales.
Classically, having some symmetry on $u_\infty$ is the first step for the dimension-reduction arguments to give a Hausdorff dimension bound on the singular strata. 
Still, even when such symmetry of blow-ups is not available, in \cite{wang2026dimensionboundsingularset}, inspired by \cite{FMS25}, it was shown that such monotonicity on cones is sufficient for the dimension-reduction argument.
\end{remark}

\subsection{Limit Solutions}\label{subsec:lim_sol}
    In this article, we only consider \textit{limit solutions} on noncollapsed Ricci limit spaces.
    These are limits of solutions to the one-phase problem \eqref{1-phase_Ber} on smooth manifolds that GH-converge to the limit spaces.
    More precisely, we consider pointed metric measure spaces $(X,d,\mathcal{H}^n,p)$, bounded domains $\Omega \subset X$ with $p\in \Omega$, $u:\Omega \to [0,\infty)$,  and strictly positive $Q\in C^\alpha(\Omega)$ satisfying the following:

    \begin{enumerate}[leftmargin=15pt]
        \item At the space level: There is a sequence of pointed $n$-dimensional complete Riemannian manifolds $(M_i,g_i,p_i)$ satisfying $\Ric_{g_i}\geq -(n-1)$ and the noncollapsing condition $\Vol_i\big(B_1^{M_i}(p_i)\big)\geq \mathrm{v}'$ for some constant $\mathrm{v}'>0$, such that $(M_i,g_i,p_i) \xrightarrow[]{\mathrm{pGH}} (X,d,p)$. 
        $(X,d,p)$ is a so-called \emph{noncollapsed Ricci limit space}.
        Note that by the celebrated volume convergence theorem \cite{colding1997ricci}, the Riemannian volume measures on $M_i$ converge to the $n$-dimensional Hausdorff measure $\mathcal{H}^n$ on $X$, and hence we actually have (see also \cite{DPG2018})
        \begin{equation}
        (M_i,g_i,\vol_i,p_i) \xrightarrow[]{\mathrm{pmGH}} (X,d,\mathcal{H}^n,p).
        \end{equation}
        In addition, $(X,d,\mathcal{H}^n,p)$ belongs to $\mathsf{ncRCD}(-(n-1),n)$ and satisfies $\mathcal{H}^n(B_1(p))\geq \mathrm v'$.\footnote{There exist $\mathsf{ncRCD}$ spaces that are not noncollapsed Ricci limit spaces; see \S 3.3, Remark 4 in \cite{Top20}.}
        We also assume that there is a sequence of (uniformly) bounded domains $\Omega_i\subset M_i$ with $p_i\in \Omega_i$ such that $\Omega_i\to\Omega$ with respect to the pmGH convergence.
        
        \item At the function and coefficient level: 
        There are a strictly positive sequence $Q_i\in C^\alpha(\Omega_i)$ and a sequence of minimizers $u_i$ of $E_{Q_i}$ on $\Omega_i$ such that $Q_i\to Q$ and $u_i\to u$ uniformly, with respect to the pmGH convergence above.
        Moreover, for each $B_R(x) \Subset \Omega$, there are $B_R(x_i)\to B_R(x)$ with $x_i\to x$, and constants $L, \Lambda$ such that for large enough $i$, we have $B_R(x_i)\Subset \Omega_i$ and
    \begin{equation}\label{eq:lip_bounded}
    \begin{aligned}
    &\|Q_i\|_{C^0(B_R(x_i))}
    +\|Q_i^{-1}\|_{C^0(B_R(x_i))}
    +R^\alpha[Q_i]_{\alpha,B_R(x_i)}
    \leq \Lambda,\\
    &u_i|_{B_R(x_i)}
    \text{ is $L$-Lipschitz and uniformly bounded along the sequence.}
    \end{aligned}
    \end{equation}
    Lastly, if $x\in\partial\{u>0\}$, then, by \Cref{thm:compactness}, the approximating centers $x_i$ can be chosen in $\partial\{u_i>0\}$ for all sufficiently large $i$.
    \end{enumerate}

As a consequence of \Cref{thm:compactness}, the limit solution $u$ minimizes the one-phase problem on each ball $B_R(x) \Subset\Omega$, although it is not clear that $u$ minimizes the problem on the whole domain $\Omega$.
In fact, it is not clear whether $u \in W^{1,2}(\Omega)$. These do not create issues for us since our main results are essentially local. On the other hand, it is not clear whether every minimizer on a noncollapsed Ricci limit space is a limit solution, i.e., approximable by minimizers on smooth manifolds in a locally uniform fashion.
One reason for this is the non-uniqueness nature of the one-phase problem.

\begin{remark}
For $B_{2R}(x_i)\Subset\Omega_i$ with $x_i\in\partial\{u_i>0\}$, the restrictions $u_i|_{B_R(x_i)}$ are uniformly $L$-Lipschitz, where $L$ depends only on $n,R,\Lambda$ and a positive lower bound for
\[
\frac{\Vol_i(B_{2R}(x_i))}{(2R)^n},
\]
which follows from our noncollapsing condition. Here $\Lambda$ denotes the coefficient bound on the $2R$-balls. This Lipschitz estimate follows from the arguments of \cite[Lemmas~5.6--5.7 and the proof of Theorem~1.4]{CZZ22}, using that the $u_i$ are minimizers and $d(B_R(x_i), \pt\Omega_i)\geq R$.
Moreover, since $u_i(x_i)=0$, we also have
\[
\sup_{B_R(x_i)}u_i\leq LR.
\]
Thus, on such balls, the uniform Lipschitz and boundedness
requirements in \eqref{eq:lip_bounded} follow from the geometric
and coefficient bounds, provided the approximating centers
lie on the free boundaries.
No additional control of the boundary data is needed.
\end{remark}

\subsection{Heat Kernel Estimates and Convergence}\label{subsec:heat_kernel}
Here we summarize some heat kernel estimates that will be used in this work, along with some convergence results. There is a well-developed theory of heat flows on $\mathsf{RCD}(K,N)$ spaces; see \cite{AmbrosioGigliSavare2014}. As a consequence of Sturm's work on the general theory of Dirichlet forms on PI spaces, there exists a locally H\"older continuous representative of the heat kernel $\rho$, see \cite{Sturm1995DirichletII,Sturm1996DirichletIII}. Moreover, in \cite{JiangLiZhang2016}, it is shown that on $\mathsf{RCD}(K,N)$ spaces the heat kernel $\rho$ admits a locally Lipschitz representative satisfying Gaussian estimates.
Since we are concerned with limit solutions, it will be sufficient to restrict our attention to the smooth manifold setting.
\par
Let $M^n$ be a complete Riemannian manifold, and denote by $\rho_t(x,y)$ the heat kernel on \( M^n \). As in the introduction, we define the function $f_t(x,y)$ by
\begin{equation}\label{eq:4.14}
f_t := -\log\!\big((4\pi t)^{n/2}\rho_t\big).
\end{equation}
The following are classical heat kernel estimates for manifolds with lower Ricci curvature bounds (see \cite{LY86,SZ06,SY94,Ham93,Kot07,CY81}). We state them in the form presented in \cite{CJN21}, where we also refer the reader for comments on the proofs and further related references. See also \cite{brena2025perelman} for the corresponding result in the $\mathsf{RCD}(K,N)$ setting.

\begin{theorem}[Heat Kernel Estimates]\label{thm:heat-kernel}
Suppose $(M^n,g,p)$ satisfy $\Ric_{g} \ge -(n-1)\delta^2$ and
$\Vol(B_r(p)) \ge \mathrm{v} r^n > 0$ for $r \le \delta^{-1}$.
Then for any $0 < t \le 10\delta^{-2}$ and
$\varepsilon > 0$ with $x,y \in B_{10\delta^{-1}}(p)$,
\begin{enumerate}[topsep=0em]
\item $-C(n,v,\varepsilon) + \frac{d^2(x,y)}{(4+\varepsilon)t} \le f_t \le C(n,v,\varepsilon) + \frac{d^2(x,y)}{(4-\varepsilon)t} $
\item $t|\nabla f_t|^2 \le C(n,v,\varepsilon)
+ \frac{d^2(x,y)}{(4-\varepsilon)t}$
\item $-C(n,v,\varepsilon) - \frac{d^2(x,y)}{(4+\varepsilon)t}
\le t\Delta f_t
\le C(n,v,\varepsilon) + \frac{d^2(x,y)}{(4-\varepsilon)t} .$
\end{enumerate}
\end{theorem}

In the arguments below, we will deal with a sequence of heat kernels on varying spaces. We refer the reader to the preliminaries in \cite{ambrosio2021embedding} for an exposition of convergence results for heat kernels on $\mathsf{RCD}(K,N)$ spaces. The next proposition is stated and proved in full generality in \cite{ambrosio2021embedding}, and relies on the $W^{1,2}$ convergence of functions with $L^2$ bounds on the Laplacian, see \cite{AH18} for a proof in full generality, and \cite{Ding2004} for a proof in a simplified setting. We also refer the reader to the preliminaries in \cite{CJN21}, where a presentation more focused on the smooth setting is given.

\begin{theorem}[\(W^{1,2}\)-strong convergence of heat kernels, {\cite[Theorem 2.19]{ambrosio2021embedding}}]\label{thm:heatkernelconvergence}
Suppose $(X_i,d_i,\mu_i)$ are $\mathsf{RCD}(K,N)$ spaces with $K\in\mathbb R$
and $1<N<\infty$ fixed such that $(X_i,d_i,\mu_i,x_i) \xrightarrow[]{\mathrm{pmGH}}(X,d,\mu,x)$.
For any sequences $t_i\to t \in (0,\infty)$ and $y_i\in X_i\to y\in X$, we have that
$\rho_{t_i}(\cdot,y_i)$
converge \(W^{1,2}\)-strongly to \(\rho_{t}(\cdot,y)\in W^{1,2}(X,d,\mu)\).
\end{theorem}
We recall that if $(C(Z),d,\mathcal H^n,z_*)$ is an $\mathsf{ncRCD}(0,n)$ cone with vertex $z_*$, then its heat kernel satisfies
\begin{equation}
\rho_t^{C(Z)}(z_*,y)
=
\frac{\Vol(S^{n-1})}{\Vol(Z)}
\frac{e^{-d^2(z_*,y)/(4t)}}{(4\pi t)^{n/2}}, \label{eq:heatkernelcone}
\end{equation}
where $\Vol(Z):=\mathcal H^{n-1}(Z)$. More generally, the corresponding formula for $N$-metric-measure cones is proved in \cite[Proposition~2.10]{HondaPeng2025}, based on \cite[Proposition~4.10]{KuwadaLi2021}. 

The following lemma is useful in the proof of the Almost Rigidity \Cref{lem:alm_rigid}. Since we did not find a reference where this consequence is stated explicitly, we provide the proof.

\begin{lemma}[Local $W^{1,2}$ convergence of $f_{t,x}$]
\label{lem:ft_convergence}
Let $(X_i,d_i,\mu_i,x_i)\xrightarrow{\mathrm{pmGH}}(X,d,\mu,x)$
be a sequence of $\mathsf{RCD}(K,N)$ spaces, with $K\in\mathbb R$
and $1<N<\infty$ fixed. Fix $t>0$, $R>0$, and a fixed integer $n$.
Let $\rho_{t,x_i}$ and $\rho_{t,x}$ denote the corresponding heat
kernels, and set
\[f_{t,x_i}
:=
-\log\bigl((4\pi t)^{n/2}\rho_{t,x_i}\bigr),
\qquad
f_{t,x}
:=
-\log\bigl((4\pi t)^{n/2}\rho_{t,x}\bigr).
\]
Then
\[
f_{t,x_i}\to f_{t,x}
\qquad\text{strongly in }W^{1,2}(B_R(x)).
\]
\end{lemma}
\begin{proof}
We drop the $t,x$ from the notation. By \Cref{thm:heatkernelconvergence}, $\rho_i\to\rho$ strongly in $W^{1,2}$. We first recall that the heat kernels converge locally uniformly.
The two-sided Gaussian estimates and the gradient estimate, recorded in
\cite[(2.13)--(2.14)]{ambrosio2021embedding} and originating from
\cite{JiangLiZhang2016}, together with volume comparison and
the pmGH convergence, imply that for every fixed $R<\infty$ there exist
constants $c=c(K,N,R,t), C=C(K,N,R,t)>0$ such that, for all sufficiently large $i$, we have $c\leq \rho_i\leq C$ and $|\nabla\rho_i|\leq C $ on $B_{R+1}(x_i)$. This gives the uniform convergence of $\rho_i$ on compact sets. 

The Gaussian lower bound also gives, on every fixed ball, a uniform positive lower bound for $\rho_i$. Thus $-\log$ is uniformly Lipschitz on the relevant range and therefore $f_i \to f$ locally uniformly. Since the boundary of every ball has zero measure in the
$\mathsf{RCD}(K,N)$ setting, the locally uniform convergence, implies $f_i\to f$ strongly in $L^2(B_R(x_i))$; see \Cref{lem:uniform_implies_L2}.
\par
It remains to prove the convergence of the Cheeger energies of $f_i$.
By \cite[Theorem~2.8(ii)]{AH18}, the global strong $W^{1,2}$
convergence of the heat kernels implies
\[
|\nabla\rho_i|^2
\to
|\nabla\rho|^2
\qquad\text{strongly in }L^1.
\]
In particular, this convergence holds strongly in $L^1(B_R(x_i))$, since $\mu_i(\partial B_R(x_i))=\mu(\pt B_R(x))=0$. On the other hand, the local uniform convergence and the uniform positive lower bound imply $\rho_i^{-2}\to\rho^{-2}$ uniformly on $B_R(x_i)$. By the Sobolev chain rule,
\[
|\nabla f_i|^2
=
\frac{|\nabla\rho_i|^2}{\rho_i^2},
\qquad
|\nabla f|^2
=
\frac{|\nabla\rho|^2}{\rho^2}.
\]
The strong $L^1$ convergence of $|\nabla\rho_i|^2$, together with the
uniform convergence and uniform boundedness of $\rho_i^{-2}$, therefore
gives
\[
\int_{B_R(x_i)}
|\nabla f_i|^2\,d\mu_i
\to
\int_{B_R(x)}
|\nabla f|^2\,d\mu.
\]
Together with the strong $L^2$ convergence proved above, this shows $f_i\to f$ strongly in $W^{1,2}(B_R(x_i))$.
\end{proof}

\section{Quantitative Singular Strata of Two Aspects}\label{sec:QuantitiveStrataTwoTypes}

This section is devoted to the singular strata as defined in \Cref{def:quant_sing_stra} and the following paragraph there. In \Cref{subsec:stab_sym_and_strata}, we establish the main properties of the quantitative strata, including their connection to the classical singular strata, which are characterized by the symmetry of blow-ups. In \Cref{subsec:reifenberg}, we introduce regular and singular points and establish the Reifenberg property.
In \Cref{subsec:strata_regularity}, we study the structure of the strata and deduce a locally bi-H\"older manifold structure away from the top quantitative stratum.

\subsection{Stability of Quantitative Symmetry and Strata}\label{subsec:stab_sym_and_strata}

We start by studying the symmetric models $(B_r(0,z_*),\tilde{u})$ from the function point of view. The lemma below shows that $\tilde{u}$ can be chosen to be almost Lipschitz with effective control depending on the Lipschitz constant of $u$.
This will be crucial for obtaining compactness of the symmetric models.

\begin{lemma}\label{lem:almost_lip_sym_model}
    Let $(B_r(x),u)$ be $(k,\eta)$-symmetric with a symmetry model $(B_r(0,z_*), \tilde{u}))$ on the cone $\mathbb{R}^k\times C(Z)$.
    If $u$ is $L$-Lipschitz on $B_r(x)$, then there is a function $\hat{u}$, defined on the same cone, that is $1$-homogeneous and $\mathbb{R}^k$-invariant on the whole cone, and \emph{$(C(n)L,\eta r)$-almost Lipschitz} on $B_r(0,z_*)$, in the sense that
    \begin{equation}
        \abs{\hat{u}(p_1)-\hat{u}(p_2)} \leq C(n)L(d(p_1,p_2) + \eta r)
    \end{equation}
    for any $p_1, p_2\in B_r(0,z_*)$.
Furthermore, we have
\begin{equation}
    r^{-n-2}\int_{B_r(0,z_*)} \abs{ u\circ F - \hat{u}}^2 \dd{}\mathcal{H}^n < C(n)\eta,
\end{equation}
for some dimensional constant $C(n)$.
In other words, up to such a constant, we can assume that $\tilde{u}$ is $(L,\eta r)$-almost Lipschitz on the corresponding ball.
\end{lemma}

\begin{proof}
By rescaling, we may assume that $r=1$.
    By definition, we have a Borel $\eta$-GH approximation $F:B_1(0,z_*)\to B_1(x)$, such that
        \begin{equation}
            \int_{B_1(0,z_*)} |u\circ F-\tilde{u}|^2 \dd\vol\leq \eta.
        \end{equation}
    By Fubini and \Cref{lem:ncRCD_cone_struct}, we can rewrite this as
    \begin{equation}
        \int_{B^k_1(0)} \int_{B^{C(Z)}_{\sqrt{1-|y|^2}}(z_*)}|u\circ F(y,\cdot)-\tilde{u}(y,\cdot)|^2 \dd{} \mathcal{H}^{n-k}\dd{}\mathcal{H}^k(y)\leq \eta.
    \end{equation}
    Then there is $y_*\in B^k_{1/2}(0)$ such that, for some constant $C(n)$,
    \begin{equation}
            \int_{B^{C(Z)}_{\sqrt{1-|y_*|^2}}(z_*)}|u\circ F(y_*,\cdot)-\tilde{u}(y_*,\cdot)|^2 \dd{} \mathcal{H}^{n-k}\leq C(n)\eta.
    \end{equation}
Next, by the coarea formula (for $\mathsf{RCD}$ spaces, see \cite{Ambrosio2019rigidity}) with the distance function to $z_*$ (or alternatively by Fubini and \Cref{lem:ncRCD_cone_struct}), there is $r_*\in [1/2, \sqrt{1-|y_*|^2}) $ such that, for another constant $C(n)$,
\begin{equation}
    \int_{\pt B^{C(Z)}_{r_*}(z_*)}|u\circ F(y_*,\cdot)-\tilde{u}(y_*,\cdot)|^2 \dd{} \mathcal{H}^{n-k-1}\leq C(n)\eta.
\end{equation}
Now, we define $\hat{u}$ on $\{y_*\}\times \pt B^{C(Z)}_{r_*}(z_*)$ to be equal to $u\circ F(y_*, \cdot)$, and then extend it by $1$-homogeneity and $\mathbb{R}^k$-invariance to the whole $\mathbb{R}^k\times C(Z)$.
\par
We first verify that $\hat{u}$ is $(C(n)L,\eta)$-almost Lipschitz.
For $(y_1, w_1), (y_2,w_2)\in B_1(0,z_*)\subset \mathbb{R}^k\times C(Z)$, we have 
\begin{align}
    |\hat{u}(y_1, w_1)-\hat{u}(y_2, w_2)|
    &=\left|\frac{|w_1|}{r_*}\hat{u}(y_*,\frac{r_*}{|w_1|}w_1)-\frac{|w_2|}{r_*}\hat{u}(y_*,\frac{r_*}{|w_2|}w_2)\right|\\
    &=\left|\frac{|w_1|}{r_*}u\circ F(y_*, \frac{r_*}{|w_1|}w_1)-\frac{|w_2|}{r_*}u\circ F(y_*,\frac{r_*}{|w_2|}w_2)\right|\\
    &\leq \left|\frac{|w_1|}{r_*}u\circ F(y_*, \frac{r_*}{|w_1|}w_1)-\frac{|w_1|}{r_*}u\circ F(y_*,\frac{r_*}{|w_2|}w_2)
    \right|\\
    &\quad+\left|\frac{|w_1|}{r_*}u\circ F(y_*, \frac{r_*}{|w_2|}w_2)-\frac{|w_2|}{r_*}u\circ F(y_*,\frac{r_*}{|w_2|}w_2)
    \right|=:I_1+I_2,
\end{align}
where $|w|$ denotes $d_{C(Z)}(w, z_*)$; without loss of generality, we assume that $w_1, w_2$ are not $z_*$.
For the first term $I_1$, since $u$ is $L$-Lipschitz and $F$ is a $\eta$-GH approximation, we have
\begin{equation}
    I_1
    \leq \frac{|w_1|L}{r_*}\left(d_{C(Z)}(\frac{r_*}{|w_1|}w_1, \frac{r_*}{|w_2|}w_2)+\eta\right).
\end{equation}
Without loss of generality, we may assume $|w_1|\leq |w_2|$.
Then by the definition of cone metrics, we have
\begin{equation}
    d_{C(Z)}(\frac{r_*}{|w_1|}w_1, \frac{r_*}{|w_2|}w_2)
    =\frac{r_*}{|w_1|}d_{C(Z)}(w_1, \frac{|w_1|}{|w_2|}w_2)
    \leq \frac{r_*}{|w_1|}d_{C(Z)}(w_1,w_2).
\end{equation}
Hence we have
\begin{equation}
    I_1\leq Ld_{C(Z)}(w_1,w_2)+ \frac{|w_1|L}{r_*}\eta
    \leq Ld((y_1, w_1), (y_2, w_2))+C(n)L\eta.
\end{equation}
For the second term $I_2$, using the above properties of $u,F$ and that $u(x)=0$, we have
\begin{equation}
    I_2\leq \frac{1}{r_*}||w_1|-|w_2||L\cdot d(x, F(y_*, \frac{r_*}{|w_2|}w_2))
    \leq 2Ld((y_1, w_1), (y_2, w_2)).
\end{equation}
This concludes that $\hat{u}$ is $(C(n)L,\eta)$-almost Lipschitz on $B_1(0,z_*)$.
\par
We here derive the $L^2$-closeness of $u\circ F$ and $\hat{u}$.
First we have
\begin{equation}
    \int_{B_1(0,z_*)} \abs{ u\circ F - \hat{u}}^2 \dd\vol
    \leq 2\int_{B_1(0,z_*)} \abs{ u\circ F - \tilde{u}}^2 \dd\vol
    +2\int_{B_1(0,z_*)} \abs{ \tilde{u} - \hat{u}}^2 \dd\vol.
\end{equation}
By definition, the first integral is already bounded by $\eta$.
For the second integral, using the $\mathbb{R}^k$-invariance, the homogeneity, and the coarea formula, we have
\begin{align}
    \int_{B_1(0,z_*)} \abs{ \tilde{u} - \hat{u}}^2 \dd\vol
    &=\int_0^1 \int_{\pt B^{C(Z)}_r(z_*)} |\tilde{u}(y_*,\zeta) -\hat{u}(y_*,\zeta)|^2\dd{}\mathcal{H}^{n-k-1}(\zeta) (1-r^2)^\frac{k}{2}|B^k_1(0)| \dd r\\
    &\kern -4em\leq \int_0^1 \int_{\pt B^{C(Z)}_{r_*}(z_*)} |\tilde{u}(y_*,\zeta) -\hat{u}(y_*,\zeta)|^2\dd{}\mathcal{H}^{n-k-1}(\zeta)(\frac{r}{r_*})^{n-k+1} \dd r\cdot C(n).
\end{align}
Now, using our construction of $\hat{u}$ and the choice of $y_*, r_*$, we have
\begin{align}
    &\int_{\pt B^{C(Z)}_{r_*}(z_*)} |\tilde{u}(y_*,\zeta) -\hat{u}(y_*,\zeta)|^2\dd{}\mathcal{H}^{n-k-1}(\zeta)\\
    &=\int_{\pt B^{C(Z)}_{r_*}(z_*)}|\tilde{u}(y_*,\cdot)-u\circ F(y_*,\cdot)|^2 \dd{} \mathcal{H}^{n-k-1}\leq C(n)\eta.
\end{align}
This concludes the proof.
\end{proof}

The next Lemma follows from the usual argument, see for instance: discontinuous Ascoli-Arzel\`a’s theorem {\cite[Theorem 6.2]{droniou_uniform--time_2016}}.
\begin{lemma}\label{lem:AAarg}
Let $X_i$ be a sequence of compact metric spaces that $\mathrm{GH}$-converge to $X$.
Suppose $f_i$ is a sequence of uniformly bounded, $(L,\eta_i)$-almost Lipschitz functions on $X_i$ with $\eta_i\to 0$.
Then up to subsequences, there is a $L$-Lipschitz function $f$ on $X$ such that $f_i\to f$ uniformly.
\end{lemma}

The following proposition provides the details we need about the singular strata.

\begin{proposition}\label{prop:basicpropstrata}
Let $(X,d,\mathcal H^n)$ be an $\mathsf{ncRCD}(K,n)$ space, $\Omega\subset X$ a bounded domain, and $u:\Omega\to[0,\infty)$  a locally Lipschitz function. 
Consider the quantitative strata $\mathcal{S}^k_{\eta,r}(u)$ as in \Cref{def:quant_sing_stra}. 
The following observations hold. 

\begin{enumerate}

\item[(i)] \textit{(Inclusions)}
$\mathcal S^k_{\eta,r}(u)\subset
\mathcal S^{k'}_{\eta',r'}(u)$ whenever
$k'\geq k$, $r'\geq r$, and $\eta'\leq\eta$.

\item[(ii)] \textit{(Blow-ups)} $\mathcal{S}^k(u) = \{ x : \text{no blow-up at } x \text{ is } (k+1)\text{-symmetric} \}$.

\item[(iii)] \textit{(Symmetry and limits)}
Let $(X_i,d_i,\mathcal H^n,x_i)$ be a sequence of $\mathsf{ncRCD}(K,n)$ spaces such that $(X_i,x_i) \xrightarrow[]{\mathrm{pmGH}} (X,x)$.
Take 
\[
B_s(x_i)\xrightarrow[]{}B_s(x)\subset \Omega,
\]
and let $u_i:B_s(x_i)\to[0,\infty)$ be uniformly $L$-Lipschitz
functions converging uniformly to $u$ with respect to this
convergence. Assume $\mathcal{H}^n(B_1(x_i))\geq \mathrm{v}$ for some constant $\mathrm{v}>0$.
Then:
\begin{enumerate}

\item If $B_s(x_i)$ is $(k,\eta_i)$-symmetric with $\eta_i\to0$, then $B_s(x)$ is $k$-symmetric. 
\par
More precisely, for any fixed sequence of $(k,\eta_i)$-symmetry models ($B_s(0,z_{*,i})\subset \mathbb{R}^k\times C(Z_i),\tilde{u}_i)$ with their corresponding Borel pGH approximations $F_i$, after modifying $\tilde{u}_i$ and passing to a subsequence, $\mathbb{R}^k\times C(Z_i)$ converge in the pmGH sense to an $\mathsf{ncRCD}(0,n)$ cone $\mathbb{R}^k\times C(Z_\infty)$, and $\tilde{u}_i$ converge locally uniformly to an $L$-Lipschitz, $1$-homogeneous, $\mathbb{R}^k$-invariant $\tilde{u}_\infty$ on $\mathbb{R}^k\times C(Z_\infty)$.
In addition, $F_i$ converge
to an isometry $F:B_s(0,z_{*,\infty})\subset \mathbb{R}^k\times C(Z_\infty)\to B_s(x)$, with $u\circ F=\tilde{u}_\infty$.
In particular, the corresponding approximate symmetry planes $F_i(\mathbb{R}^k\times \{z_{*,i}\})$
converge to $\mathcal L^k:=F(\mathbb R^k\times\{z_{*,\infty}\})$.
\item
More generally, if $B_s(x_i)$ is $(k,\eta)$-symmetric for every $i$,
then $B_s(x)$ is $(k,2\eta)$-symmetric.
\end{enumerate}

\item[(iv)] \textit{(Approximation through smooth manifolds)}
Suppose $u$ is a limit solution as in \Cref{subsec:lim_sol}.
Let $B_{4}(x)\subset \Omega$,
$q\in\mathcal{S}^k_{\eta,r}(u)\cap B_1(x)$,
and $u_i$ be the approximating minimizing solutions converging to $u$ along $B_4(x_i)\to B_4(x)$.
Then there exist $q_i\in \mathcal{S}^k_{\eta/2,r}(u_i)\cap B_1(x_i)$ such that $q_i\to q$.
\end{enumerate}
\end{proposition}

\begin{proof}

Item (i) is immediate. The proof of item (ii) is standard, but the stratification involves functions defined on varying ambient spaces, which is a new feature in the present setting. Consequently, one has to keep track simultaneously of the convergence of the ambient spaces and of the functions defined on them. In particular, the argument relies on the behavior under convergence of the notion of closeness underlying the stratification. It is therefore useful to spell out this point.
\par
Let $x \in \mathcal{S}^k_{\eta}(u)$.
We claim that no blow-up map at $x$ can be $(k+1,\eta/2)$-symmetric, and in particular $(k+1)$-symmetric.
Suppose otherwise that there is some blow-up $(B_1(p), u_x)$ at $x$ that is $(k+1,\eta/2)$-symmetric.
    Then there exists $s_i\to 0$ such that 
    \begin{align}
        B_1^{\hat{d_i}}(x) &\xrightarrow[]{\mathrm{pGH}} B_1(p)\subset T_x X, \\
        \hat{u}_i&\to u_x \quad \text{uniformly on $B_1^{\hat{d_i}}(x)$},
    \end{align}
    where $\hat{d_i}=s_i^{-1}d$, $\hat{u}_i=s_i^{-1}u$, and 
\begin{enumerate}
        \item  There is a Borel $\eta/2$-pGH approximation map $F:B_1(0,z_*)\to B_1(p)$, where $B_1(0,z_*)\subset \mathbb{R}^{k+1}\times C(Z)$ is a ball in an $\mathsf{ncRCD}(0,n)$ cone.
        \item $\int_{B_1(0,z_*)}|u_x\circ F-\tilde{u}|^2 d\mathcal{H}^n<\eta/2$ for some $1$-homogeneous and $\mathbb{R}^{k+1}$-invariant $\tilde{u}$.
    \end{enumerate}
The uniform convergence follows from the fact that the rescaling
preserves the Lipschitz bound. In particular, $u_x$ is Lipschitz continuous. Moreover, by volume convergence, the
above convergence is in fact pmGH with respect to the corresponding
Hausdorff measures.
\par
Let $G_i:B_1(p)\to B^{\hat{d_i}}_1(x)$ be $\varepsilon_i$-pGH approximations with $\varepsilon_i\to 0$ such that $\abs{\hat{u}_i \circ G_i - u_x}\to 0$ uniformly on $B_1(p)$.
Then $G_i\circ F:B_1(0,z_*)\to B^{\hat{d}_i}_1(x)$ is an $\eta$-pGH for $i$ large enough. At the function level, we have
\begin{equation}
    \begin{split}
                \left(\int_{B_{1}(0,z_*)}|\hat{u}_i\circ G_i\circ F-\tilde{u}|^2d\mathcal H^n\right)^{1/2}
        &\leq \left(\int_{B_{1}(0,z_*)}|\hat{u}_i\circ G_i\circ F- u_x\circ F|^2d\mathcal H^n\right)^{1/2}\\
        &\quad +\left(\int_{B_{1}(0,z_*)}|u_x\circ F-\tilde{u}|^2d\mathcal H^n\right)^{1/2}.
    \end{split}
\end{equation}
The first term goes to zero due to the uniform convergence, and the second term is less than $\sqrt{\eta/2}$.
These imply that for $i$-large enough, $B_{s_i}(x)$ is $(k+1,\eta)$-symmetric, giving a contradiction. This shows that $\mathcal{S}^k_{\eta} \subset \{\text{no blow-up is $(k+1)$-symmetric}\}$.
\par
In the other direction, suppose no blow-up at $x$ is $(k+1)$-symmetric.
We have to prove that there is some $\eta > 0$ for which $x \in \mathcal{S}^k_{\eta}$. 
Assume this is not the case. Then for any $\eta>0$, there is some $s(\eta)\in (0,1)$ such that $B_s(x)$ is $(k+1,\eta)$-symmetric.
Take a sequence $\eta_i\to 0$ and let $s_i=s(\eta_i)$, $\hat{d_i}=s_i^{-1}d$, and $\hat{u}_i=s_i^{-1}u$.
Then we have 
\begin{enumerate}
    \item There are Borel $\eta_i$-pGH approximation maps $F_i:B_1(0,z_{*,i})\to B^{\hat{d}_i}_1(x)$, where $B_1(0,z_{*,i})\subset \mathbb{R}^{k+1}\times C(Z_i)$ is a ball in an $\mathsf{ncRCD}(0,n)$ cone.
    \item There exist $1$-homogeneous and $\mathbb{R}^{k+1}$-invariant functions $\tilde v_i$ on $\mathbb{R}^{k+1}\times C(Z_i)$ such that $\int_{B_1(0,z_{*,i})}|\hat u_i\circ F_i-\tilde v_i|^2\,d\mathcal H^n<\eta_i$.
\end{enumerate}
Notice that by \Cref{lem:almost_lip_sym_model}, we may assume that the functions $\tilde v_i$ are $(C(n)L,\eta_i)$-almost Lipschitz on $B_1(0,z_{*,i})$. 

\noindent \textsc{Claim.} \emph{After passing to a subsequence, there exist an $\mathsf{ncRCD}(0,n)$ cone $\mathbb{R}^{k+1}\times C(Z_\infty)$ and a function $\tilde v_\infty$ on $\mathbb{R}^{k+1}\times C(Z_\infty)$, such that}
\begin{align}
    &\mathbb{R}^{k+1}\times C(Z_i) \xrightarrow[]{\mathrm{pmGH}}  \mathbb{R}^{k+1}\times C(Z_\infty);\\
    &\tilde v_i \to \tilde v_\infty \ \text{uniformly on compact subsets}, \text{along this pmGH convergence.} \label{eq:xz}
\end{align}
\emph{Moreover, $\tilde v_\infty$ is Lipschitz continuous, $1$-homogeneous and $\mathbb{R}^{k+1}$-invariant.}
\par
\noindent\textit{Proof of Claim}:
By \Cref{lem:ncRCD_cone_struct} and Gromov precompactness for $\mathsf{RCD}(K,n)$ spaces, up to a subsequence, the compact links $Z_i$ GH-converge to some $Z_\infty$, and hence the corresponding pointed cones $Y_i:=\mathbb{R}^{k+1}\times C(Z_i)$ pGH-converge to the pointed cone $Y_\infty:=\mathbb{R}^{k+1}\times C(Z_\infty)$. We claim that this sequence is noncollapsing.
By Item 1.\@ above, $d_{\mathrm{GH}}\bigl(B_1^{\hat d_i}(x),B_1^{Y_i}(0,z_{*,i})\bigr)\to 0$.
Moreover, the ambient spaces and the model cones belong, for all sufficiently large $i$, to a common class $\mathsf{RCD}(K,n)$ for some fixed $K\in\mathbb{R}$. Hence, by the continuity of the $n$-dimensional Hausdorff measure on $\mathsf{RCD}(K,n)$ balls (\cite[Theorem 1.3]{DPG2018}), we have
$$\abs{\mathcal{H}^n\bigl(B_1^{Y_i}(0,z_{*,i})\bigr) - \mathcal{H}^n\bigl(B_1^{\hat d_i}(x)\bigr)} \to 0.$$
Since $\hat{d}_i$ is a rescaling of $d$, we have $\mathcal{H}^n\bigl(B_1^{\hat d_i}(x)\bigr)\geq \mathrm{v}$ for some $\mathrm{v}>0$ and therefore $\mathcal{H}^n\bigl(B_1^{Y_i}(0,z_{*,i})\bigr)\geq \mathrm{v}/2$ for all sufficiently large $i$. By volume convergence (\cite[Theorem 1.2]{DPG2018}), we have
$$
(Y_i,d_i,\mathcal{H}^n,(0,z_{*,i}))
\xrightarrow{\mathrm{pmGH}}
(Y_\infty,d_\infty,\mathcal{H}^n,(0,z_{*,\infty})),
$$
and $Y_\infty$ is an $\mathsf{ncRCD}(0,n)$ cone. This proves the claim at the space level.
\par
Moreover, choosing measured GH approximations $\sigma_i:Z_i\to Z_\infty$ realizing the convergence of the links, we may realize the above pmGH convergence of the model cones by the induced maps $S_i:\mathbb R^{k+1}\times C(Z_i)\to\mathbb R^{k+1}\times C(Z_\infty)$ with $S_i(a,(r,z)):=(a,(r,\sigma_i(z)))$. On every bounded ball, these maps are pGH approximations with error tending to zero, and they are equivariant with respect to cone dilations and translations along the $\mathbb R^{k+1}$-factor. Moreover, they are measured approximations, namely, $\int_{Y_i} \varphi\circ S_i\,d\mathcal H^n \to \int_{Y_\infty} \varphi\,d\mathcal H^n$ for every $\varphi\in C_c(Y_\infty)$.

For the function level, for every $R>0$, by the $1$-homogeneity, we have that $\tilde v_i$ are $(C(n)L, R\eta_i)$-almost Lipschitz on $B_R(0,z_{*,i})$.
Since $\tilde{v}_i(0,z_{*,i})=0$, they are also uniformly bounded on $B_R(0,z_{*,i})$.
Hence, by a diagonal argument with \Cref{lem:AAarg}, after passing to a subsequence, $\tilde v_i$ converge uniformly to a $C(n)L$-Lipschitz function $\tilde v_\infty$ along $B_R(0,z_{*,i})\to B_R(0,z_{*,\infty})$ with respect to the above pmGH convergence, for any $R>0$.
More precisely, let $\tilde S_i:B_R(0,z_{*,\infty})\to B_R(0,z_{*,i})$ be approximate inverses of the maps $S_i$ above, chosen equivariantly with respect to cone dilations and translations along the $\mathbb R^{k+1}$-factor. Then the functions $\tilde v_i\circ\tilde S_i$ converge uniformly to $\tilde v_\infty$ over $B_R(0,z_{*,i})$.
Since each $\tilde v_i\circ\tilde S_i$ is $1$-homogeneous and $\mathbb R^{k+1}$-invariant, we conclude that $\tilde v_\infty$ has the same properties.\footnote{Uniform convergence is formally expressed using the maps $S_i$, whereas here we work with approximate inverses $\tilde S_i$. This is harmless because $S_i$ and $\tilde S_i$ are approximate inverses and the family $\tilde v_i$ is asymptotically equicontinuous. Indeed, $d_i\bigl(\tilde S_i(S_i(x_i)),x_i\bigr)\to0$, and hence $\bigl|\tilde v_i(\tilde S_i(S_i(x_i)))-\tilde v_i(x_i)\bigr|\to0$. Thus the two ways of expressing the uniform convergence are equivalent.} This proves the claim.
\hfill {$\square\square$}
\par
Now note that by Bishop-Gromov volume comparison, the boundary of any ball has zero measure. Therefore, applying \Cref{lem:uniform_implies_L2}, it follows that $\tilde v_i$ converges in $L^2(B_1(0,z_{*,i}))$ to $\tilde v_\infty$. As $\norm{\hat{u}_i\circ F_i - \tilde{v}_i}_{L^2(B_1(0,z_{*,i}))}\to 0$, it follows that $\hat{u}_i\circ F_i \to \tilde{v}_\infty$ in $L^2(B_1(0,z_{*,i}))$.
With the Arzel\`a--Ascoli-type precompactness from \Cref{lem:AAarg}, we then get $\hat{u}_i\circ F_i$ to $\tilde{v}_\infty$ uniformly.
\par
Finally, we have $F_i \circ \tilde{S}_i:B_1(0,z_{*,\infty})\to B_1^{\hat{d_i}}(x)$ $(\eta_i+\varepsilon_i)$-pGH such that $\abs{\hat{u}_i \circ F_i \circ \tilde{S}_i - \tilde{v}_\infty}$ converges uniformly to zero. Recall that $\hat{d_i} = s_i^{-1}d$. Now the proof continues as usual: if $\liminf_i s_i=0$, then, up to a subsequence, this is exactly saying that a blow-up of $u$ at $x$ is $(k+1)$-symmetric, and hence it is a contradiction. If this is not the case, up to a subsequence we can assume $s_i \to s_\infty$. In this case, the whole ball $B_{s_\infty/2}(x)$ is $(k+1)$-symmetric, and in particular any blow-up will be.
\par
To prove Item (iii)-(a), let $H_i:B_s(x_i) \to B_s(x)$ be an $\varepsilon_i$-pGH approximation such that $\abs{u\circ H_i - u_i}$ converges uniformly to zero.
Note that by the Proof of Claim above, we can assume the existence of a sequence of $s\eta_i$-pGH approximations $T_i: B_s(0,z_*)\subset \mathbb{R}^{k}\times C(Z) \to B_s(x_i)$ and an $L$-Lipschitz, $1$-homogeneous, $\mathbb{R}^k$-invariant $\tilde{v}$ on $\mathbb{R}^k\times C(Z)$, such that $\abs{u_i \circ T_i - \tilde{v}}$ converge uniformly to zero on $B_s(0,z_*)$. After passing to a subsequence, the maps $H_i\circ T_i$ converge uniformly to an isometry $F:B_s(0,z_*)\to B_s(x)$. Moreover, the uniform convergence of the functions gives $u\circ F=\tilde v$. Consequently, the corresponding approximate symmetry planes converge to $\mathcal L^k=F(\mathbb R^k\times\{z_*\})$.
The proof of Item (iii)-(b) is easier, as we don't need to take sublimits on the symmetry models $\tilde{v}_i$.

For item (iv), by the compactness theorem and the convergence of the free boundaries in \Cref{thm:compactness}, there exist points $q_i\in \partial\{u_i>0\}\cap B_1(p_i)$ such that $q_i\to q$. Suppose, by contradiction, that $q_i\notin S^k_{\eta/2,r}(u_i)$ for infinitely many $i$.
Then there exist $s_i\in[r,1]$ such that $B_{s_i}(q_i)$ is $(k+1,\eta/2)$-symmetric. After passing to a subsequence, we may assume that $s_i\to s_\infty\in[r,1]$.
Rescaling the balls $B_{s_i}(q_i)$ to unit size and applying Item (iii)-(b), we conclude that the corresponding ball centered at $q$ is $(k+1,\eta)$-symmetric.
This contradicts $q\in S^k_{\eta,r}(u)$.
\end{proof}

The next proposition shows that the symmetry of a minimizer $u$ translates into symmetry of its free boundary. \begin{proposition}[Almost symmetry of $u$ gives almost symmetry of the free boundary] \label{prop:almostsymualmostfree} Given $\rv,L,\Lambda>0$ and $\alpha\in(0,1)$, for every $\eta'>0$ there exists $\eta=\eta(\eta',n,r^2K,\mathrm{v},L,\Lambda,\alpha)>0$ such that the following holds.
Let $X$ be an $\mathsf{ncRCD}(K,n)$ space and $u:B_{2r}(x)\subset X\to[0,\infty)$ an $L$-Lipschitz minimizing solution of $E_Q$ with $x\in\pt\{u>0\}$. Assume $\mathcal H^n(B_{2r}(x))\geq \mathrm{v}r^n$, and $\|Q\|_{C^0}+\|Q^{-1}\|_{C^0}+r^\alpha[Q]_\alpha\leq\Lambda$ on $B_{2r}(x)$. If $B_r(x)$ is $(k,\eta)$-symmetric, then $\partial\{u>0\}\cap B_{r/2}(x)$ is $(k,\eta')$-symmetric, in the sense that there exists $(n-k-1)$-dimensional metric cone $C(Z)$ such that
\begin{equation}
d_{\mathrm{GH}}\left(
\partial\{u>0\}\cap B_{r/2}(x),
B_{r/2}(0,z_*)
\right)
\leq \frac r2\,\eta',
\end{equation}
where $B_{r/2}(0,z_*)\subset\mathbb R^k\times C(Z)$.
\end{proposition}
\begin{proof}
By rescaling, we may assume $r=1$. Suppose the statement is false. Then there exist $\eta'_0>0$, a sequence $\eta_i\to0$, $\mathsf{ncRCD}(K,n)$ spaces $X_i$, and $L$-Lipschitz minimizing solutions $u_i:B_2(x_i)\subset X_i\to[0,\infty)$ of $E_{Q_i}$, satisfying the same uniform noncollapsed and coefficient bounds, such that $B_1(x_i)$ is $(k,\eta_i)$-symmetric, but
\begin{equation}
d_{\mathrm{GH}}\left(
\partial\{u_i>0\}\cap B_{1/2}(x_i),
B_{1/2}(0,z_*)
\right)>\eta'_0
\end{equation}
for every $(n-k-1)$-dimensional cone $C(Z)$. After passing to a subsequence, we have $(X_i,x_i)\xrightarrow[]{\mathrm{pmGH}} (X_\infty, x_\infty)$, $B_2(x_i)\to B_2(x_\infty)$, and $Q_i\to Q_\infty$ locally uniformly. By \Cref{thm:compactness}, $u_i\to u_\infty$ locally uniformly, where $u_\infty$ minimizes $E_{Q_\infty}$ on $B_2(x_\infty)$, and the free boundaries converge with respect to the pmGH convergence. Since $B_1(x_i)$ is $(k,\eta_i)$-symmetric and $\eta_i\to0$, \Cref{prop:basicpropstrata}(iii) implies that $B_1(x_\infty)$ is $k$-symmetric. Hence, for some $\mathsf{ncRCD}(0,n)$ cone $\mathbb R^k\times C(Y)$ and isometry $F$, the function $u_\infty\circ F$ is $1$-homogeneous and $\mathbb R^k$-invariant. Therefore its free boundary is a cone invariant under translations along the $\mathbb R^k$-factor, and thus
\begin{equation}
\partial\{u_\infty>0\}\cap B_{1/2}(x_\infty)
=
F\left(B_{1/2}(0,z_*)\right)
\end{equation}
for some $(n-k-1)$-dimensional metric cone $C(Z)$. The convergence of the free
boundaries gives a contradiction.
\end{proof}

\subsection{Regular and Singular Points, and the Reifenberg Property}\label{subsec:reifenberg}

In this subsection, we introduce regular and singular points and
establish a Reifenberg property: sufficiently regular balls remain
almost regular at all smaller scales around nearby free boundary
points.
This property is claimed in \cite{CZZ22} for the regular set under their definition \cite[Definition 8.3]{CZZ22}.
For the sake of completeness, we provide here an independent argument, inspired by the ideas of \cite{mondino_weak_2025}.
In Euclidean space, the Reifenberg property typically follows from the fact that almost-minimal pinching of a monotone quantity implies closeness to the model of maximal symmetry; the monotonicity also allows one to propagate this pinching across infinitely many scales.
This mechanism is not available in the nonsmooth setting here due to the lack of bounded monotonicity. The key observation in \cite{mondino_weak_2025} is that the Reifenberg property can still be recovered provided one has, in addition, a density gap for homogeneous solutions in $\mathbb{R}^n$.  
\par
We begin with the relevant definitions. By \Cref{prop:basicpropstrata} and \Cref{cor:blowupmaps}, it follows that $\mathcal{S}^{n-1}$ is the whole free boundary (also see \Cref{prop:structurestrata} below).
Now inside $\mathcal{S}^{n-1}$, we define the regular points and singular points as follows.
\begin{definition}[Regular and singular points, and $\delta$-regular balls]
We say $x \in \partial\{u>0\}\cap \Omega$ is a \emph{regular point} if both
\begin{enumerate}[nosep]
    \item The tangent cone of the space at $x$ is $\mathbb{R}^n$; 
    \item Every blow-up $u_x$ of $u$ at $x$, is $(n-1)$-symmetric and hence, up to rotation, a half-plane solution $u_x(y)=Q(x)y_n^+$ (see the proof of \cite[Corollary~7.2]{CZZ22} and the references therein).
\end{enumerate}
We say $x$ is a \emph{singular point} if it is not a regular point. 
We denote by $\text{Sing}(u)$ the set of singular points. 
We also say $B_r(x)\subset \Omega$ is $\delta$-regular if it is $(n-1,\delta)$-symmetric, with a $\delta r$-GH approximation from $B_r(0^n)\subset \mathbb{R}^n$.
\end{definition}

\par
To prove the Reifenberg property, we use the following Weiss-type formula on $X$ from \cite[(8.1)]{CZZ22}:
\[
W_u(x,s):= \frac{1}{s^n}\int_{B_s(x)} |\nabla u|^2 + Q^2(x)1_{u>0}\dd{}\mathcal{H}^n - \frac{1}{s^{n+1}}\int_{\partial B_s(x)} u^2 \dd|D1_{B_s(x)}|.
\]
Here $|D1_{B_s(x)}|$ is the perimeter measure of the ball (defined for almost every $s$; see e.g.\@ \cite{Miranda2003,Ambrosio2019rigidity}).
Under the assumptions of \Cref{thm:compactness}, $W_{u_i}(x_i,s)\to W_u(x,s)$ for almost every $s$, provided $x_i\to x$. Indeed, $Q_i(x_i)\to Q(x)$, while the remaining terms converge by \Cref{thm:compactness} together with the continuity of the perimeter measure; see \cite[Lemma~2.16 and Lemma~8.2]{CZZ22}.
Recall that by \Cref{thm:orginal_weiss}, the Weiss formula is monotone when $X=\mathbb{R}^n$ and $Q$ is constant.

\begin{lemma}\label{lem:density_regular}
Let $K\in\mathbb R$, and let $u$ be an $L$-Lipschitz minimizing
solution of $E_Q$ on $B_2(x)$, with $x\in\pt\{u>0\}$,
in an $\mathsf{ncRCD}(K,n)$ space. Assume
$\|Q\|_{C^0}+\|Q^{-1}\|_{C^0}+[Q]_\alpha\leq\Lambda$ on $B_2(x)$.
For every $\varepsilon>0$, there exists
$\delta=\delta(\varepsilon,n,K,L,\Lambda,\alpha)>0$ such that, if
$B_2(x)$ is $\delta$-regular, then there exists $s\in(1,2)$ for which
$B_s(x)$ has finite perimeter and
\begin{equation}
\frac12-\varepsilon
\leq
\frac{W_u(x,s)}{Q^2(x)\omega_n}
\leq
\frac12+\varepsilon.
\label{eq:densityest}
\end{equation}
\end{lemma}

\begin{proof}
Arguing by contradiction, there exist $\delta_i\to0$ and minimizing solutions $u_i$ of $E_{Q_i}$ on $B_2(x_i)\subset (X_i,x_i)$ satisfying the uniform assumptions above,
such that $B_2(x_i)$ is $\delta_i$-regular, while
\eqref{eq:densityest} fails for every $s\in(1,2)$ for which the Weiss
density is defined. After passing to a subsequence, $(X_i, x_i)$ converge in pmGH to some $\mathsf{ncRCD}(K,n)$ $(X_\infty,x_\infty)$ and $B_2(x_i)\to B_2(0)\subset\mathbb R^n$. 
Moreover, $Q_i\to Q_\infty$ locally uniformly and, by
\Cref{thm:compactness}, $u_i\to u_\infty$ locally uniformly. Since
$\delta_i\to0$, by \Cref{prop:basicpropstrata} the limit $u_\infty$ is $(n-1)$-symmetric on $B_2(0^n)$. 
\par
Blowing up $u_\infty$ at $0$, we obtain a global $(n-1)$-symmetric minimizer $u_0$ of $E_{Q_\infty(0)}$ on $\mathbb R^n$. Since $u_\infty$ is already $1$-homogeneous on $B_2(0)$, $u_0$ agrees with $u_\infty$ on this ball. Thus, up to rotation, $u_0(y)=Q_\infty(0)y_n^+$, and the same holds for $u_\infty$ on $B_2(0)$.
\par
By the continuity of the Weiss density under the above convergence \cite[Lemma 8.2]{CZZ22},
for almost every $s\in(1,2)$ we have
\[
W_{u_i}(x_i,s)\to W_{u_\infty}(0,s)
=\frac12Q_\infty^2(0)\omega_n.
\]
Since $Q_i(x_i)\to Q_\infty(0)$, for any such $s$ we obtain
\[
\frac{W_{u_i}(x_i,s)}{Q_i^2(x_i)\omega_n}\to\frac12,
\]
contradicting the choice of the sequence.
\end{proof}

The following density gap result is well known in Euclidean space.
\begin{proposition}[Density gap in Euclidean space
(\cite{AC81,CK98,Wei99})]\label{prop:densitygap}
There exists $\varepsilon_n>0$ such that the following holds. Let
$Q_0>0$ be constant and let $u$ be a
$1$-homogeneous global minimizer of $E_{Q_0}$ on $\mathbb{R}^n$, with
$0\in\partial\{u>0\}$. Then
\[
W_u(0,1)\geq Q_0^2\frac{\omega_n}{2}.
\]
Moreover, if
\[
W_u(0,1)
<
Q_0^2\frac{\omega_n}{2}(1+\varepsilon_n),
\]
then $u$ is a regular half-plane solution.
In other words, there is a density gap between the half-plane solution and all the other $1$-homogeneous global minimizers of the same constant-coefficient problem on $\mathbb{R}^n$.
\end{proposition}

We now prove the Reifenberg property of the free boundary.

\begin{proposition}[Reifenberg property]\label{prop:reifenbergprop}
Given $\varepsilon,L,\Lambda>0$ and $\alpha\in(0,1)$, there exists
$\delta=\delta(n,L,\Lambda,\alpha,\varepsilon)>0$ with the following
property.
Let $u$ be an $L$-Lipschitz minimizing solution of $E_Q$ on $B_4(x)$, with $x\in\partial\{u>0\}$, in an $\mathsf{ncRCD}(-\delta,n)$ space $(X,d,\mathcal{H}^n)$.
Assume that
$\|Q\|_{C^0}+\|Q^{-1}\|_{C^0}+[Q]_\alpha\leq\Lambda$ on $B_4(x)$, and $B_2(x)$ is $\delta$-regular.
Then every $B_r(y)\subset B_2(x)$,
with $y\in\partial\{u>0\}\cap B_1(x)$ and $r\in(0,1)$, is
$\varepsilon$-regular.
\end{proposition}
\begin{proof}
This property is given in \cite{CZZ22}. As mentioned above, we provide
here a proof based on the arguments of \cite{mondino_weak_2025}.
It is enough to prove the result for $\varepsilon$ sufficiently small. Suppose the statement is false. Then there exist $\delta_j\to0$, pointed $\mathsf{ncRCD}(-\delta_j,n)$ spaces $(X_j,d_j,\mathcal H^n,x_j)$, coefficients $Q_j$, and minimizing solutions $u_j$ satisfying the uniform assumptions above, such that $B_2(x_j)$ is $\delta_j$-regular, but there exist $y_j\in\partial\{u_j>0\}\cap B_1(x_j)$ and $r_j\in(0,1)$ such that
\begin{enumerate}[topsep=0pt]
    \item $B_r(y_j)$ is $\varepsilon$-regular for every $r_j<r<1$;
    \item $B_{r_j/2}(y_j)$ is not $\varepsilon$-regular.
\end{enumerate}
\par
Clearly we have $r_j\to 0$. Rescaling by $(u,d,Q_j)\mapsto (r_j^{-1}u,\hat{d}_j:=r_j^{-1}d_j,Q_j)$, and not changing the notation, we obtain balls $B_{1/2}(y_j)$ that are not $\varepsilon$-regular, while $B_r(y_j)$ is $\varepsilon$-regular for all $1<r<r_j^{-1}\to\infty$. 
We also have that $B_{r_j^{-1}}(y_j)\subset B_{2r_j^{-1}}(x_j)$ is $2r_j^{-1}\delta_j$-GH close to a ball in $\mathbb R^n$. Since $\delta_j\to0$ and $r_j^{-1}\to\infty$, the Reifenberg theorem for noncollapsed spaces (see \cite{DPG2018} and \cite{KapovitchMondino2021} for the nonsmooth case) implies that $(X_j, \hat{d}_j, y_j)$ converge in pmGH to $(\mathbb{R}^n,0)$.
Regarding the coefficients, with respect to the rescaled metric, for every fixed $R<3r_j^{-1}$, we have $[Q_j]_{\alpha,B_R(y_j)}\leq \Lambda r_j^\alpha
\to0$. Therefore, after passing to a subsequence, we may assume that $Q_j$ converge to some constant $Q_\infty\in[\Lambda^{-1},\Lambda]$ locally uniformly.
Hence, by \Cref{thm:compactness}, after
passing to a further subsequence, $u_j$ converges locally uniformly to
a global minimizer $u_\infty:\mathbb R^n\to[0,\infty)$ of the constant-coefficient problem $E_{Q_\infty}$, with $0\in\partial\{u_\infty>0\}$.

By \Cref{lem:density_regular}, if $\varepsilon$ is sufficiently small, there exists
$\varepsilon'=\varepsilon'(\varepsilon\mid n,L,\Lambda,\alpha)$, with
$\varepsilon'\to0$ as $\varepsilon\to0$, such that
\[
\frac12-\varepsilon'
\leq
\frac{W_{u_\infty}(0,r)}{Q_\infty^2\omega_n}
\leq
\frac12+\varepsilon'
\qquad\text{for every }r>2.
\]
Indeed, applying \Cref{lem:density_regular} to the rescalings of $u_\infty$ at every scale $R\geq1$ yields a good radius $s_R\in(R,2R)$. The claimed bounds then follow from the monotonicity of the Euclidean Weiss functional. Consider now a blow-down of $u_\infty$ at the origin, denoted by $v:\mathbb R^n\to[0,\infty)$. Since $Q_\infty$ is constant, by \Cref{thm:orginal_weiss} $v$ is a $1$-homogeneous global minimizer of $E_{Q_\infty}$, and
\[
W_v(0,1)
=
\lim_{r\to\infty}W_{u_\infty}(0,r)
\leq
Q_\infty^2\omega_n\left(\frac12+\varepsilon'\right).
\]
Choosing $\varepsilon$ sufficiently small so that $2\varepsilon'<\varepsilon_n$,
where $\varepsilon_n$ is the constant in \Cref{prop:densitygap}, we have
\[
W_v(0,1)
<
Q_\infty^2\frac{\omega_n}{2}(1+\varepsilon_n).
\]
Hence, by \Cref{prop:densitygap}, $v$ is a regular half-plane solution.
Consequently,
\[
\lim_{r\to\infty}W_{u_\infty}(0,r)
=W_v(0,1)
=Q_\infty^2\frac{\omega_n}{2}.
\]
On the other hand, let $u_{\infty,0}$ be a blow-up of $u_\infty$ at $0$.
Then again by \Cref{thm:orginal_weiss}, $u_{\infty,0}$ is a $1$-homogeneous global minimizer of
$E_{Q_\infty}$.
Hence by \Cref{prop:densitygap}, we have
\[
Q_\infty^2\frac{\omega_n}{2}
\leq W_{u_{\infty,0}}(0,1)
=
\lim_{r\to0}W_{u_\infty}(0,r).
\]
Since $W_{u_\infty}(0,\cdot)$ is monotone, we now have
\[
Q_\infty^2\frac{\omega_n}{2}
\leq
\lim_{r\to0}W_{u_\infty}(0,r)
\leq
\lim_{r\to\infty}W_{u_\infty}(0,r)
=
Q_\infty^2\frac{\omega_n}{2}.
\]
Thus $W_{u_\infty}(0,\cdot)$ is constant.
By \Cref{thm:orginal_weiss} with constant $Q_\infty$, $u_\infty$ is $1$-homogeneous, and by
\Cref{prop:densitygap} again, it is a regular half-plane solution, with slope
$Q_\infty$. Therefore, by the convergence above,
$B_{1/2}(y_j)$ is $\varepsilon$-regular for all sufficiently large $j$,
contradicting item~2.
\end{proof}

\subsection{Structure and Regularity of the Strata}\label{subsec:strata_regularity}

We now study the structure of the strata and apply the Reifenberg property to obtain regularity of the free boundary. We first show that, when the ambient space has no boundary, $\operatorname{Sing}(u)\subset\mathcal S^{n-3} =\mathcal S^{n-2}$. We then prove that almost $(n-2)$-symmetric balls are almost regular. Combined with the Reifenberg property, this yields a Hölder manifold structure on the complement of $\mathcal S^{n-3}_\delta$ in the free boundary, for some $\delta>0$ depending on the local data.

\begin{proposition}\label{prop:structurestrata}
Let $(X,d,\mathcal{H}^n)$ be an $\mathsf{ncRCD}(K,n)$ space without boundary, and $u:\Omega\subset X\to [0,\infty)$ a minimizer of $E_Q$.
The stratification satisfies
\begin{enumerate}[nosep]
    \item $\cS^{n-1} = \cS^n = \partial\{u>0\}\cap\Omega$;
    \item $\mathrm{Sing}(u) \subset \cS^{n-3}=\cS^{n-2}$.
\end{enumerate}
\end{proposition}
\begin{proof}
As commented above, the first statement follows from \Cref{prop:basicpropstrata} and \Cref{cor:blowupmaps}, since blowups cannot be $n$-symmetric, namely, trivial constant functions.
\par
 Here we prove the second statement. 
Since having an empty boundary is preserved under pmGH convergence for $\mathsf{ncRCD}$ spaces \cite[Theorem~1.6]{BrueNaberSemola2022}, every tangent cone of $X$ also has empty boundary. 
Now, suppose $x\in \mathcal{S}^{n-1}\backslash\mathcal S^{n-2}$. By \Cref{prop:basicpropstrata}, there exists
an $(n-1)$-symmetric blow-up $u_x$ at $x$.
The corresponding tangent cone at $x$ splits off an $\mathbb R^{n-1}$-factor. Since this cone has no boundary, the remaining one-dimensional factor is a line, and hence the tangent cone is $\mathbb R^n$.
Consequently, $u_x$ is an $(n-1)$-symmetric global minimizer of $E_{Q(x)}$ on $\mathbb{R}^n$ and hence, up to rotation, the regular half-plane solution.
\par
We next show that every blow-up of $u$ at $x$ is $(n-1)$-symmetric.
Fix $\varepsilon>0$, and let $\delta>0$ be given by \Cref{prop:reifenbergprop}, using fixed local bounds for $u$ and $Q$. By the convergence to the half-plane blow-up $u_x$, we can choose $r>0$ sufficiently small along the sequence defining this blow-up so that, after rescaling $(d,u)$ to $(r^{-1}d,r^{-1}u)$, the ball $B_2(x)$ is $\delta$ regular and all the hypotheses of \Cref{prop:reifenbergprop} are satisfied. In particular, we choose $r$ so that $B_{4r}(x)\subset\Omega$ and $r^2K\geq-\delta$. The proposition then implies that, in the original metric, $B_s(x)$ is $\varepsilon$-regular for every $s\in(0,r)$. Since $\varepsilon>0$ is arbitrary, every tangent cone at $x$ is $\mathbb R^n$ and every blow-up of $u$ at $x$ is $(n-1)$-symmetric. Thus $x$ is a regular point.
\par
It remains to show that
$\mathcal S^{n-2}\setminus\mathcal S^{n-3}$ is empty. Suppose, by contradiction, that there is $x\in\mathcal S^{n-2}\setminus\mathcal S^{n-3}$. Then there exists an $(n-2)$-symmetric blow-up $u_x$ of $u$ at $x$. In particular, $u_x$ is a $1$-homogeneous global minimizer of $E_{Q(x)}$ on a tangent cone
\[
C_xX\simeq\mathbb R^{n-2}\times C(Z),
\]
and it is invariant along the $\mathbb R^{n-2}$-factor. Moreover, $Z$
is a one-dimensional $\mathsf{ncRCD}$ space with diameter at most
$\pi$ (\cite[Corollary~1.3]{Ketterer2015}). Since $X$ has no boundary, $Z$ is a circle $S_a$ of radius
$a\leq1$ \cite{kitabeppu2016characterizationlowdimensionalrcdkn}. As in \cite[Lemma 8.15]{CZZ22}, compare also \cite{Wei99}, the restriction
$\bar u:=u_x|_{\{0\}\times C(S_a)}$ is a $1$-homogeneous global
minimizer of $E_{Q(x)}$ on $C(S_a)$.
We distinguish the following two cases.
\begin{itemize}[nosep]
    \item If $a=1$, then $C(S_a)\simeq\mathbb R^2$. By \cite{AC81}, $\bar u$
is a regular half-plane solution, with slope $Q(x)$ (see also \cite[Proposition 9.13]{Vel23}). Hence $\bar u$,
and therefore $u_x$, is invariant along one additional Euclidean
direction. Thus $u_x$ is $(n-1)$-symmetric, contradicting
$x\in\mathcal S^{n-2}$.
    \item If $a<1$, then by \cite{AL15} the vertex of $C(S_a)$ cannot belong to
the free boundary of $\bar u$. This is a contradiction to \Cref{cor:blowupmaps}.
\end{itemize}
We conclude that
$\mathcal S^{n-2}\setminus\mathcal S^{n-3}=\varnothing$ and hence $\mathcal{S}^{n-2}=\mathcal{S}^{n-3}$.
Together with the paragraph above, this shows that every point outside
$\mathcal S^{n-3}$ is regular, and hence $\operatorname{Sing}(u)\subset\mathcal S^{n-3}$.
\end{proof}

The following proposition shows that almost $(n-2)$-symmetric balls are almost regular.

\begin{proposition}\label{prop:almostregularpoints}
Given $\varepsilon,\mathrm{v},L,\Lambda>0$ and $\alpha\in(0,1)$, there exists
$\delta=\delta(n,r^2K,\mathrm{v},L,\Lambda,\alpha,\varepsilon)>0$ with the following property. Let $(X,d,\mathcal H^n)$ be an $\mathsf{ncRCD}(K,n)$ space without boundary, and let $u:B_{2r}(x)\to[0,\infty)$ be an $L$-Lipschitz minimizing solution of $E_Q$, with $x\in\partial\{u>0\}$ and $\|Q\|_{C^0}+\|Q^{-1}\|_{C^0}+r^\alpha[Q]_\alpha\leq\Lambda$ on $B_{2r}(x)$. If $B_r(x)$ is $(n-2,\delta)$-symmetric and $\mathcal H^n(B_r(x))\geq \rv r^n$, then $B_r(x)$ is $\varepsilon$-regular.
\end{proposition}
\begin{proof}
By scaling, we may assume $r=1$. Suppose the statement is false. Then
there exist $\delta_i\to0$, boundary-free pointed $\mathsf{ncRCD}(K,n)$ spaces
$(X_i,d_i,x_i)$, coefficients $Q_i$, and minimizing solutions $u_i$
satisfying the uniform assumptions above, such that $B_1(x_i)$ is
$(n-2,\delta_i)$-symmetric but is not $\varepsilon$-regular. After passing to a subsequence, $X_i$ pmGH converge to some pointed $\mathsf{ncRCD}(K,n)$ space $(X_\infty, d_\infty, x_\infty)$, and $B_2(x_i)\to B_2(x_\infty)$. By volume convergence, the convergence is in fact
pmGH. Moreover, since each $X_i$ has empty boundary,
\cite[Theorem~1.6]{BrueNaberSemola2022} implies that $X_\infty$ has
empty boundary as well. The same applies to every tangent cone of
$X_\infty$, since the rescalings of $X_\infty$ are boundary free and
their pmGH limits are again boundary free by the same result. After passing to a further subsequence, the uniform H\"older bounds give
$Q_i\to Q_\infty$ uniformly on $B_2(x_\infty)$, and
\Cref{thm:compactness} gives $u_i\to u_\infty$ uniformly for some
$u_\infty$ minimizing $E_{Q_\infty}$ on $B_2(x_\infty)$.
Since $B_1(x_i)$ is $(n-2,\delta_i)$-symmetric and $\delta_i\to0$, \Cref{prop:basicpropstrata}(iii) implies that $B_1(x_\infty)$ is $(n-2)$-symmetric.
\par
Now, we blow up $u_\infty$ at $x_\infty$.
Since $B_1(x_\infty)$ is already conical and $u_\infty$ is already $1$-homogeneous, we get a global $(n-2)$-symmetric minimizer of $E_{Q_\infty(x_\infty)}$ on the whole boundary-free $(n-2)$-symmetric cone that agrees with $u_\infty$ on $B_1(x_\infty)$. We then apply the classification argument in \Cref{prop:structurestrata}, to conclude that the cone is $\mathbb R^n$ and $u_\infty$ is a regular half-plane solution on $B_1(x_\infty)$, with slope $Q_\infty(x_\infty)$. Thus $B_1(x_\infty)$ is regular. By the convergence above, for all sufficiently large $i$ the ball $B_1(x_i)$ is $\varepsilon$-regular, contradicting our assumption.
\end{proof}

\begin{corollary}\label{cor:Reifenberg}
Let $\rv,L,\Lambda>0$, and $\alpha\in(0,1)$.
There exists $\delta=\delta(n,\mathrm{v},L,\Lambda,\alpha)>0$ with the following property. Let $(X,d,\mathcal H^n)$ be an $\mathsf{ncRCD}(-\delta,n)$ space, and $u$ an $L$-Lipschitz minimizing solution of $E_Q$ on $B_4(x)\subset X$, with $x\in \pt\{u>0\}$. Assume that $\|Q\|_{C^0}+\|Q^{-1}\|_{C^0}+[Q]_\alpha\leq\Lambda$ on $B_4(x)$ and $\mathcal H^n(B_2(x))>\rv$.
If $B_2(x)$ is $(n-2,\delta)$-symmetric, then $\partial\{u>0\}\cap B_{1/2}(x)$ is an $(n-1)$-dimensional topological manifold with locally bi-H\"older coordinate charts.
\end{corollary}
\begin{proof}
Fix $\varepsilon'>0$, to be chosen sufficiently small below. By \Cref{prop:almostregularpoints}, after choosing $\delta$ sufficiently small depending on $n,\mathrm{v},L,\Lambda,\alpha,\varepsilon'$, the $(n-2,\delta)$-symmetry of $B_2(x)$ implies that $B_2(x)$ is $\delta'$-regular, where $\delta'$ is sufficiently small for \Cref{prop:reifenbergprop} with $\varepsilon'$. Choosing also $\delta\leq\delta'$, we conclude that every $B_r(y)\subset B_2(x)$, with $y\in\partial\{u>0\}\cap B_1(x)$ and $r\in(0,1)$, is $\varepsilon'$-regular.
\par
By \Cref{prop:almostsymualmostfree}, if we choose $\varepsilon'$ sufficiently small, depending additionally on $\eta>0$, then for every such $y$ and $r$ the set $\partial\{u>0\}\cap B_{r/2}(y)$, equipped with the distance induced by $d$, is $r\eta$-GH close to $B_{r/2}(0^{n-1})$. Choosing $\eta$ dimensionally small enough, the metric Reifenberg theorem \cite{CC97}, applied locally to the free boundary, implies that $\partial\{u>0\}\cap B_{1/2}(x)$ is an $(n-1)$-dimensional topological manifold with locally bi-H\"older coordinate charts.
\end{proof}

\begin{proof}[Proof of \Cref{cor:intro_biholder}]
    Choose $\tilde{\mathcal{S}}=\mathcal{S}^{n-3}_\delta$, with $\delta$ as in \Cref{cor:Reifenberg}. Then, the Corollary follows from \Cref{cor:Reifenberg}, the Minkowski content bound \eqref{eq:vol1}, and the dimension bound \eqref{eq:vol3} in our main theorem for $k=n-3$.
\end{proof}

\subsubsection{Higher Regularity under Two-sided Ricci Bounds}

Here we briefly explore the case where $X$ is a noncollapsed limit with two-sided Ricci curvature bounds. In this case, the top stratum has higher regularity, and we have the following substitute for \Cref{prop:reifenbergprop}.

\begin{proposition}[$\varepsilon$-regularity under two-sided Ricci curvature bounds] \label{prop:epsregdoublericcibounds} Let $L,\Lambda>0$ and $\alpha\in(0,1)$. There exists $\varepsilon=\varepsilon(n,L,\Lambda,\alpha)>0$ such that the following holds. Suppose that $(X,d,\mathcal H^n)$ is a noncollapsed limit of $n$-dimensional manifolds with two-sided Ricci curvature bounds $|\Ric|\leq\varepsilon$.
Let $u$ be an $L$-Lipschitz minimizing solution of $E_Q$ on $B_4(x)\subset X$, with $x\in\partial\{u>0\}$.
Assume $\|Q\|_{C^0}+\|Q^{-1}\|_{C^0}+[Q]_\alpha\leq\Lambda$ on $B_4(x)$ and $B_4(x)$ is $\varepsilon$-regular.
Then $B_1(x)$ is a $C^{1,\gamma}$ manifold and $\partial\{u>0\}\cap B_1(x)$ is a $C^{1,\gamma}$ submanifold for some $\gamma=\gamma(n,L,\Lambda,\alpha)\in (0,1)$.
\end{proposition}
\begin{proof}
By hypothesis, we have $d_{\mathrm{GH}}(B_4(x),B_4(0^n))<C(n)\varepsilon$, while the approximating manifolds satisfy $|\Ric_{g_i}|_{g_i}\leq\varepsilon$. By volume convergence \cite{colding1997ricci}, for $\varepsilon$ sufficiently small the volume of $B_4(x)$ is arbitrarily close to that of the Euclidean ball. Hence the harmonic radius is uniformly bounded from below \cite{Anderson1990}. Moreover, after decreasing $\varepsilon$ if necessary, $B_3(x)$ admits a harmonic coordinate chart in which the metric coefficients are arbitrarily close to the Euclidean ones in $C^{1,\gamma}$, for some $\gamma(n,L,\Lambda,\alpha)\in(0,1)$. Let $U\subset\mathbb R^n$ denote the image of this chart, and write $u$ and $Q$ also for their coordinate representations. In these coordinates, the Alt--Caffarelli functional becomes
\[
\int_U
\left(
g^{ij}(z)\partial_i u(z)\partial_j u(z)
+Q^2(z)1_{\{u>0\}}(z)
\right)\sqrt{\det g(z)}\,dz.
\]
Setting the matrix $G:=(g^{ij})\sqrt{\det g}$, we obtain a uniformly elliptic $C^{1,\gamma}$ coefficient matrix, arbitrarily close to the identity. Minimizers of this functional are then viscosity solutions of the free boundary problem on $U \subset \mathbb{R}^n$ (see, for instance, \cite{Caffarelli1987} and its sequels):
\begin{equation}
    \begin{cases}
        \div(G\nabla u)=0
        &\text{in $\{u>0\}\cap U$};\\
        \langle G\nabla u,\nabla u\rangle
        =Q^2\sqrt{\det g}
        &\text{on $\partial\{u>0\}\cap U$}.
    \end{cases}
\end{equation}

Fix $y\in\partial\{u>0\}\cap B_1(x)$.
After an affine change of harmonic coordinates, we may
assume that $y=0$ and $g(0)=I$. We keep the notation $u,Q,g$ for their expressions in these coordinates and again set $G:=(g^{ij})\sqrt{\det g}$.
Choose a fixed small scale $\rho>0$ and set
\[
v(z):=\frac{u(\rho z)}{\rho Q(0)},
\qquad z\in B_1(0).
\]
The matrix and free boundary coefficient of the rescaled
problem are
\[
G_\rho(z):=G(\rho z),
\qquad
q_\rho(z):=\frac{Q^2(\rho z)}{Q^2(0)}
\sqrt{\det g(\rho z)}.
\]
They satisfy $G_\rho(0)=I$, $q_\rho(0)=1$, and
\[
[G_\rho]_{C^{0,1}(B_1)}\leq C\rho,
\qquad
[q_\rho]_{C^{0,\alpha}(B_1)}\leq C\rho^\alpha.
\]
We choose $\rho=\rho(n,L,\Lambda,\alpha)$ sufficiently small
to satisfy the coefficient assumptions of \cite[Theorem~6.2]{LT15}, uniformly in $y$. On the other hand, since $B_4(x)$ is
$(n-1,\varepsilon)$-symmetric, by
\Cref{prop:almostsymualmostfree} and
\eqref{eq:freeBoundaryvolumeconv}, together with
\Cref{thm:compactness}, we have
\[
(z_n-\eta)^+\leq v(z)\leq(z_n+\eta)^+
\qquad\text{for }z\in B_1(0),
\]
possibly after a rotation, where $\eta\to0$ as
$\varepsilon\to0$ for the fixed scale $\rho$, uniformly
with respect to the choice of $y$.

We are now in the perturbative regime of the improvement-of-flatness theorem \cite[Theorem~6.2]{LT15}. Thus, when $\varepsilon=\varepsilon(n,L,\Lambda,\alpha)$ is sufficiently small, $\partial\{v>0\}\cap B_{1/2}(0)$ is a $C^{1,\gamma}$ hypersurface, after decreasing $\gamma$ if necessary. Returning to the original coordinates, this gives a $C^{1,\gamma}$ hypersurface in a neighborhood of $y$. Since $y\in\partial\{u>0\}\cap B_1(x)$ was arbitrary, it follows that $\partial\{u>0\}\cap B_1(x)$ is a $C^{1,\gamma}$ submanifold.
\end{proof}
In the same manner as we proved \cref{cor:Reifenberg}, we can prove the following corollary, this time using \cref{prop:epsregdoublericcibounds} instead of \cref{prop:reifenbergprop}.

\begin{corollary}\label{cor:2sided_Ricci_FBP_mfd_struct} 
Let $\mathrm{v},L,\Lambda>0$, and $\alpha\in(0,1)$.
There exists $\delta=\delta(n,\mathrm{v},L,\Lambda,\alpha)>0$ such that the following holds. Let $X$ be a noncollapsed limit of manifolds with two-sided Ricci curvature bounds $|\Ric|<\delta$, and let $u$ be an $L$-Lipschitz minimizing solution of $E_Q$ on $B_4(x)\subset X$, with $x\in\partial\{u>0\}$.
Assume $\|Q\|_{C^0}+\|Q^{-1}\|_{C^0}+[Q]_\alpha\leq\Lambda$ on $B_4(x)$ and $\mathcal H^n(B_2(x))>\mathrm{v}$.
If $B_2(x)$ is $(n-2,\delta)$-symmetric, then $\partial\{u>0\}\cap B_{1/2}(x)$ is a $C^{1,\alpha}$ manifold.
\end{corollary}

\begin{proof}[Proof of \Cref{cor:doublericciboundreg}]
    By the recent work \cite{wang2026dimensionboundsingularset} of the second author, \Cref{cor:2sided_Ricci_FBP_mfd_struct} holds with $n-4$ instead of $n-2$.
    Indeed, the rigidity result in \cite[Theorem~1.4]{wang2026dimensionboundsingularset} allows us to repeat the compactness argument in \Cref{prop:almostregularpoints} and conclude that almost $(n-4)$-symmetric balls are almost regular.
    The local $C^{1,\gamma}$-manifold structure then follows from \Cref{prop:epsregdoublericcibounds}. 
    Combining this strengthened version of \Cref{cor:2sided_Ricci_FBP_mfd_struct} with the Minkowski content bound \eqref{eq:vol1} and the dimension bound \eqref{eq:vol3} in our main theorem for $k=n-5$, we prove \Cref{cor:doublericciboundreg}.
\end{proof}

\section{Truncated Parabolic Weiss Formula}\label{sec:paraweiss}

Here we introduce the truncated Parabolic Weiss quantity, $W_{t_0,t}(x)$, obtained by replacing the Euclidean radial vector field with a vector field constructed from the heat kernel $\rho_t$.
Its derivative in $t$ controls quantitatively the deviation of $u$ from $1$-homogeneity, while the parameter $t_0$ determines the range of scales where we have some bound on $W_{t_0,t}$. The lower cutoff $t_0\leq t$ is essential since it is not clear that the resulting quantity remains finite as $t_0\downarrow 0$. Consequently, with the estimates available, the formula does not provide information at the infinitesimal scale.
\par
Note that this parabolic Weiss quantity is actually defined for solutions on smooth manifolds, as it involves smooth differential calculus in deriving the inner variation formula. Still, this is enough for our purpose since it suffices to prove Theorem \ref{thm:main} about limit solutions in this smooth setting.
\par

In our notation, $\rho_{t,x}$ denotes the heat kernel on $M$ with the first argument fixed at $x$, and the normalization is $\int \rho_{t,x} \dd\vol=1$.
We also introduce the function:
\begin{equation}\label{def:f_t}
	f_t:=-\log((4\pi t)^{\frac{n}{2}}\rho_t),
\end{equation}
so that $2tf_t(x,y)=\frac{1}{2}d(x,y)^2$ in the Euclidean space.
Likewise, $f_{t,x}$ denotes the function with $x$ fixed.
We also denote
\begin{equation}\label{def:H_t}
     H_{t,x}=-\rho_{t,x} \Hess f_{t,x}.
\end{equation}
\par
The main goal of this section is to prove \Cref{thm:parabolic_weiss}, which we restate below for convenience. For the cutoff function $\phi$ used below, recall that such functions exist by \cite{CC96}; see also \cite{MN14} for the more general $\mathsf{RCD}$ setting. 

\begin{theorem}[Truncated Parabolic Weiss Quantity]\label{thm:parabolic_weiss_restated}
Let $u$ be a minimizer of the one-phase problem \eqref{1-phase_Ber} with $X$ being a smooth manifold $(M,g)$. 
Fix some ball $B_{4\bar{r}}(\bar{x})\Subset\Omega$. Assume $\bar{r}\leq \delta^{-1}$, $\Ric \geq -(n-1)\delta^2$, and $\Vol(B_{\bar r}(\bar x))\ge \mathrm{v}\,\bar r^n$, for some $0<\delta\leq 1$ and $\mathrm{v}>0$. Assume also that $\abs{Q+1/Q}_{C^0,B_{4\bar r}(\bar x)} + \overline{r}^{\alpha}[Q]_{\alpha, B_{4\bar r}(\bar x)}<\Lambda$, and that $u$ is $L$-Lipschitz on $B_{4\bar r}(\bar x)$. Pick some $t_0 < \bar{r}^2/2$. Let $\phi:M\to [0,1]$ be a cutoff function with support in $B_{4\bar{r}}(\bar{x})$, satisfying
\begin{equation}\label{cutoff_restated}
    \phi|_{B_{2\bar{r}}(\bar{x})}\equiv 1, \quad |\nabla \phi|^2+|\Delta \phi|\leq C(n)\bar{r}^{-2}.
\end{equation} 
For $x\in B_{\bar{r}}(\bar{x})\cap \pt\{u>0\}$ and $t\in (t_0,\overline{r}^2]$, define
\begin{equation}
    \begin{split}
        W_{t_0,t}(u,g,Q,x):=&\int_M \left[2(|\nabla u|^2+Q^2(x)1_{u>0})-t^{-1}u^2\right] \phi \rho_{t,x}\dd\vol\\
    +&4\int_{t_0}^t \int_M \left|(H_{s,x} +\frac{\rho_{s,x}}{2s}g)(\nabla u, \nabla u)\right| \phi\dd\vol \dd s
    +\mathcal{E}_{t}(u,g,Q,x),
    \end{split}
\end{equation}
    where $\mathcal{E}_t(u,g,Q,x)\geq 0$ is an error term defined by
\begin{equation}\label{eq:parabolic_weiss_error_restated}
        \begin{split}
            \mathcal{E}_t(u,g,Q,x)
:=&
\int_0^t \int_M
\left[
(6|\nabla u|^2+2Q^2(x))|\nabla f_{s,x}||\nabla\phi|
+s^{-1}u^2|\Delta\phi|
\right]
\rho_{s,x}\dd\vol\dd s
\\
&
+2\int_0^t \int_{\partial\{u>0\}}
|\delta_x Q^2|\,|\nabla f_{s,x}|\,
\rho_{s,x}\phi
\dd{}\mathcal{H}^{n-1}\dd s,
        \end{split}
    \end{equation}
where $\delta_{x}Q^2(y):=Q^2(y)-Q^2(x)$. Then $ t \mapsto W_{t_0,t}(u,g,x)$ belongs to $C^1([t_0,\overline{r}^2])$. Moreover, we have 
\begin{align}
\frac{d}{dt}W_{t_0,t}(u,g,x) & \geq t^{-2}\int_M \left[ u-\nabla u\cdot \nabla(2tf_{t,x})\right]^2 \phi \rho_{t,x}\dd\vol. \label{eq:monder_res} \\
|W_{t_0,t}(u,g,x)| & 
\le C(n,\mathrm{v},L,\Lambda,\alpha)
\left(
1+\left|\log\frac{t}{t_0}\right|^{1/2}
+[Q]_{\alpha,B_{4\bar r}(\bar x)}t^{\alpha/2}
\right), \label{eq:Wboundgen_restated}\\
\mathcal E_t(u,g,x)
&\leq
C(n,\mathrm{v},L,\Lambda)\exp\left(-\frac{\bar r^2}{10t}\right)
+
C(n,\mathrm{v},L,\Lambda,\alpha)[Q]_{\alpha, B_{4\overline{r}}(\overline{x})}
t^{\alpha/2}  \label{eq:errorbound}.
\end{align}
\end{theorem}
    \begin{remark}
    When there is no room for confusion, we write $W_{t_0,t}$ instead of $W_{t_0,t}(u,g,Q,x)$. The same applies for the other quantities.
    \end{remark}
    \begin{remark}[Dependence on the lower cutoff]
    Notice that \eqref{eq:Wboundgen_restated} does not provide control on $W$ as $t_0\to 0$, and therefore we cannot deduce infinitesimal information from this quantity. 
    \end{remark}
\begin{remark}[Scaling]
Let $\lambda>0$. The quantity $W$ is invariant under the scaling
\[
(t_0,t,u,g,Q, \bar{r})\mapsto (\lambda^{-2}t_0,\lambda^{-2}t,\lambda^{-1}u,\lambda^{-2}g,Q, \lambda^{-1}\bar{r})
\] 
with the same (in the set-theoretic sense) cutoff $\phi$.
Indeed, since
$\rho^{\lambda^{-2}g}_{\lambda^{-2}t,x}\,d\vol_{\lambda^{-2}g}
=
\rho^g_{t,x}\,d\vol_g$
and $f^{\lambda^{-2}g}_{\lambda^{-2}t,x}=f^g_{t,x}$, it follows that
\[
W_{\lambda^{-2}t_0,\lambda^{-2}t}
\bigl(\lambda^{-1}u,\lambda^{-2}g,Q,x\bigr)
=
W_{t_0,t}(u,g,Q,x).
\]
Also note that $$[Q]_{\alpha,B_1^{\lambda^{-2}g}} = [Q]_{\alpha,B_\lambda^{g}}\lambda^{\alpha} \xrightarrow[]{\lambda\to0} 0.$$
\end{remark}
\begin{remark}[Error term]
The error term $\mathcal{E}_t$ has two sources: the cutoff function $\phi$ and the nonconstancy of $Q$. Regarding the first, ideally one would take $\phi\equiv 1$, in which case the corresponding contribution vanishes. This is the way the reader should think about this term. Indeed, under the rescaling $g\mapsto\lambda^{-2}g$, the cutoff radius becomes $\lambda^{-1}\bar r\to\infty$ as $\lambda\to0$, so that $\phi$ becomes identically one on larger and larger balls. At any fixed rescaled time, the cutoff contribution therefore tends to zero. The contribution arising from the nonconstancy of $Q$ also vanishes, since its H\"older seminorm with respect to the rescaled metric tends to zero. Both effects are quantified by \eqref{eq:errorbound}, which shows directly, without rescaling, that the error tends to zero as $t\to 0$.
\end{remark}
\begin{remark}[Hessian Term]
The appearance of the Hessian term is well known and usually originates from the fact that, on a general manifold, the vector field $\nabla d(x,\cdot)^2$ is not a homothetic Killing field. The same phenomenon occurs here with $f_{t,x}$. This Hessian term is precisely what causes the $t_0$-dependence of the bound in \eqref{eq:Wboundgen_restated}.
On a fixed smooth manifold, such a term is usually controlled by the regularity scale, and hence $W_{t_0,t}$ can provide an almost monotonicity with effective bounds even when $t_0\to 0$; see for instance \cite{SchoenUhlenbeck1982} or \cite{NV17}. This approach is not available here, as the regularity scale might go to zero when singularities form along the GH convergence.
\end{remark}
\begin{remark}[$\delta$-parameter]
In the applications later, after rescaling we will mainly work with $\overline{r}=1$. In this case,
the assumptions of \Cref{thm:parabolic_weiss_restated} allow us to take
$\delta=1$. The resulting bound is recorded in the following
corollary.

\end{remark}

\begin{corollary}
Under the assumptions of \Cref{thm:parabolic_weiss_restated} with $ \overline{r}=1$ and $\Ric\geq -(n-1)$. Then, for $t_0\le t\le1$, and $t_0 \leq 1/2$, we have
\begin{equation}
    |W_{t_0,t}|
    \le C(n,\mathrm{v},L,\Lambda,\alpha)
    \bigl|\log t_0|^{1/2}.\label{eq:Wbound_restated}
\end{equation}
\end{corollary}

We start the proof of Theorem \ref{thm:parabolic_weiss_restated} above with the first inner variation formula:

\begin{lemma}\label{lem:1st_var}
    Let $E$ be the Alt-Cafarelli functional, \Cref{1-phase_Ber}, on a bounded domain $\Omega$ in a smooth Riemannian manifold $M$. Consider domain variations of $u$ given by a smooth compactly supported field $V$. Then we have
    \begin{equation}
        \delta E(u)[V] = \int_\Omega \left(|\nabla u|^2\operatorname{div}_gV-2g(\nabla_{\nabla u}V,\nabla u)\right)
\,\mathrm{dvol}_g
+
{\int_{\partial\{u>0\}}
Q^2\,g(V,\nu)\dd{}\mathcal H_g^{n-1}}.
    \end{equation}
\end{lemma}
\begin{proof}
Consider the flow $\Phi_t$ given by a compactly supported field $-V$. The Alt-Cafarelli functional is composed of the Dirichlet energy and the volume part. For the Dirichlet energy, one has
\begin{equation}
\int_\Omega \abs{\nabla^g (u\circ \Phi_t)}^2_g \dd\vol_g = \int_\Omega \abs{\nabla^{\Phi_{-t}^* g} u}^2_{\Phi_{-t}^*g} \ \dd\vol_{\phi_{-t}^*g}. \label{eq:usualvariation}
\end{equation}
Now the derivative hits $g$, which is differentiable.
Hence, rather than differentiating $u$ along the flow inside the integral, one performs a change of variables and differentiates the volume measure and the metric instead.
The divergence term comes from the variation of the measure, while the other term comes from the Lie derivative of the metric. In this way, no additional regularity of $u$ is required beyond what is already needed to formulate the one-phase problem. We omit the details here since this is a standard computation.
\par
We now turn to the volume term. Set $\Omega_+:=\{u>0\}$. Since $ 1_{\{u\circ\Phi_t>0\}} = 1_{\Omega_+}\circ\Phi_t$,
the term under consideration is
\[
\mathcal V(t)
:=
\int_\Omega
Q^2\,( 1_{\Omega_+}\circ\Phi_t)\,\mathrm{dvol}_g.
\]
The preceding change of variable argument cannot be applied directly, since it would require differentiating $Q^2$ along the flow, while $Q$ is only Hölder continuous. Instead, we use the fact that ${\Omega_+}$ has locally finite perimeter (see e.g. \Cref{thm:mainCZZ22}) to get the following formula:

\begin{equation}
\left.\frac{d}{dt}\right|_{t=0}\mathcal V(t)
=
\int_{\partial^*{\Omega_+}}
Q^2\, g(V,\nu)\dd{}\mathcal H_g^{n-1},\label{eq:lgoal}
\end{equation}
where $\nu$ denotes the measure-theoretic outward unit normal to ${\Omega_+}$, and $\partial^*{\Omega_+}$ is the reduced boundary of $\Omega_+$ (in $\Omega)$.
To see this, first by the proof of \cite[Theorem 2.5]{AmbrosioGhezziMagnani2015}, $\frac{1_{\Omega_+}\circ\Phi_t-1_{\Omega_+}}{t}\dd\vol_g \rightharpoonup^* D_{-X}1_{\Omega_+}$. Hence, testing against the continuous function $Q^2$ gives $\mathcal V'(0)=\int_\Omega Q^2\,dD_{-X}1_{\Omega_+}$.
Next, by the Gauss--Green formula for sets of locally finite perimeter on Riemannian manifolds (\cite[(1.70)]{Ritore2023}), we have $D_{-X}1_{\Omega_+}=g(X,\nu)\dd P(\Omega_+,\cdot)$, where $P(\Omega_+,\cdot)$ is the perimeter measure. Finally, \cite[Theorem 1.39]{Ritore2023} gives $P(\Omega_+,\cdot)=\mathcal H_g^{m-1}\llcorner\partial^*\Omega_+$, which yields the desired formula. Automatically, the same equation also holds with the essential boundary, and then with the free boundary due to Item~3 of \Cref{thm:mainCZZ22}.
\end{proof}

Next, we fix $x\in B_{\bar r}(\bar x)\cap \partial\{u>0\}$ and consider inner variations in the direction of $\nabla \rho_{t,x}$. Although the natural vector field is $\nabla(2t f_{t,x})$, the factor $2t$ can be omitted by linearity.
We also insert $\rho_{t,x}$ as weights in the integrals, since it provides a pseudo-localization on a ball centered at $x$ of radius $\sqrt{t}$ and converts the boundary integral into an $n$-dimensional integral; compare Weiss's quantity in \Cref{thm:orginal_weiss} with our parabolic version here.
Thus, we take $\nabla \rho_{t,x}=-\rho_{t,x}\nabla f_{t,x}$ as the variational vector field. To apply Lemma \ref{lem:1st_var}, we further localize this vector field by multiplying it with the cutoff function $\phi$ from \eqref{cutoff_restated}. This localization produces error terms, which are bounded by the quantity $\mathcal E_t$ defined in \eqref{eq:parabolic_weiss_error_restated}. The resulting computation is recorded in the following lemma, and the error term $\mathcal E_t$ will be estimated in \Cref{lem:parabolic_weiss_error}.

\begin{lemma}\label{lem:parabolic_weiss_compu}
For $x\in B_{\bar{r}}(\bar{x})\cap \pt\{u>0\}$ fixed, we define
\begin{align}
	I(t)&=\int_M \left(|\nabla u|^2+{Q^2(x)}1_{u>0}\right)\phi \rho_{t,x}\dd\vol,\\
	J(t)&=t^{-1}\int_M u^2 \phi \rho_{t,x}\dd\vol.
\end{align}
Then for $t\in (0,1]$, we have
\begin{align}
		2I'(t)-J'(t)+\mathcal{E}_{t}'\geq&t^{-2}\int_M \left[ u-g(\nabla u,\nabla(2tf_{t,x}))\right]^2 \phi \rho_{t,x}\dd\vol\\
		&+4\int_M \left[H_{t,x} +\frac{\rho_{t,x}}{2t}g\right](\nabla u, \nabla u) \phi \dd\vol.
\end{align}
\end{lemma}
\begin{proof}
We omit the subscript $x$ from $\rho_{t,x}$ and $f_{t,x}$ throughout the proof since it is fixed, while we keep it only in the terms involving $Q$, where the distinction between $Q(\cdot)$ and its value $Q(x)$ at the base point is relevant.
Using Lemma \ref{lem:1st_var} with $V=\phi\nabla \rho_{t}$ and $ \delta E(u)[V]=0$ by minimality, we have
\begin{align}
    I'(t)&=\int \Bigl[2 \Hess \rho_{t}(\nabla u, \nabla u)\phi
    +2g(\nabla u, \nabla \phi)g(\nabla u,\nabla \rho_{t})\\
    &\quad\quad -(|\nabla u|^2+{Q^2(x)}1_{u>0})g(\nabla \phi, \nabla \rho_{t})\Bigr] \dd \vol {- \int_{\partial\{u>0\}} \delta_{x}Q^2 \ g(V,\nu) \dd{}\mathcal{H}^{n-1}_g } \\
    &=\int \Bigl[2H_t(\nabla u, \nabla u)\phi
    +2g(\nabla u, \nabla f_t)^2 \phi\rho_t
    -2g(\nabla u, \nabla \phi)g(\nabla u,\nabla f_{t})\rho_t \\
    &\quad \quad +(|\nabla u|^2+{Q^2(x)}1_{u>0})g(\nabla \phi, \nabla f_{t})\rho_t \Bigr]\dd \vol 
    - \int_{\partial\{u>0\}} \delta_{x}Q^2 \ g(V,\nu)  \dd{}\mathcal{H}^{n-1}_g ,
\end{align}
where $\delta_{x}Q^2(y):=Q^2(y)-Q^2(x)$. For $J'(t)$, integration by parts gives
\begin{align}
        J'(t)&=-t^{-2}\int u^2\phi \rho_t \dd \vol
    -t^{-1}\int 2ug(\nabla u, \nabla \rho_t)\phi \dd\vol\\
    &\quad +t^{-1}\int 2u g(\nabla u, \nabla \phi)\rho_t \dd\vol
    +t^{-1}\int u^2 \Delta \phi \rho_t \dd\vol.
\end{align}
An additional integration by parts with the equation $u\Delta u=0$ in the distribution sense (\Cref{thm:mainCZZ22} item 4) gives us
\begin{align}
    t^{-1}\int 2u g(\nabla u, \nabla \phi)\rho_t \dd\vol
    &=-t^{-1}\int 2|\nabla u|^2 \phi\rho_t \dd\vol
    -t^{-1}\int 2u g(\nabla u, \nabla \rho_t)\phi \dd\vol.
\end{align}
Hence we have
\begin{align}
        J'(t)&=-t^{-2}\int u^2\phi \rho_t \dd \vol
    +4t^{-1}\int ug(\nabla u, \nabla f_t)\phi \rho_t \dd\vol\\
    &\quad -2t^{-1}\int |\nabla u|^2\phi\rho_t \dd\vol
    +t^{-1}\int u^2 \Delta \phi \rho_t \dd\vol.
\end{align}
Combining the results above, we have
\begin{align}
    &2I'(t)-J'(t)\\
    =&t^{-2}\int \left[ u-g(\nabla u,\nabla(2tf_{t}))\right]^2 \phi \rho_{t}\,d\vol
	+4\int \left[H_t +\frac{\rho_t}{2t}g\right](\nabla u, \nabla u) \phi\,d\vol\\
        &+\int 2\big(|\nabla u|^2+{Q^2(x)}1_{u>0}\big)g(\nabla \phi, \nabla f_t)\rho_t
        -4 g(\nabla u, \nabla \phi)g(\nabla u, \nabla f_t)\rho_t
        -t^{-1}u^2 \Delta \phi \rho_t\,d\vol\\
        &{-2\int_{\partial\{u>0\}} \delta_{x}Q^2
        g(V,\nu)\,d\mathcal H_g^{n-1}}.
\end{align}
The lemma can now be concluded by the Cauchy--Schwarz inequality.
\end{proof}
Before estimating the error term $\mathcal{E}_t$, we give an upper Ahlfors regularity estimate for the free boundary.

\begin{lemma}[Local upper Ahlfors bound]\label{lem:free_boundary_upper_ahlfors}
Fix $\delta>0$ and $B_R(\bar x)\Subset\Omega$. Assume $\Ric_g\geq-(n-1)\delta^2$, and let $u$ be a minimizing solution of $E_Q$ on $\Omega$. Suppose that $u$ is $L$-Lipschitz and $Q\geq\Lambda^{-1}>0$ on $B_R(\bar x)$. Then, for every $x\in\partial\{u>0\}$ and every ball $B_{2s}(x)\subset B_R(\bar x)$ with $0<s\leq\delta^{-1}$, we have
\begin{equation}
\mathcal H^{n-1}\bigl(\partial\{u>0\}\cap B_s(x)\bigr)
\leq C(n,L,\Lambda)s^{n-1}.
\label{eq:AhlforsRegularity}
\end{equation}
In particular, $\mathcal H^{n-1}\llcorner\partial\{u>0\}$ satisfies a local upper Ahlfors bound. Moreover, assume $\delta R\leq4$, and let $\phi\in C_c^\infty(B_R(\bar x))$ satisfy $0\leq\phi\leq1$ and $|\nabla\phi|\leq C(n)/R$. Then, for every $x\in B_{R/4}(\bar x)$ and every $s>0$,
\begin{equation}
\int_{\partial\{u>0\}\cap B_s(x)}
\phi\,d\mathcal H^{n-1}
\leq C(n,L,\Lambda)s^{n-1}.
\label{eq:weighted_Ahlfors}
\end{equation}
\end{lemma}
\begin{proof}
Choose $\psi_s\in\operatorname{Lip}_c(B_{2s}(x))$ such that $0\leq\psi_s\leq1$, $\psi_s\equiv1$ on $B_s(x)$, and $|\nabla\psi_s|\leq C/s$. By the Euler--Lagrange equation (\Cref{thm:mainCZZ22} item 4) we have
\begin{align*}
\Lambda^{-1}\mathcal H^{n-1}\bigl(\partial\{u>0\}\cap B_s(x)\bigr)
&\leq
\int_{\partial\{u>0\}\cap B_R(\bar{x})}\psi_s Q\dd{}\mathcal H^{n-1}
=
-\int_{B_R(\bar{x})}\langle\nabla u,\nabla\psi_s\rangle\,d\mathcal H^n\\
&\leq
\frac{CL}{s}\mathcal H^n\bigl(B_{2s}(x)\bigr)
\leq
C(n,L)s^{n-1}.
\end{align*}
The weighted estimate follows similarly, testing against $\phi\psi_s$ when $s\leq R/4$, and against $\phi$ when $s\geq R/4$.
\end{proof}
The following lemma provides estimates on the error term $\mathcal{E}_t$.
\begin{lemma}[Error estimation]\label{lem:parabolic_weiss_error}
Assuming $\Ric_g\geq -(n-1)\delta^2$, $\bar r\leq \delta^{-1}$, and
$\Vol B_{\bar r}(\bar x)\geq \mathrm{v}\bar{r}^n$. For $\mathcal E_t$ defined in
\eqref{eq:parabolic_weiss_error_restated}, the cutoff and free-boundary
contributions satisfy, respectively,
\begin{align}
&\int_0^t\int_M
\left[(6|\nabla u|^2+2Q^2(x))|\nabla f_{s,x}||\nabla\phi|
+s^{-1}u^2|\Delta\phi|\right]\rho_{s,x}\dd\vol\dd s
\leq
C(n,\mathrm{v},L,\Lambda)\exp\left(-\frac{\bar r^2}{10t}\right),\\
&\int_0^t\int_{\partial\{u>0\}}
|\delta_xQ^2||\nabla f_{s,x}|\rho_{s,x}\phi
\dd{}\mathcal H^{n-1}\dd s
\leq
C(n,\mathrm{v},L, \Lambda,\alpha)[Q]_{\alpha, B_{4\overline{r}}(\overline{x})}
t^{\alpha/2}.
\end{align}
Consequently,
\begin{equation}
\mathcal E_t(u,g,x)
\leq
C(n,\mathrm{v},L,\Lambda)\exp\left(-\frac{\bar r^2}{10t}\right)
+
C(n,\mathrm{v},L,\Lambda,\alpha)[Q]_{\alpha, B_{4\overline{r}}(\overline{x})}
t^{\alpha/2}\label{eq:errorboundv2}.
\end{equation}
In particular, we have $\mathcal E_t\leq C(n,\mathrm{v},L,\Lambda,\alpha)(1+[Q]_{\alpha, B_{4\overline{r}}(\overline{x})}\overline{r}^{\alpha})$.
\end{lemma}
\begin{proof}
Again, we omit the subscript $x$.
For the contribution from $\phi$, using the Lipschitz bound of $u$, it suffices to bound the terms
    \begin{equation}
        \int_0^t\int_M |\nabla f_{s}||\nabla \phi| \rho_s\dd \vol\dd s\quad \text{and}\quad
        \int_0^t\int_M s^{-1}d_x^2 |\Delta \phi| \rho_{s} \dd\vol \dd s.
    \end{equation}
    For the first term, using \eqref{cutoff_restated} about $\phi$, the heat kernel estimates (\Cref{thm:heat-kernel}) and that $x\in B_{\bar{r}}(\bar{x})$, we have
\begin{align}
    \int_0^t\int_M |\nabla f_{s}||\nabla \phi| \rho_s\dd \vol\dd s
    \leq C(n,\mathrm{v})\bar{r}^{-1}\int_0^t\int_{A_{2\bar{r}, 4\bar{r}}(\bar{x})} (s^{\frac{-1}{2}}+\frac{d_x}{s})s^\frac{-n}{2}\exp(\frac{-d_x^2}{5s})\dd \vol\dd s\\
    \leq C(n,\mathrm{v})\bar{r}^{-1} |A_{2\bar{r}, 4\bar{r}}(\bar{x})|\int_0^t (s^{\frac{-1}{2}}+\frac{5\bar{r}}{s})s^\frac{-n}{2}\exp(\frac{-\bar{r}^2}{5s})\dd s,
\end{align}
where $A_{a,b}$ denotes the annulus with radii $a<b$. The last integral can
be estimated by the change of variables
$\tau=\bar r^2/(5s)$,
$$
\begin{aligned}
\int_0^t
\left(s^{-1/2}+\frac{5\bar r}{s}\right)
s^{-n/2}
\exp\left(-\frac{\bar r^2}{5s}\right)\dd s
&\leq
C(n)\,\bar r^{-n+1}
\int_{\bar r^2/(5t)}^\infty
\left(\tau^{-1/2}+1\right)
\tau^{n/2-1}e^{-\tau}\dd\tau  \\
&\leq
C(n)\,\bar r^{-n+1}
\exp\left(-\frac{\bar{r}^2}{10t}\right).
\end{aligned}
$$
Hence, we can conclude
\begin{equation}
    \int_0^t\int_M |\nabla f_s||\nabla \phi| \rho_s\dd\vol\dd s
    \leq C(n,\mathrm{v})\bar{r}^{-n} |A_{2\bar{r},4\bar{r}}(\bar{x})|
    \exp\left(-\frac{\bar{r}^2}{10t}\right)
    \leq C(n,\mathrm{v})\exp\left(-\frac{\bar{r}^2}{10t}\right).
\end{equation}
The second term can be estimated in a similar way:
\begin{align}
\int_0^t \int_M s^{-1}d_x^2|\Delta\phi|\rho_s\dd\vol\dd s
&\leq C(n)\bar r^{-2}
\int_0^t \int_{A_{2\bar r,4\bar r}(\bar x)}
s^{-1}d_x^2\rho_s\dd\vol\dd s \\
&\leq C(n,\mathrm{v})|A_{2\bar r,4\bar r}(\bar x)|
\int_0^t s^{-\frac n2-1}
\exp\left(-\frac{\bar r^2}{5s}\right)\dd s \\
&= C(n,\mathrm{v})|A_{2\bar r,4\bar r}(\bar x)|\bar r^{-n}
\int_{\bar r^2/(5t)}^\infty
\tau^{\frac n2-1}e^{-\tau}\dd\tau \\
&\leq C(n,\mathrm{v})|A_{2\bar r,4\bar r}(\bar x)|\bar r^{-n}
\exp\left(-\frac{\bar r^2}{10t}\right)
\leq C(n,\mathrm{v})\exp\left(-\frac{\bar r^2}{10t}\right).
\end{align}
It remains to estimate the free-boundary contribution. Since $Q\in C^{0,\alpha}$
and $\Lambda^{-1}\leq Q\leq\Lambda$, we have
\[
\abs{\delta_xQ^2(y)}
\leq C(\Lambda)[Q]_{C^{0,\alpha}}d(x,y)^\alpha.
\]
Using this and the heat kernel estimates from \Cref{thm:heat-kernel}, we obtain
\begin{align}
&\int_{\partial\{u>0\}}
\abs{\delta_xQ^2}\abs{\nabla f_{s,x}}\rho_{s,x}\phi
\dd{}\mathcal H^{n-1} \notag\\
&\qquad\leq
C(n,\mathrm{v},L,\Lambda)[Q]_{C^{0,\alpha}}s^{-n/2}
\int_0^\infty
r^{\alpha}
\left(s^{-1/2}+\frac{r}{s}\right)
\exp\left(-\frac{r^2}{5s}\right)\,r^{n-2}\dd r. \label{eq:Qerror_radial}
\end{align}
Here the factor $r^{n-2}\,dr$ follows from weighted upper Ahlfors bound \eqref{eq:weighted_Ahlfors} and a decomposition into concentric annuli centered at $x$. With the change of variables
$r=\sqrt{s}\,\tau$, the right-hand side of
\eqref{eq:Qerror_radial} is bounded by
\begin{align}
C(n,\mathrm{v},L,\Lambda)[Q]_{C^{0,\alpha}}
s^{-1+\alpha/2}
\int_0^\infty
\tau^{n-2+\alpha}(1+\tau)e^{-\tau^2/5}\dd \tau
\leq
C(n,\mathrm{v},L,\Lambda,\alpha)[Q]_{C^{0,\alpha}}
s^{-1+\alpha/2}. \label{eq:Qerror_pointwise}
\end{align}
Therefore,
\begin{align}
\int_0^t\int_{\partial\{u>0\}}
\abs{\delta_xQ^2}\abs{\nabla f_{s,x}}
\rho_{s,x}\phi\dd{}\mathcal H^{n-1}\dd s 
\leq
C(n,\mathrm{v},L,\Lambda,\alpha)[Q]_{C^{0,\alpha}}
t^{\alpha/2}. \label{eq:Qerror}
\end{align}
This finishes the proof.
\end{proof}

\begin{proof}[Proof of \Cref{thm:parabolic_weiss_restated}]

The inequality for $W_{t_0,t}'$ in \eqref{eq:monder_res} is clear by Lemma \ref{lem:parabolic_weiss_compu}. The inequality in \eqref{eq:errorbound} was proved in \Cref{lem:parabolic_weiss_error}.
Therefore, we focus on the bound of $W_{t_0,t}$ claimed in \eqref{eq:Wboundgen_restated}.
\par
Let's start by bounding the Hessian term in the formula of $W_{t_0,t}$. 
By Cauchy--Schwarz, we have
\begin{align}
    &4\int_{t_0}^t\int_M \left|\left(H_{s,x}+\frac{\rho_{s,x}}{2s}g\right)(\nabla u, \nabla u)\right|\phi \dd\vol\dd s\\
    &\leq C(n,\mathrm{v},L)\int_{t_0}^t\int_M |H_{s,x}+\frac{\rho_{s,x}}{2s}g|\phi \dd\vol\dd s\\
    &=C(n,\mathrm{v},L)\int_{t_0}^t\int_M \abs{\Hess\!\left( f_{s,x}\right) - \frac{g}{2s}}\rho_{s,x}\phi \dd\vol \dd s\\
    &\leq C(n,\mathrm{v},L)\int_{t_0}^t\sqrt{\int_M \abs{\Hess\!\left(f_{s,x}\right) - \frac{g}{2s}}^2 \rho_{s,x}\phi \dd\vol} \dd s\\
    &\leq C(n,\mathrm{v},L)\sqrt{\int_{0}^t 2s\int_M \left|\Hess f_{s,x}-\frac{g}{2s}\right|^2 \rho_{s,x}\phi \dd\vol \dd s} \, \abs{\log \frac{t}{t_0}}^{1/2}.
\end{align}
Thus, it remains to estimate the double integral appearing under the square root in the last line. This integral is closely related to the time-derivative of the $\mathcal{W}$-functional for the linear heat equation, as studied in \cite{ni2004entropy}. The localized analogue of this functional was introduced in \cite[Section 4.6]{CJN21} as \emph{local pointed entropy}.
The argument in \cite[Section~4.6]{CJN21} shows that, for a Riemannian manifold $(M,g)$ satisfying $\Ric_g\geq-(n-1)\delta^2$ and $\Vol(B_r(p))\geq\mathrm v r^n$ for $r\leq\delta^{-1}$, there exists a constant $C(n,\mathrm v)$ such that the quantity (compare with \cite[eq. 0.2]{ni2004entropy}),
\[
\begin{aligned}
\mathcal{W}_t(x)
:= {} & \int_{M^n}
\Big(2t \Delta_f f + t\abs{\nabla f}^2 + f - n\Big)
\phi \rho_{t,x} \dd\vol 
- \int_0^t 4s \int_M
(n-1)\delta^2\abs{\nabla f_{s,x}}^2\phi\rho_{s,x}\dd\vol\dd s \\
& - C(n,\mathrm{v})\delta^2\int_0^t e^{-\frac{1}{100\delta^2s}}\dd s ,
\end{aligned}
\]
is monotone for $0<t\leq\delta^{-2}$. Here $f:=f_{t,x}$, $\Delta_f:=\Delta-\langle\nabla f,\nabla\cdot\rangle$, and $\phi$ is a cutoff as in \cref{cutoff_restated} with $\phi=1$ on $B_{2\delta^{-1}}(p)$ and $0$ outside $B_{4\delta^{-1}}(p)$, with $0\leq\phi\leq1$, and $x\in B_{\delta^{-1}}(p)$.
Below we use this with $\delta:=\bar r^{-1}$ and $p:=\bar x$. As in \cite[Theorem~4.22]{CJN21}, we have
$$\partial_t \mathcal{W}_t(x) \leq -2t \int_{M^n} \abs{\Hess f - \frac{1}{2t}g}^2\phi \rho_{t,x} \dd\vol,$$
and therefore
$$\int_0^{\overline{r}^2} 2s \int_M \abs{\Hess f_{s,x} - \frac{g}{2s}}^2\phi \rho_{s,x} \dd\vol \dd s \leq \mathcal{W}_{0^+} - \mathcal{W}_{\overline{r}^2}.$$
Additionally, because of the heat kernel estimates in \Cref{thm:heat-kernel}, we have $-\mathcal{W}_{\overline{r}^2} \leq C(n,\mathrm{v})$. 
Regarding $\mathcal{W}_{0^+}$, as we are in a smooth manifold, one can check using heat kernel convergence \Cref{thm:heatkernelconvergence}, and the Laplacian estimate in \Cref{thm:heat-kernel}, that $$\mathcal{W}_{0^+} = \int_{\mathbb{R}^n} (\abs{\nabla f_1}^2 + f_1-n)\rho_{1,x} d\vol = 0.$$
This also follows from Item (2) of \cite[Theorem 4.22]{CJN21}.
Altogether, this provides the bound on the Hessian term in $W_{t_0,t}$.
\par
We now show a uniform bound on
$J(t)=t^{-1}\int u^2\phi\rho_t\,d\vol$. Since $u(x)=0$ and $u$ is
$L$-Lipschitz, we have $u\le Ld_x$. Moreover, since
$\bar r\le \delta^{-1}$, the heat kernel estimates from
\Cref{thm:heat-kernel} apply on the support of $\phi$. Hence,
\begin{align}
t^{-1}\int u^2\phi\rho_t\,d\vol
&\le C(n,\mathrm{v},L)\int_{B_{5\bar r}(x)}
d_x^2\,t^{-\frac n2-1}\exp\left(-\frac{d_x^2}{5t}\right)\,d\vol \\
&\le C(n,\mathrm{v},L)\int_0^{5\bar r}
s^{n+1}t^{-\frac n2-1}\exp\left(-\frac{s^2}{5t}\right)\,ds \\
&\le C(n,\mathrm{v},L).
\end{align}
The energy term $I(t)$ is uniformly bounded as well. Indeed, since $|\nabla u|\le L$, $\|Q\|_{C^0}\le \Lambda$, $0\le \phi\le 1$, and the heat kernel is normalized, we have
\[
I(t)=\int_M \bigl(|\nabla u|^2+Q^2(x) 1_{\{u>0\}}\bigr)\phi\rho_{t,x}\,d\vol
\le L^2+\Lambda^2.
\]
\par
Combining all the estimates above with \eqref{eq:errorboundv2} about $\mathcal{E}_t$, we obtain
\[
|W_{t_0,t}|
\le C(n,\mathrm{v},L,\Lambda,\alpha)
\left(
1+\left|\log\frac{t_0}{t}\right|^{1/2}
+[Q]_{\alpha,B_{4\bar r}(x)}t^{\alpha/2}
\right).
\]
This finishes the proof.
\end{proof}

Theorem \ref{thm:parabolic_weiss_restated} shows that the derivative of $W_{t_0,t}$ quantitatively measures how far $u$ is from being $1$-homogeneous, as the function $2t f_{t,x}$ is our replacement for the radial function $d_x^2/2$.
This will be formulated in the almost rigidity lemma, \Cref{lem:alm_rigid}, where we show that the pinching on $W_{t_0,t}$ between scales, together with the pinching of the volume ratio, provides the almost $0$-symmetry at such scales.

\section{Volume Estimates}\label{sec:vol_est}

As mentioned in the introduction, the key point in the volume estimate is to control the number of bad scales around each point. Here, the bad scales are those where we do not have effective closeness to a homogeneous model. After some preparatory work, we will prove the precise statement in \eqref{eq:BoundBadScales}.

\par
The main tools for controlling the number of bad scales are the truncated parabolic Weiss quantity and the associated estimates established in \Cref{sec:paraweiss}, the almost-rigidity lemma, and the quantitative
cone-splitting principle. In \Cref{subsec:QuantRigidityConeSplitting}, we show that even in the absence of boundedness at the infinitesimal scale, the almost rigidity result still holds; namely, pinching in the truncated parabolic Weiss quantity implies almost $0$-symmetric behavior. In the same section, we also prove the quantitative cone-splitting principle for this problem, which relies on the addition of symmetries for symmetry models and the compactness of the problem.
\par
In the last \Cref{sec:pf_of_main}, we carry out the covering argument of \cite{CN13} to prove our main theorem.
Although the argument is by now well known, we include the details because, in our setting, the total number of bad scales is no longer bounded. We show that the available estimate on the number of bad scales is still enough to adapt the argument to this unbounded scenario.

\subsection{Quantitative Rigidity and Cone Splitting}\label{subsec:QuantRigidityConeSplitting}

From Theorem \ref{thm:parabolic_weiss_restated}, we know that the pinching on $W_{t_0,t}$ provides almost homogeneity on $u$ in terms of $2t f_{t,x}$, which is our replacement for the radial function $d_x^2/2$. In the almost rigidity lemma (\Cref{lem:alm_rigid}) below, we use the pinching of the volume ratio  $\mathcal{V}$ to control the closeness of $\nabla(2t f_{t,x})$ and the radial vector field, while also obtaining Gromov--Hausdorff closeness of the
underlying metric ball to a ball in a metric cone, via the ``almost volume cone implies almost metric cone theorem" in \cite{CC96}. In conclusion, the pinching of $W_{t_0}$ and $\mathcal{V}$ together provides almost $0$-symmetry as in our \Cref{def:quant_sym}.

\begin{lemma}[Almost Rigidity]\label{lem:alm_rigid}
Let $u$ be a minimizer of the one-phase problem \eqref{1-phase_Ber} with $X$ being a smooth manifold $(M,g)$ with $\Ric\geq -(n-1)$. 
Fix some ball $B_{4\bar{r}}(\bar{x})\Subset\Omega$ for some $\bar{r}\leq 1$.
Assume $\Vol(B_{\bar r}(\bar x))\ge \mathrm{v}\,\bar r^n$, $\abs{Q+1/Q}_{C^0,B_{4\bar r}(\bar x)} + \overline{r}^{\alpha}[Q]_{\alpha, B_{4\bar r}(\bar x)}<\Lambda$, and that $u$ is $L$-Lipschitz on $B_{4\bar r}(\bar x)$. For $x\in B_{\bar{r}}(\bar{x})\cap \pt\{u>0\}$, let $\mathcal{V}$ be the volume ratio:
    \begin{equation}
        \mathcal{V}_{-1}(r):=\frac{\Vol^M(B_r(x))}{\Vol^{\mathbb{H}^n_{-1}}(B_r(x))},
    \end{equation}
    and let $W_{t_0,t}$ be as in \Cref{thm:parabolic_weiss_restated}. Then, 
    for every $\varepsilon>0$ and $\gamma\in(0,\frac12)$, there exists $\delta$ with $\delta = \delta(n,\mathrm{v},L,\Lambda,\alpha,\varepsilon,\gamma)\in(0,\frac12)$
such that if
    \begin{equation}
        a) \ \ W_{t_0,t_2}(x)-W_{t_0,t_1}(x)<\delta \quad \text{and} \quad b) \ \ \log \mathcal{V}_{-1} ( \sqrt{t_1})-\log \mathcal{V}_{-1}(\delta^{-1}\sqrt{t_2})<\delta
    \end{equation}
    for some $t_0\leq t_1<t_2\leq \delta^2 \overline{r}^2$ with $\frac{t_1}{t_2}=\gamma^2$, then $B_{\sqrt{t_1}}(x)$ is $(0,\varepsilon)$-symmetric.
\end{lemma}
\begin{remark}
None of the quantities above actually depends on the choice of $t_0$, since $W_{t_0,t_2}-W_{t_0,t_1}=W_{t_0',t_2}-W_{t_0',t_1}$.
Thus $t_0$ plays no role other than making $W$ well-defined.
\end{remark}
\begin{proof}
First, set $\kappa:=-\bar r^2\in[-1,0)$ and consider the rescaling $$(t_0,t,u,g,Q)\mapsto (\lambda^{-2}t_0,\lambda^{-2}t,\lambda^{-1}u,\lambda^{-2}g,Q)$$ with $\lambda=\bar r$. Rescaling the hyperbolic reference metric as well, the volume ratio becomes $\mathcal V_\kappa$. Thus, we may assume that $\bar r=1$ and $\Ric_g\geq(n-1)\kappa g$, with the volume pinching computed using $\mathcal V_\kappa$. The pinching assumptions and the Lipschitz bound are preserved, the norm on $Q$ is scale-invariant, and so is the notion of symmetry.

\par
We argue by contradiction and suppose that the statement is false. Then there exist $\varepsilon>0$, $\gamma^2\in(0,\frac14)$, a sequence $\delta_i\to0$, manifolds $M_i$, bounded domains $\Omega_i\subset M_i$, minimizers $u_i$, points $\bar x_i,x_i$, functions $Q_i$, and scales
$t_{0,i}<t_{1,i}=t_{2,i}\gamma^2<t_{2,i}\leq \delta_i^2$
satisfying all the assumptions above, except that $B_{\sqrt{t_{1,i}}}(x_i)$ is not $(0,\varepsilon)$-symmetric. For each $i$, we rescale again by $(t_0,t,u,g,Q)_i\mapsto (\lambda_i^{-2}t_0,\lambda_i^{-2}t,\lambda_i^{-1}u,\lambda_i^{-2}g,Q)_i =: (\widetilde{t}_0,\widetilde{t}, \widetilde{u}, \widetilde{g},\widetilde{Q})_i$,
with $\lambda_i=t_{2,i}^{1/2}$. In the rescaled sequence, we have
\begin{equation}
\begin{aligned}
a)'\quad&
W_{\widetilde t_{0,i},1}
    (\widetilde u_i,\widetilde g_i,\widetilde Q_i,x_i)
-
W_{\widetilde t_{0,i},\gamma^2}
    (\widetilde u_i,\widetilde g_i,\widetilde Q_i,x_i)
<\delta_i,\\
b)'\quad&
\log\widetilde{\mathcal V}_i(\gamma)
-\log\widetilde{\mathcal V}_i(\delta_i^{-1})
<\delta_i,\\
c)'\quad&
[Q_i]_{\alpha,B_R^{\lambda_i^{-2}g_i}(x_i)}
=
\lambda_i^\alpha
[Q_i]_{\alpha,B_{R\lambda_i}^{g_i}(x_i)}
\le \Lambda\lambda_i^\alpha
\le \Lambda\delta_i^\alpha
\to 0
\text{ for every fixed $R<\infty$}.
\end{aligned}
\label{eq:aybprima}
\end{equation}
Here the rescaled volume ratio is
\[
\widetilde{\mathcal V}_i(r)
:=
\mathcal V_{\kappa t_{2,i}}^{\widetilde g_i}(r)
=
\mathcal V_{\kappa}^{g_i}
\bigl(\sqrt{t_{2,i}}\,r\bigr).
\]
Moreover,
\begin{equation}
\Ric_{\widetilde g_i}
\geq (n-1)\kappa t_{2,i}\widetilde g_i
\geq -(n-1)t_{2,i}\widetilde g_i
\geq -(n-1)\delta_i^2\widetilde g_i.
\label{eq:dprima}
\end{equation}
Finally, $B_\gamma^{\widetilde g_i}(x_i)$ is not
$(0,\varepsilon)$-symmetric. From now on, we drop the tildes and use the original notation for the rescaled sequence. Thus, the Ricci lower bounds converge to zero,  \eqref{eq:aybprima} holds, and $B_{\gamma}(x_i)$ is not $(0,\varepsilon)$-symmetric.
We now derive a contradiction based on these conditions.
\par
By $b)'$ and $\eqref{eq:dprima}$, and the almost volume cone implies almost metric cone theorem from \cite{CC96}, we have that $(M_i,g_i,x_i)\xrightarrow[]{\mathrm{pGH}} (C(Z),d_\infty,x_\infty)$.
Moreover, the convergence is also in the pmGH-sense, as we assume a lower volume bound in the sequence. It follows that $(C(Z),d_\infty, \mathcal{H}^n)$ is an $\mathsf{ncRCD}(0,n)$ cone.
\par

We next consider the coefficients $Q_i$.
By $c)'$, the oscillation of $Q_i$ on every fixed ball tends to zero.
Since $\Lambda^{-1}\le Q_i\le \Lambda$, after passing to a subsequence we may assume that $Q_i(x_i)\to Q_\infty$ for some constant $Q_\infty\in[\Lambda^{-1},\Lambda]$. Together with $c)'$, this implies that $Q_i\to Q_\infty$ locally uniformly.
\par
At the function level, we use the uniform Lipschitz bound on $u_i$. As $x_i$ is a zero of $u_i$, we are in the hypothesis of \Cref{thm:compactness} for each radius $R$, and hence $u_i$ converges uniformly on compact subsets to a minimizer $u_\infty$ of the one-phase problem with constant coefficient $Q_\infty$ on the cone $C(Z)$. 
In addition, $u_i$ converges to $u_\infty$ also in the strong $W^{1,2}(B_R(x_i))$ sense, for every $R>0$.
It now suffices to prove that $u_\infty$ is $1$-homogeneous on $B_{\gamma}(x_\infty)$ to get a contradiction.

By Theorem \ref{thm:parabolic_weiss_restated} for $u_i$ with the heat kernel $\rho_{t,x_i}$ on $M_i$ and the corresponding $f_{t, x_i}$, we have
\begin{equation}
    \int_{\gamma^2}^{1}t^{-2} \int_{M_i} 
    |u_i-\nabla u_i\cdot \nabla (2tf_{t,x_i})|^2 \phi_i \rho_{t,x_i}\dd \vol \dd t
    \leq W_{t_0,1} -W_{t_0,\gamma^2}<\delta_i\to 0 . 
\end{equation}
Therefore, up to a subsequence, the integrand of the outer integral converges to $0$ for a.e.\@ $t \in(\gamma^2,1)$. For such $t$, with \Cref{thm:heat-kernel} about $\rho_{t,x_i}$, we have
\begin{equation}\label{eq:quant_homog_u_i}
    \int_{B_{\gamma}(x_i)} |u_i-\nabla u_i\cdot \nabla 2tf_{t,x_i}|^2\dd\vol \to 0
\end{equation}
Note that $B_\gamma(x_i)$ is with respect to $g_i$, as already suggested by the fact $x_i \in M_i$. The same is true for $\dd\vol$, the heat kernel, the dot product, and the gradient operator.
\par
By \Cref{lem:ft_convergence}, $f_{t,x_i}\to f_{t,x_\infty}$ strongly in $W^{1,2}(B_\gamma(x_i))$.
Together with  $u_i\to u_\infty$ strongly in $W^{1,2}(B_\gamma(x_i))$ , we can apply the local weak stability of the gradient (\cite[Theorem~4.6(2)]{AH18} with $p=2$) and get
\[
\nabla u_i\cdot\nabla(2t f_{t,x_i}) \rightharpoonup 
\nabla u_\infty \cdot \nabla(2t f_{t,x_\infty}) \ \ \text{weakly in } L^2(B_\gamma(x_i)).
\]
It now follows from \eqref{eq:quant_homog_u_i} that $u_\infty = \nabla u_\infty\cdot \nabla(2t f_{t,x_\infty})$ holds $\mathcal H^n$-a.e.\@ on $B_\gamma(x_\infty)$. 
Since $B_\gamma(x_\infty)$ is in the $\mathsf{ncRCD}(0,n)$ cone $C(Z)$ with vertex $x_\infty$, the heat-kernel formula \eqref{eq:heatkernelcone} gives $\nabla(2t f_{t,x_\infty}) = d_{x_\infty}\nabla d_{x_\infty}$. 
Thus
\[
u_\infty
=
d_{x_\infty}
\left\langle
\nabla u_\infty,\nabla d_{x_\infty}
\right\rangle
\quad
\mathcal H^n\text{-a.e. on $B_\gamma(x_\infty)$}.
\]
Using the polar coordinates $(0,\infty)_r\times Z$ on $C(Z)$, the Sobolev calculus on warped products \cite{GigliHan2018} gives
$\langle \nabla u_\infty,\nabla d_{x_\infty}\rangle
=\partial_r u_\infty$ for $\mathcal{H}^n$- a.e.\@ point. Hence, for $\mathcal H^{n-1}_Z$-a.e.\@ $z\in Z$, the Lipschitz function $r\mapsto u_\infty(r,z)$ satisfies $r\,\partial_r u_\infty(r,z)=u_\infty(r,z)$ for a.e.\@ $r\in(0,\gamma)$. Therefore $u_\infty$ is $1$-homogeneous on $B_\gamma(x_\infty)$. We now have a contradiction since $B_{\gamma}(x_i)$ is not $(0,\varepsilon)$-symmetric. Since the argument only uses $-1\leq\kappa<0$, the choice of $\delta$ is uniform in $\kappa$, and hence independent of the original radius $\bar r$.
\end{proof}

In the remainder of this subsection, we establish the cone-splitting principle. The first lemma below addresses the rigid case, both at the space level and at the function level. A cone-splitting statement at the space level for general $\mathsf{RCD}(0,N)$ cones was proved in \cite[Proposition~1.18]{ABS19}, relying on Gigli's splitting theorem \cite{Gigli2026}. In our setting, however, the cone models are $\mathsf{ncRCD}(0,n)$ cones with the $n$-dimensional Hausdorff measure.
Therefore, any metric isometry between cone models automatically preserves the reference measure, and it suffices to establish the splitting at the metric level. Based on this observation, we give below a purely metric proof for proper metric cones with links of diameter at most $\pi$, without using Ricci curvature bounds.\footnote{The possibility of this was observed in \cite{CN13}, though a proof using Cheeger-Gromoll Splitting in Ricci limits Cones $X$ with $\Ric_X \geq 0 $ is provided instead.}.
The corresponding addition of symmetries at the function level is included in the statement as well.
\par
We use the following convention. Suppose that $\widetilde u:\mathbb R^k\times C(Z)\to [0,\infty)$ is a $k$-symmetric model, so that $\widetilde u$ is $1$-homogeneous and $\mathbb R^k$-invariant. If $z_*$ denotes the vertex of $C(Z)$, we set $\mathcal L^k:=\mathbb R^k\times\{z_*\}$, and say that the model is $k$-symmetric with respect to $\mathcal L^k$.

\begin{lemma}[Adding Symmetries/Cone Splitting, Rigid Case]\label{lemma:addingsym}
Let $(X,d,\mathcal{H}^n)$ be an $\mathsf{ncRCD}(K,n)$ space and $u:\Omega \subset X \to [0,\infty)$ a locally Lipschitz function. Suppose that
\begin{enumerate}
\item $B_{1}(x)$ is $k$-symmetric with respect to $\mathcal{L}^k$.
\item There exists $y \in B_1(x) \setminus \mathcal{L}^k$ such that $B_{2}(y)$ is $0$-symmetric.
\end{enumerate}
Then $B_1(x)$ is $(k+1)$-symmetric with respect to a
$(k+1)$-dimensional plane $\mathcal{L}^{k+1}$ satisfying $\mathcal{L}^{k}\subset \mathcal{L}^{k+1}$ and
$y\in \mathcal{L}^{k+1}$.
\end{lemma}
\begin{proof}
By hypothesis, there are isometries $I:B_2^{C(Z_y)}(y_*)\longrightarrow B_2(y)$, and $J:B_1^M(\hat x)\longrightarrow B_1(x)$, where $M:=C(Z_x)\times\mathbb R^k$, $\hat x:=(z_*,0)$, and $y_*$ and $z_*$ are the vertices of $C(Z_y)$ and $C(Z_x)$, respectively.
These isometries satisfy $I(y_*)=y$ and $J(\hat x)=x$, and $J$
identifies the original symmetry plane with
$\{z_*\}\times\mathbb R^k$ within $B_1^M(\hat x)$.

Since $B_1(x)\subset B_2(y)$, the composition
\[
\widetilde J:=I^{-1}\circ J:
B_1^M(\hat x)\longrightarrow B_2^{C(Z_y)}(y_*)
\]
is a well-defined isometric embedding.
Let $\hat y:=J^{-1}(y)$, so that $\widetilde J(\hat y)=y_*$. Using polar coordinates $(r,z)$ on $C(Z_x)$, write $\hat y=((b,z_0),a), a\in\mathbb R^k, b>0$, where $b>0$ since $y\notin\mathcal L^k$. Set $\ell:=d_M(\hat x,\hat y)=d(x,y)<1$ and choose $0<\rho<1-\ell$. Then $\widetilde J$ restricts to an isometry $\widetilde J:B_\rho^M(\hat y)\longrightarrow B_\rho^{C(Z_y)}(y_*)$.
\par
Let $r_i\downarrow0$, and denote by $\delta_t^y$ the homothety of ratio $t$ centered at the vertex $y_*$ of $C(Z_y)$.
For every fixed $S>0$ and all sufficiently large $i$, the map
\[
F_i:=\delta_{1/r_i}^y\circ\widetilde J
\]
restricts to an isometry
\begin{equation}\label{eq:rigid_cone_rescaled_isometry}
F_i:
\bigl(B_{Sr_i}^M(\hat y),r_i^{-1}d_M,\hat y\bigr)
\longrightarrow
\bigl(B_S^{C(Z_y)}(y_*),d_{C(Z_y)},y_*\bigr).
\end{equation}
Consequently,
\[
(M,r_i^{-1}d_M,\hat y)
\xrightarrow{\mathrm{pGH}}
(C(Z_y),d_{C(Z_y)},y_*).
\]
In particular, the tangent cone of $M$ at $\hat y$ exists and is unique. By the dimension reduction argument, after passing to a subsequence, there are pointed GH approximations
\[
\phi_i:
(Z_x,(b/r_i)d_{Z_x},z_0)\longrightarrow (Y,p)
\]
such that the maps
\begin{equation}
\Phi_i((r,z),e)
:=
\left(
\phi_i(z),
\frac{r-b}{r_i},
\frac{e-a}{r_i}
\right)\label{eq:GHrealization}
\end{equation}
realize the pointed GH convergence
\[
(M,r_i^{-1}d_M,\hat y)
\xrightarrow{\mathrm{pGH}}
(Y\times\mathbb R\times\mathbb R^k,(p,0,0)).
\]
By \eqref{eq:rigid_cone_rescaled_isometry}, after passing to a
further subsequence, the maps $F_i$ converge on bounded sets to
a pointed isometry
\begin{equation}\label{eq:rigid_cone_limit_isometry}
\psi:
(Y\times\mathbb R\times\mathbb R^k,(p,0,0))
\longrightarrow
(C(Z_y),y_*).
\end{equation}

We now identify the original symmetry plane under this isometry.
For $c\in\mathbb R^k$ with $|c|<1$, set
$\hat z_c:=(z_*,c)$ and consider the geodesic $\gamma_c(t)
= \bigl(((1-t)b,z_0),\,a+t(c-a)\bigr), 0\leq t\leq1,$
from $\hat y$ to $\hat z_c$.
This geodesic lies entirely in $B_1^M(\hat x)$, since $d_M(\hat x,\gamma_c(t))
\leq
(1-t)\ell+t|c|
<1.$
Therefore, $\widetilde J\circ\gamma_c$ is the radial geodesic
from $y_*$ to $\widetilde J(\hat z_c)$, and hence
$\widetilde J(\gamma_c(t))=\delta_t^y\bigl(\widetilde J(\hat z_c)\bigr)$. In particular, for all sufficiently large $i$,
\[
F_i(\gamma_c(r_i))
=
\widetilde J(\hat z_c).
\]
With respect to the GH approximations $\Phi_i$ introduced above, we have
\[
\Phi_i(\gamma_c(r_i))
=
\bigl(\phi_i(z_0),-b,c-a\bigr)
\longrightarrow
(p,-b,c-a).
\]
Since $F_i$ converges to $\psi$ with respect to this realization of the
GH convergence and $F_i(\gamma_c(r_i))=\widetilde J(\hat z_c)$, passing to the limit yields
\begin{equation}\label{eq:rigid_cone_original_plane}
\psi(p,-b,c-a)=\widetilde J(z_*,c),
\qquad |c|<1.
\end{equation}
Thus both $\widetilde J(\hat x)$ and the image of the original
symmetry plane lie in
$\psi(\{p\}\times\mathbb R\times\mathbb R^k)$. By \eqref{eq:rigid_cone_limit_isometry}, the product
$Y\times\mathbb R\times\mathbb R^k$ is a metric cone centered at $(p,0,0)$. Its radial homotheties preserve the slice $Y\times\{0\}\times\{0\}$, since the Euclidean components of geodesics issuing from $(p,0,0)$ vary linearly. Restricting the cone distance formula to this slice shows that $Y$ is itself a metric cone with vertex $p$.
We therefore write $Y=C(\widehat Z_y)$,
allowing the trivial cone consisting of a single point.

Finally, since $\ell<1$, the ball of radius $1$ centered at $\widetilde J(\hat x)$ is contained in $B_2^{C(Z_y)}(y_*)$. Hence $I\circ\psi$ restricts to an isometry
\begin{equation}
I \circ \psi: B_1^{C(\widehat Z_y)\times\mathbb R\times\mathbb R^k}
(p,-b,-a)
\longrightarrow B_1(x).
\end{equation}
This proves the desired conclusion at the space level.

The preceding argument actually shows that the original
$\mathbb R^k$-factor is identified with the last $\mathbb R^k$-factor in the new coordinates, and not only that the corresponding symmetry planes are identified. Indeed, set
$H:=\psi^{-1}\circ\widetilde J$. By \eqref{eq:rigid_cone_original_plane}, $H(z_*,c)=(p,-b,c-a)$. Now let $(q,e)\in B_1^M(\hat x)$ and write $H(q,e)=(Q,s,\eta)$. Comparing the distance from $(q,e)$ to the points $(z_*,c)$ with the
distance from its image to $(p,-b,c-a)$, for all $c$ in the relevant ball, gives $\eta = e-a$.

Moreover, if
\[
H(q,e_i)=(Q_i,s_i,e_i-a), \qquad i=1,2,
\]
then the fact that $H$ is an isometry yields
\[
|e_1-e_2|^2
=
d_{C(\widehat Z_y)\times\mathbb R}
\bigl((Q_1,s_1),(Q_2,s_2)\bigr)^2
+
|e_1-e_2|^2,
\]
and hence $(Q_1,s_1)=(Q_2,s_2)$. Thus, locally,
\[
H(q,e)=\bigl(H_0(q),e-a\bigr)
\]
for some map $H_0$. In particular, translations along the original
$\mathbb R^k$-factor correspond exactly to translations along the last
$\mathbb R^k$-factor in the new coordinates.

We now turn to the function level. Denote by $\widetilde u$ the function transported to the new coordinates, and continue to denote the corresponding centers by $x$ and $y$. After a rigid motion of the Euclidean factor $E:=\mathbb R^{k+1}$, we may write
\[
x=(p,0),\qquad y=(p,\ell e_1),
\]
where $e_1$ is the first coordinate vector and $\ell=d(x,y)<1$. The preceding observation shows that $\widetilde u$ is invariant along a $k$-dimensional subspace $V\subset E$. Since $y$ lies outside the original symmetry plane, we have
\[
E=\mathbb R e_1\oplus V,
\]
where the sum need not be orthogonal. Write each Euclidean coordinate as $re_1+\xi$, with $\xi\in V$. By $V$-invariance, $\widetilde u$ descends to a function of $(Q,r)$
on the image of $B_1(x)$ under the projection
\[
(Q,re_1+\xi)\longmapsto(Q,r).
\]
We keep the same notation for this function. Using polar coordinates
$Q=(z,t)$ on $C(\widehat Z_y)$, we write its arguments as
$((z,t),r)$.

The homogeneity with respect to both centers is preserved by the
preceding identifications, which preserve radial geodesics.
Moreover, contractions centered at either $x$ or $y$ preserve
$B_1(x)$: for $P\in B_1(x)$ and $0<\mu\leq1$, the product structure
gives
\[
d\bigl(x,\delta_\mu^y(P)\bigr)
\leq \mu d(x,P)+(1-\mu)\ell<1,
\]
and the corresponding assertion for contractions centered at $x$ is immediate. Consequently, both homogeneity identities descend to the reduced coordinates. Take two points $((z,t),r)$ and $((z,t),r')$ in the projected domain.
Interchanging them if necessary, assume $r'\geq r$ and set
\[
\lambda:=1+\frac{r'-r}{\ell}\geq1.
\]
In the reduced coordinates, contraction by $\lambda^{-1}$
centered at $y$ sends the first point to
\[
\left((z,\lambda^{-1}t),
\ell+\lambda^{-1}(r-\ell)\right)
=
\left((z,\lambda^{-1}t),\lambda^{-1}r'\right),
\]
which is also the image of the second point under contraction
by $\lambda^{-1}$ centered at $x$. By the two homogeneity identities,
\[
\lambda^{-1}\widetilde u((z,t),r)
=
\widetilde u\left((z,\lambda^{-1}t),\lambda^{-1}r'\right)
=
\lambda^{-1}\widetilde u((z,t),r').
\]
Hence
\[
\widetilde u((z,t),r)=\widetilde u((z,t),r').
\]
Thus the reduced function is independent of $r$. Together with
the previously established $V$-invariance, this proves that
$\widetilde u$ is invariant along the entire Euclidean factor
$E=\mathbb R^{k+1}$ on $B_1(x)$, completing the proof.

\end{proof}

By a compactness argument, the preceding lemma yields
the following quantitative cone-splitting theorem. At the space level, compactness follows from the Ricci lower bound and the noncollapsed assumption, while at the function level it follows from the uniform Lipschitz bound.
\par
In the arguments below, we will keep track of the approximate symmetry plane. More precisely, suppose that $B_r(x)$ is $(k,\delta)$-symmetric with a symmetry model $\mathbb{R}^k\times C(Z)$, and let
$F:B_r(0,z_*)\subset\mathbb R^k\times C(Z)\to B_r(x)$ be the Borel $\delta r$-pGH approximation appearing in
\Cref{def:quant_sym}.
We set $\mathcal L_\delta^k:=F\bigl((\mathbb R^k\times\{z_*\})\cap B_r(0,z_*)\bigr),$
and say that $B_r(x)$ is $(k,\delta)$-symmetric with respect to
$\mathcal L_\delta^k$.

\begin{theorem}[Quantitative cone splitting]
Let $(X,d,\mathcal{H}^n)$ be an $\mathsf{ncRCD}(K,n)$ space, $K<0$, $\Omega\subset X$ a bounded domain, and $u:\Omega\to [0,\infty)$ an $L$-Lipschitz function.
Assume that $\mathcal{H}^n(B_1(p))\geq \mathrm{v}$ for any $p\in \Omega$. Fix some positive $\varepsilon,\tau$ and $0<\chi\leq 1/2$, and let $k\in \mathbb{N}_0\cap[0,n-1]$.
Then there exists a positive $\delta=\delta(n,K,\mathrm{v}, L,\tau,\varepsilon,\chi)$ such that, if $0<r\leq 1$ and
\begin{enumerate}
    \item $B_{\chi^{-1} r}(x)$ is $(k,\delta)$-symmetric with respect to $\mathcal{L}^k_\delta$.
    \item For some $y\in B_{r}(x)\setminus B_{\tau r}(\mathcal{L}^k_\delta)$, $B_{\chi^{-1} r}(y)$ is $(0,\delta)$-symmetric,
\end{enumerate}
then $B_{r}(x)$ is also $(k+1,\varepsilon)$-symmetric.
\end{theorem}
\begin{proof}
Suppose by contradiction that the statement is not true.  Then, there is a sequence $u_i,\Omega_i,X_i$ with the uniform constants in the statement, and two sequences ${r}_i\leq 1$ and $\delta_i \to 0$, with $B_{\chi^{-1} r_i}(x_i)$ being $(k,\delta_i)$-symmetric with respect to $\mathcal{L}_{\delta_i}^k$, and $y_i \in B_{r_i}(x_i)\backslash B_{\tau r_i}(\mathcal{L}_{\delta_i}^k)$, with $B_{\chi^{-1} r_i}(y_i)$ being $(0,\delta_i)$-symmetric. Rescale the sequence  by $(u_i,d_i,x_i)\mapsto (\lambda^{-1}_iu_i,\lambda^{-1}_id_i,x_i)$, with $\lambda_i = r_i$. The new sequence satisfies 
\begin{enumerate}
\item[1'.] $B_{\chi^{-1}}(x_i)$ is $(k,\delta_i)$-symmetric with respect to $\mathcal{L}_{\delta_i}^k$. 
\item[2'.] $y_i \in B_{1}(x_i) \backslash B_{\tau}(\mathcal{L}_{\delta_i}^k)$, and $B_{\chi^{-1}}(y_i)$ is $(0,\delta_i)$-symmetric.
\end{enumerate}
but $B_1(x_i)$ is not $(k+1,\varepsilon)$-symmetric. It suffices to show that the existence of this sequence provides a contradiction.
\par
Extend each $u_i$ to a nonnegative $L$-Lipschitz function on $X_i$, without changing the notation. By compactness and the uniform volume lower bound (see \cite{DPG2018}), after passing to a subsequence, the spaces $(X_i,x_i)$ converge in the pmGH sense to an $\mathsf{ncRCD}(K,n)$ space $(X,x)$. In particular, $B_{\chi^{-1}}(x_i)\to B_{\chi^{-1}}(x)$ with respect to this convergence. Moreover, since $u_i(x_i)=0$, by Arzel\'a--Ascoli, after passing to a further subsequence, the functions $u_i$ converge locally uniformly to an $L$-Lipschitz function $u:X\to[0,\infty)$.
\par
Now, consider the symmetry models $(B_{\chi^{-1}}(0,z_{*,i})\subset\mathbb R^k\times C(Z_i),
\tilde u_i)$ and the Borel $\delta_i\chi^{-1}$-pGH approximations $F_i:B_{\chi^{-1}}(0,z_{*,i}) \to B_{\chi^{-1}}(x_i)$ defining $\mathcal L_{\delta_i}^k$ given in Item~1'.
By \Cref{prop:basicpropstrata}, Item~(iii)-(a), replacing only the model functions $\tilde u_i$ if necessary, while keeping $\mathbb{R}^k\times C(Z_i)$ and $F_i$ fixed, we may pass to a subsequence such that $F_i$ converge to an isometry $F:B_{\chi^{-1}}(0,z_*)\subset\mathbb R^k\times C(Z) \to B_{\chi^{-1}}(x)$, and $F$ pulls back $u$ to a $k$-symmetric model $\tilde u$. Setting $\mathcal L^k:= F\bigl((\mathbb R^k\times\{z_*\}) \cap B_{\chi^{-1}}(0,z_*)\bigr)$, we obtain $\mathcal L_{\delta_i}^k\to\mathcal L^k$. After passing to a further subsequence, we also have $y_i\to y$, with $d(y,\mathcal L^k)\geq\tau$.

\par

The same operation can be done using Item 2' and following the points $y_i$. Therefore there is an isometry $G: B_{\chi^{-1}}(y^*) \subset C(Z_y) \to B_{\chi^{-1}}(y)$, where $y^*$ denotes the vertex of $C(Z_y)$, that pulls back $u$ to a 1-homogeneous model. Now by \Cref{lemma:addingsym} we have that $B_{1}(x)$ is $(k+1)$-symmetric ($\chi \leq 1/2$). By the convergence of the spaces and the functions, $B_1(x_i)$ is $(k+1,\varepsilon)$-symmetric for all sufficiently large $i$, a contradiction.
\end{proof}

As usual, we will use the following iterative consequence of the cone-splitting theorem. See \cite[Corollary~4.2]{CN13} for a proof.

\begin{corollary}\label{cor:conesplitting}
Let $(X,d,\mathcal{H}^n)$ be an $\mathsf{ncRCD}(K,n)$ space, $\Omega\subset X$ a bounded domain, and $u:\Omega\to [0,\infty)$ an $L$-Lipschitz function.
Assume that $\mathcal{H}^n(B_1(p))\geq \mathrm{v}$ for any $p\in \Omega$. For every $\tau,\psi>0$ and
$0<\chi\leq 2^{-n}$, there exists
$\varepsilon=\varepsilon(n,\mathrm{v},K,L,\tau,\psi,\chi)>0$ such that the following
holds: Suppose that $0<r\leq1$ and that $B_{\chi^{-1}r}(x)$ is $(0,\varepsilon)$-symmetric. Then there exists an integer
$0\leq\ell\leq n$ such that $B_r(x)$ is $(\ell,\psi)$-symmetric with a Borel $\psi r$-pGH approximation $F:B_r(0,z_*)\subset\mathbb R^\ell\times C(Z)\to B_r(x)$ associated with a symmetry model, such that
\[
\left\{
y\in B_r(x):
B_{\chi^{-1}r}(y)\text{ is }(0,\varepsilon)\text{-symmetric}
\right\}
\subset
B_{\tau r}\left(
F\bigl(\mathbb R^\ell\times\{z_*\}\bigr)
\right).
\]
\end{corollary}

The integer $\ell$ above can be viewed as the largest number of possible quantitative symmetries. Thus if $x\in\mathcal S^k_{\eta,r}$ for some $\eta> \psi$, then necessarily $\ell\leq k$.

\subsection{Proof of Main Theorem}\label{sec:pf_of_main}

Here we show that the covering argument from \cite{CN13} adapts to our case with unbounded monotonicity.
For the reader's convenience, we restate the main theorem:

\begin{theorem}\label{thm:mainrestated}
Let $X$ be an $n$-dimensional noncollapsed Ricci limit space,
    and let $u$ be a limit solution with $C^\alpha$ coefficient $Q$ on a bounded domain $\Omega\subset X$.
    Without loss of generality, we fix $B_4(x) \Subset \Omega$ and assume $\mathcal{H}^n(B_4(x))\geq \mathrm{v}>0$.
    Let $L,\Lambda$ denote, respectively, the Lipschitz and coefficient bounds for the approximating sequence of $u$ as in \Cref{subsec:lim_sol} on $B_4(x)$.
    Then for any $\eta\in(0,1)$ and $\varepsilon \in (0,1)$, there exists $C=C(n,\mathrm{v},L,\Lambda,\alpha,\eta,\varepsilon)>0$ such that for any $r\in(0,1)$, we have
    \begin{equation}
      \Vol(T_r(\mathcal{S}^k_{\varepsilon}) \cap B_1(x)) \leq \Vol(T_r(\mathcal{S}^k_{\varepsilon, r})\cap B_1(x))\leq Cr^{n-k-\eta}, \label{eq:vol1restated}
    \end{equation}
    where $T_r(\cdot)$ denotes the tubular neighborhood, $\mathcal{S}^k_{\varepsilon,r}$ is the $(k,\varepsilon,r)$-singular stratum.
    We also have 
\begin{equation}
\mathcal{H}^{n-1} \llcorner \partial \{u > 0\}\left(T_r(\mathcal{S}^k_{\varepsilon, r})\cap B_1(x)\right)\leq C(n,\mathrm{v},L,\Lambda,\alpha,\eta,\varepsilon)r^{n-k-1-\eta} \label{eq:vol2restated},
\end{equation}
and the Minkowski dimension bound
    \begin{equation}
    \dim_{\mathcal M}(\mathcal{S}^{k}_\varepsilon\cap B_1(x))\leq k. \label{eq:vol3restated}
    \end{equation}
\end{theorem}

Morally speaking, our adaptation from \cite{CN13} works due to the following facts. 
First, in the Decomposition \Cref{lem:decomp}, we use that after fixing $i_0\in \mathbb{N}$ and $\gamma<1$, the number of bad scales above $\gamma^{i_0}$ is bounded by $C\sqrt{i_0}$, locally uniformly for each point.
This leads to the specific growth rate of $N_j$ in \Cref{lem:decomp}.
Second, we have the numerical fact that $j^{\sqrt{j}}$ has sub-exponential growth, i.e.\@ slower growth than $b^j$ for any $b>1$.
Hence the slow growth of $N_j$ suffices for the volume estimate.
\par
We start with some preparation for the proof.
As we mention in the introduction, for our limit solutions it suffices to consider minimizers on smooth manifolds; this will be rigorously justified at the end of the proof.
Hence we may assume that $u$ is a minimizer on a bounded domain $\Omega$ in a smooth manifold $M$, as in \Cref{subsec:lim_sol}.
We fix $B_4(\bar{x})\Subset \Omega$, and prove the volume estimates on $B_1(\bar{x})$.
\par
Our choice of parameters $\mathrm{v}, L, \Lambda, \alpha$: By the volume convergence and comparison, we may assume that $\vol (B_1(x))\geq C(n)\mathrm{v}$ for any $x\in B_1(\bar{x})$.
Additionally, by assumption, we have the Lipschitz bound $L$ and the $C^\alpha$-coefficient bound $\Lambda$ on $B_4(\bar{x})$.
\par
The following decomposition lemma is an adaptation of \cite[Lemma 3.2]{CN13}; compare our $N_j$ with their $Q$, which has no growth in $j$.

\begin{lemma}[Decomposition Lemma]\label{lem:decomp}
There are some constants $c_0(n),c_1(n)$, such that, for every $\eta\in (0,1)$ and $j\in\mathbb{N}$, we have
    \begin{enumerate}
        \item $\mathcal{S}^k_{\eta, \gamma^j}\cap B_1(\bar{x})$ is contained in the union of at most $j^{N_j}+1$ nonempty open sets $\mathcal{C}^k_{\eta, \gamma^j}$.
        \item Each set $\mathcal{C}^k_{\eta, \gamma^j}$ is the union of at most $(c_1\gamma^{-n})^{N_j}(c_0\gamma^{-k})^{j-N_j}$ balls of radius $\gamma^j$,
    \end{enumerate}
    and we have a bound $N_j \leq C(n,\mathrm{v},L,{\Lambda,\alpha},\eta)(1+j^{1/2})$.
    Here we set $\gamma(n,\eta) := c_0^{-3/\eta}$.
\end{lemma}
\begin{proof}[Proof of Decomposition Lemma] 
Choosing $c_0(n)$ sufficiently large, we may assume that $\gamma<\eta$ and $\gamma\leq2^{-n}$. We also set $\chi = \gamma$ in \Cref{cor:conesplitting}, and define $\varepsilon=\varepsilon(n,\mathrm{v}, L,\gamma)$ as given by that lemma, with $\tau := \gamma/10$, and $\psi:=\gamma/10$. 
\par
To each point $x\in \pt\{u>0\}\cap B_1(\bar x)$, we associate a $j$-tuple $T^j(x)\in\{0,1\}^j$ by setting $T_a^j(x)=1$ if $B_{\gamma^{a-1}}(x)$ is not $(0,\varepsilon)$\text{-symmetric}, and $T_a^j(x)=0$ otherwise. For every integer $1\le a\le j$ and every $a$-tuple $\overline{T}^a$, we set $E(\overline T^a):=\{x\in\partial\{u>0\}\cap B_1(\bar x):T^a(x)=\overline T^a\}$, where $T^a(x)$ denotes the $a$-tuple formed by the first $a$ entries
of $T^j(x)$.

\noindent \textit{1. Construction.}

\noindent Fix $j\geq 1$. We construct inductively, for $0\leq a\leq j$,
families $\mathcal C^k_{\eta,\gamma^a}$ as follows. For $a=0$,
there is a single set $\mathcal C^k_{\eta,\gamma^0}$, consisting of
the ball $B_1(\bar x)$. For $1\leq a\leq j$, the sets
$\mathcal C^k_{\eta,\gamma^a}$ are indexed by the possible
$a$-tuples $\overline T^a$, and
$\mathcal C^k_{\eta,\gamma^a}(\overline T^a)$ covers $E(\overline T^a)\cap \mathcal S^k_{\eta,\gamma^a}$. Assume that the family at level $a-1$ has already been constructed.
Given an $a$-tuple $\overline T^a$, let $\overline T^{a-1}$ be the
$(a-1)$-tuple obtained by dropping its last entry. For each ball
$B_{\gamma^{a-1}}(y)$ forming
$\mathcal C^k_{\eta,\gamma^{a-1}}(\overline T^{a-1})$, we re-cover
\[
B_{\gamma^{a-1}}(y)
\cap \mathcal S^k_{\eta,\gamma^a}
\cap E(\overline T^a)
\]
by a minimal collection of balls of radius $\gamma^a$, centered at
points of $B_{\gamma^{a-1}}(y) \cap \mathcal S^k_{\eta,\gamma^a}\cap E(\bar T^{a})$. The union of all these balls defines
$\mathcal C^k_{\eta,\gamma^a}(\overline T^a)$. This is possible since
$\mathcal S^k_{\eta,\gamma^a}
\subset
\mathcal S^k_{\eta,\gamma^{a-1}}$ and $
E(\overline T^a)\subset E(\overline T^{a-1})$.
If the set to be re-covered is empty, we simply take the corresponding
collection to be empty. \vspace{0.2 cm}

\noindent 2. \textit{Bounding the number of sets $\mathcal{C}^k_{\eta,\gamma^j}$.}

\noindent The family of sets $\mathcal{C}^k_{\eta,\gamma^j}$ corresponds to the set of $j$-tuples such that $E(T^j)$ is not empty.
To bound the number of such $j$-tuples, consider an arbitrary point $x\in\partial\{u>0\}\cap B_1(\bar x)$ and put $\delta:=\delta(n,\mathrm{v},L,\Lambda,\alpha,\varepsilon,\gamma)$ given by \Cref{lem:alm_rigid} with $\overline{r}=1$.
We choose $m:=m(\delta,\gamma)$ such that $\gamma^{-m} \geq \delta^{-1}$.
By \Cref{lem:alm_rigid}, it follows that if
\begin{enumerate}[nosep]
    \item[(i)]$\gamma^{i-2}< \delta$,
    \item[(ii)]$\log \mathcal{V}_{-1}(\gamma^{i-1}) - \log \mathcal{V}_{-1}(\gamma^{i-2-m}) < \delta$,
    \item[(iii)]$W_{t_0,\gamma^{2i-4}}-W_{t_0,\gamma^{2i-2}} < \delta$,
\end{enumerate}
then $B_{\gamma^{i-1}}(x)$ is $\varepsilon$-symmetric, and consequently $T^j_i(x)=0$. Now, note that
\begin{enumerate}[nosep]
    \item[(a)] The number of $i\geq 2$ for which Item (i) doesn't hold is bounded by $2 + \log \delta/\log \gamma$.
    \item[(b)] The number of $i\geq 2$ for which Item (ii) doesn't hold is bounded by
    $$-(m+2)\delta^{-1}\log(c\mathrm{v}).$$ 
    This holds since $\log  \mathcal{V}_{-1}(r)$ is decreasing,  $\log  \mathcal{V}_{-1}(r) \to 0 $ as $r\to 0$, and $\log  \mathcal{V}_{-1}(\gamma^{-m})$ is bounded below by $\log (c\mathrm{v})$ for some $c=c(n,\gamma,m)\in(0,1)$. 
    \item[(c)] We argue as in (b) but with more care. 
    We write
$$(W_{t_0,1} - W_{t_0,\gamma^2}) + (W_{t_0,\gamma^2} - W_{t_0,\gamma^4}) + \ldots (W_{t_0,\gamma^{2j-2}}- W_{t_0,\gamma^{2j}}) = W_{t_0,1} - W_{t_0,\gamma^{2j}}.$$ By \eqref{eq:monder_res}, each term is nonnegative.
Moreover, by taking
\[t_0=\gamma^{2j},
\]
we have by \eqref{eq:Wbound_restated} that the sum above is bounded by 
$$2\,C(n,\mathrm{v},\Lambda,\alpha, L)\abs{\log t_0}^{1/2} = C(n,\mathrm{v},\Lambda,\alpha, L,\gamma)j^{1/2}.$$
Thus, the number of $2\leq i\leq j+1$ for which Item (iii) doesn't hold is bounded by
$$
C(n,\mathrm{v},\Lambda,\alpha, L,\gamma)j^{1/2}\delta^{-1}.
$$
\end{enumerate}
Putting all together, with $\gamma=\gamma(n,\eta)$, it follows that \begin{equation}
\text{Bad Scales}_j\ := \#\{1\leq i\leq j:T_i^j(x)=1\} \leq C(n,\mathrm{v},\Lambda,\alpha, L,\eta)(1+j^{1/2}) =: N'_j, \label{eq:BoundBadScales}
\end{equation}
which may be rounded up to an integer. Now using that $\sum_{\ell=0}^{\min\{N_j',j\}}\binom{j}{\ell}
\le j^{N_j'}+1$, we conclude the first item using $N_j:= \min\{N_j' +1,j\}$. \vspace{0.2 cm}

\noindent\textit{3. Bounding the number of balls whose union is $\mathcal{C}^{k}_{\eta,\gamma^j}$}.

Fix a $j$-tuple $\overline T^j$ for which $\mathcal C^k_{\eta,\gamma^j}(\overline T^j)$ is nonempty, and denote its prefixes by $\overline T^a$. Fix $2\leq a\leq j$ and one of the parent balls $B_{\gamma^{a-1}}(x)$ appearing in $\mathcal C^k_{\eta,\gamma^{a-1}}(\overline T^{a-1})$. By construction, its center satisfies $x\in\mathcal S^k_{\eta,\gamma^{a-1}}\cap E(\overline T^{a-1})$. Assume that $\overline T^j_{a-1}=0$. Then, by the definition of the
tuple, $B_{\gamma^{a-2}}(x)$ is $(0,\varepsilon)$-symmetric. Moreover,
every point
$y\in A:=\mathcal S^k_{\eta,\gamma^a}\cap
B_{\gamma^{a-1}}(x)\cap E(\overline T^a)$ also satisfies
$T^j_{a-1}(y)=0$, and hence $B_{\gamma^{a-2}}(y)$ is
$(0,\varepsilon)$-symmetric. We may therefore apply \Cref{cor:conesplitting} with
$r=\gamma^{a-1}$, $\chi=\gamma$, and $\tau=\psi=\gamma/10$.
We obtain an integer $\ell\le n$, a cone
$\mathbb R^\ell\times C(Z)$, and a
$\frac{1}{10}\gamma^a$-GH approximation
$F:B_{\gamma^{a-1}}((0,z_*))\to B_{\gamma^{a-1}}(x)$ such that
$A$ is contained in the $\frac{1}{10}\gamma^a$-neighborhood of
$F((\mathbb R^\ell\times\{z_*\})\cap
B_{\gamma^{a-1}}((0,z_*)))$. Since
$x\in\mathcal S^k_{\eta,\gamma^{a-1}}$ and
$\psi=\gamma/10<\eta$, necessarily $\ell\le k$.

The Euclidean factor
$(\mathbb R^\ell\times\{z_*\})\cap
B_{\gamma^{a-1}}((0,z_*))$ can be covered by at most
$c_0(n)\gamma^{-\ell}\le c_0(n)\gamma^{-k}$ balls of radius
$\frac{1}{10}\gamma^a$. Taking into account the GH error and the
$\frac{1}{10}\gamma^a$-neighborhood above, their images cover $A$ by
balls of radius $\frac{3}{10}\gamma^a$. Re-centering these balls at
points of $A$ and enlarging their radius to $\gamma^a$, we conclude
that the re-covering at level $a$ multiplies the number of balls by at
most $c_0(n)\gamma^{-k}$ whenever $\overline T^j_{a-1}=0$. If $\overline T^j_{a-1}=1$, we instead use the general bound
$c_1(n)\gamma^{-n}$. By \eqref{eq:BoundBadScales}, this can occur at
most $N'_j$ times, in addition to the initial step $a=1$, for which
$T^j_{a-1}$ is not defined. Thus, using $N_j:= \min\{N'_j+1,j\}$, there are at most $N_j$ steps at which the factor $c_1(n)\gamma^{-n}$ appears,
while at all remaining steps the factor is at most
$c_0(n)\gamma^{-k}$. Putting $c_1 > c_0$ this proves item~2.
\end{proof}
\begin{proof}[Proof of the Volume Estimate, \Cref{thm:mainrestated}] 

We first prove $\Vol(T_r(\mathcal{S}^k_{\eta, r}\cap B_1(\bar x)))\leq Cr^{n-k-\eta}$ for $\eta\in (0,1)$ on smooth manifolds.
Given such $\eta$, we fix $\gamma = c_0(n)^{-3/\eta}$ as in \Cref{lem:decomp}, and notice that it suffices to prove this bound for $r = \gamma^j$ for every $j$. From \Cref{lem:decomp}, we have
\begin{align}
    \Vol(T_{\gamma^j}(\mathcal{S}^k_{\eta, \gamma^j}\cap B_1(\bar{x})))
    &\leq (j^{{N_j}}+1)\cdot (c_1\gamma^{-n})^{N_j}(c_0\gamma^{-k})^{j-{N_j}}\cdot c_2(n)\cdot(\gamma^{j})^n \\ 
    & = c_2(n)\cdot(j^{N_j}+1)\cdot [(c_1\gamma^{-n})^{N_j}c_0^{-{N_j}}(\gamma^{-k})^{-{N_j}}]\cdot (\gamma^{j})^{n-k-\eta/3} \\ & \leq c_2(n)\cdot(j^{N_j}+1)\cdot c_3(n,\eta)^{N_j}\cdot (\gamma^{j})^{n-k-\eta/3}
\end{align}
with ${N_j}\leq C(n,\mathrm{v},L,\Lambda,\alpha,\eta)(1+j^{1/2})$.
We then have the numerical facts
$$
j^{N_j}+1 \leq C(n,\mathrm{v},L,\Lambda,\alpha,\eta)(\gamma^j)^{-\eta/3},
$$
and also $c_3(n,\eta)^{N_j} \leq C(n,\mathrm{v},L,\Lambda,\alpha,\eta)(\gamma^j)^{-\eta/3}$. 
Plugging these back into the inequality above, we get
\begin{align}
    \Vol(T_{\gamma^j}(\mathcal{S}^k_{\eta, \gamma^j}\cap B_1(\bar{x})))
    &\leq C(n,\mathrm{v},L,\Lambda,\alpha,\eta)(\gamma^j)^{n-k-\eta} \label{eq:asasw}.
\end{align}
To obtain \eqref{eq:vol1restated} on $B_1(x)$ in the statement, choose a minimal finite covering of $\partial\{u>0\}\cap B_2(x)$ by balls $B_{1/4}(x_a)$ centered in that set.
By volume comparison, the number of these balls is uniformly bounded.
Since $B_1(x_a)\subset B_4(x)$, we may apply the preceding estimate after rescaling to each of these balls. Since $T_r(\mathcal S^k_{\varepsilon,r})\cap B_1(x) \subset T_r(\mathcal S^k_{\varepsilon,r}\cap B_2(x)),$
summing over the covering gives \eqref{eq:vol1restated} for $r<1/4$. 
The cases for the remaining radii follow from volume comparison after enlarging the constant.
\par
In the present smooth setting, \eqref{eq:vol2restated} follows readily from
\eqref{eq:vol1restated} and the upper Ahlfors regularity. Indeed, fix $0<r<1/4$ and choose a maximal pairwise disjoint family of $N$ balls
$B_{r/2}(x_i)$ with centers in $\mathcal S^k_{\varepsilon,r}\cap B_2(x)$. By
maximality, the portion of the tubular neighborhood under consideration
is contained in the union of the balls $B_{2r}(x_i)$. Therefore, using \eqref{eq:AhlforsRegularity},
\begin{equation}\label{eq:i}
\begin{aligned}
&\mathcal H^{n-1}\llcorner\partial\{u>0\}
\left(T_r(\mathcal S^k_{\varepsilon,r})\cap B_1(x)\right) \\
&\qquad\leq
\sum_{i=1}^{N}
\mathcal H^{n-1}
\bigl(\partial\{u>0\}\cap B_{2r}(x_i)\bigr)
\leq
C(n,L,\Lambda)Nr^{n-1}.
\end{aligned}
\end{equation}
By \eqref{eq:asasw} and the noncollapsed assumption, we have
$N\leq C(n,\mathrm{v},\eta,L,\Lambda,\alpha,\varepsilon)r^{-k-\eta}$, and hence
\begin{equation}
\mathcal H^{n-1}\llcorner\partial\{u>0\}
\left(T_r(\mathcal S^k_{\varepsilon,r})\cap B_1(x)\right)
\leq
C(n,\mathrm{v},L,\Lambda,\alpha,\eta,\varepsilon)
r^{n-k-1-\eta}.
\end{equation}
The cases for the remaining radii $1/4\leq r<1$ follow from the local upper Ahlfors bound after enlarging the constant.
\par
We end the proof by explaining in detail why it is enough to establish the estimates above in the smooth setting. Let $(M_i,g_i,u_i)$ be an approximating sequence converging to the limit solution $(X,d,u)$, and let $f_i:B_1^{M_i}(x_i)\longrightarrow B_1^X(x)$ be $\varepsilon_i$-GH approximations realizing this convergence, with $\varepsilon_i\to0$.
Fix $r\in(0,1)$ and $\tau>0$ sufficiently small such that $s:=r+\tau<1$.
Let
\[
A:=\mathcal S^k_{\eta,r}(u)\cap B_1(x),
\qquad
A_i:=\mathcal S^k_{\eta/2,s}(u_i)\cap B_1(x_i).
\]
Applying \Cref{prop:basicpropstrata}-(iv) to a finite
$\tau/2$-net of $A$, together with \Cref{prop:basicpropstrata}-(i), for all sufficiently large $i$ we have $A\subset T_\tau^X(f_i(A_i))$, and hence $T_r^X(A)\subset T_s^X(f_i(A_i))$.
\par
Now, choose a maximal pairwise disjoint family of balls
\[
\{B_\rho^{M_i}(z_{i,j})\}_{j=1}^{N_i},
\qquad
\rho:=s/10,\quad z_{i,j}\in A_i.
\]
By \eqref{eq:asasw} and the preceding covering
argument in the smooth case at scale $s$, together
with the uniform lower volume bounds supplied by
noncollapsing, we have
\[
N_i\leq C(n,\mathrm v,L,\Lambda,\alpha,\eta)s^{-k-\eta/2}.
\]
On the other hand, the maximality gives
\[
A_i\subset\bigcup_{j=1}^{N_i}B_{2\rho}^{M_i}(z_{i,j}).
\]
Consequently, using the distortion bound for $f_i$, we have
\[
T_s^X(f_i(A_i))
\subset
\bigcup_{j=1}^{N_i}
B_{s+2\rho+\varepsilon_i}^X(f_i(z_{i,j})).
\]
For $i$ sufficiently large, $s+2\rho+\varepsilon_i<2s$.
Thus, by the uniform upper volume bound for balls in $X$,
\[
\begin{aligned}
\Vol\left(
T^X_r(A)
\right)\leq \Vol \left(T_s^X(f_i(A_i))\right)
\leq C(n)N_i s^n
\leq
C(n,\mathrm v,L,\Lambda,\alpha,\eta)
(r+\tau)^{n-k-\eta/2}.
\end{aligned}
\]
Letting $\tau\to0$ with $r\leq 1$ yields the estimate \eqref{eq:asasw} on the Ricci limit spaces $X$.
 The same covering argument as before then gives \eqref{eq:vol1restated} on $X$. The boundary measure estimate \eqref{eq:vol2restated} on $X$ then follows from the same argument above, using the upper Ahlfors regularity.
\end{proof}

\begingroup
\small
\setlength{\parskip}{0pt}
\bibliographystyle{alpha}
\bibliography{ref.bib}
\endgroup

\end{document}